\documentclass[reqno]{amsart}
\usepackage[english]{babel}
\usepackage[utf8]{inputenc}
\usepackage[T1]{fontenc}
\usepackage{amsmath,amssymb,amsthm,mathtools,enumitem, dynkin-diagrams}
\usetikzlibrary{trees}
\usetikzlibrary{matrix}
\usepackage{mathrsfs}
\usepackage{lmodern,subcaption,bbm}
\usepackage[all,cmtip]{xy}
\usepackage[linktocpage=true, pdfusetitle,pagebackref]{hyperref}
\renewcommand*{\backref}[1]{}
\usepackage{array}
\usepackage{csquotes}
\usepackage{multirow}
\usepackage{todonotes}

 \usepackage{url,tikz-cd}
  \usetikzlibrary{arrows.meta}
  \usetikzlibrary{calc}
  \usetikzlibrary{matrix}

\calclayout
\newcolumntype{C}{>{$}c<{$}}

\usepackage{xcolor}
\hypersetup{
    colorlinks,
    linkcolor={red!80!black},
    citecolor={green!50!black},
    urlcolor={blue!80!black}
}
\makeatletter
\def\th@plain{
  \thm@notefont{}
  \itshape
}
\def\th@definition{
  \thm@notefont{}
  \normalfont
}
\makeatother
\newtheorem*{thmmain}{Main Theorem}
\newtheorem{theorem}{Theorem}[section]

\newtheorem{corollary}[theorem]{Corollary}

\newtheorem{proposition}[theorem]{Proposition}
\newtheorem{lemma}[theorem]{Lemma}

\theoremstyle{definition}
\newtheorem{definition}[theorem]{Definition}

\newtheorem{remark}[theorem]{Remark}

\DeclarePairedDelimiter{\abs}\lvert\rvert

\newcommand\restr[2]{{
  \left.\kern-\nulldelimiterspace
  #1
  \vphantom{\big|}
  \right|_{#2}
  }}

\makeatletter

\newcommand*\bigcdot{\mathpalette\bigcdot@{.5}}
\newcommand*\bigcdott{\mathpalette\bigcdot@{.7}}
\newcommand*\bigcdot@[2]{\mathbin{\vcenter{\hbox{\scalebox{#2}{$\m@th#1\bullet$}}}}}
\g@addto@macro\bfseries{\boldmath}
\makeatother

\newcommand{\eop}[1]{\ensuremath{\vec{e}^{\,\mathrm{op}}}}

\newcommand{\mc}[1]{\ensuremath{\mathcal{#1}}}
\newcommand{\mf}[1]{\ensuremath{\mathfrak{#1}}}

\newcommand{\R}{\ensuremath{\mathbb{R}}}
\newcommand{\N}{\ensuremath{\mathbb{N}}}
\newcommand{\Z}{\ensuremath{\mathbb{Z}}}
\newcommand{\C}{\ensuremath{\mathbb{C}}}

\newcommand{\intd}{\ensuremath{\,\mathrm{d}}}

\date{\today}

\begin{document}
\pagenumbering{arabic}
\title{A pairing formula for resonant states on finite regular graphs}
\author[Arends]{Christian Arends}
\address{Department of Mathematics, Aarhus University, Ny Munkegade 118,
        8000 Aarhus C, Denmark}
        \email{arends@math.au.dk}
        \author[Frahm]{Jan Frahm}
\address{Department of Mathematics, Aarhus University, Ny Munkegade 118,
        8000 Aarhus C, Denmark}
        \email{frahm@math.au.dk}
\author[Hilgert]{Joachim Hilgert}
\address{Institut f\"ur Mathematik, Universit\"at Paderborn, Warburger Str. 100,
        33098 Paderborn, Germany}
        \email{hilgert@math.upb.de}

\begin{abstract}
On a finite regular graph, (co)resonant states are eigendistributions of the transfer operator associated to the shift on one-sided infinite non-backtracking paths. We introduce two pairings of resonant and coresonant states, the \emph{vertex pairing} which involves only the dependence on the initial/terminal vertex of the path, and the \emph{geodesic pairing} which is given by integrating over all geodesics the evaluation of the coresonant state on the first half of the geodesic times the resonant state on the second half. The main result is that these two pairings coincide up to a constant which depends on the resonance, i.e.\@ the corresponding eigenvalue of the transfer operator.
\end{abstract}

\maketitle

\section*{Introduction}

Regular graphs can be viewed as either toy models or as non-Archimedean analogs of rank one locally symmetric spaces. In the first case one views them as metric spaces with universal coverings which are homogeneous trees and as such share a lot of properties with Riemannian symmetric spaces. In the second case one considers homogeneous trees as generalizations of Bruhat--Tits trees associated with $p$-adic groups. Typically, analytic questions become easier when transferred from locally symmetric spaces to graphs. In this way, regular graphs can serve as testing grounds for the analysis on locally symmetric spaces.

We are particularly interested in the spectral geometry of locally symmetric spaces and their graph analogs. This includes the spectral theory of various kinds of Laplacians as well as operators related to the classical dynamics such as geodesic flows. Relations between these two kinds of spectral theories have been studied ever since the Poisson summation formula was established. In recent years a number of interesting new connections were found in the context of resonances for geodesic flows on locally symmetric spaces \cite{DFG,GHWa,GHWb,HWW21,AH21} as well as regular graphs, \cite{LP16,An17,BHW21,BHW23,AFH23}.

The spectral theory of geodesic flows is more subtle than the spectral theory of Laplacians as the generators of geodesic flows are not elliptic. This is where resonances enter the picture. For locally symmetric spaces they are derived from the hyperbolic structure of the geodesic flow on the sphere bundle, for graphs such an interpretation is expected (see \cite[Rem.~8.5]{BHW23}), but not yet available. In the latter case resonances are defined in terms of a one-sided shift dynamics which is expected to reflect the expansive part of the hypothetical hyperbolic dynamics. The shift dynamics lives on the spaces of one-sided infinite non-backtracking paths in the graphs, the shift being simply taking one step in the path. From this shift dynamics one derives a transfer operator acting on functions of such paths by shifting the argument and summing over values of preimages. The resonances are then the eigenvalues of the dual operators of these transfer operators, whereas the resonant states are the corresponding eigenvectors. Reversing the orientation of paths one arrives at the concept of coresonant states.

In this paper we study pairings between resonant states and coresonant states in the case of finite regular graphs. As in the case of (compact) rank one locally symmetric spaces there are more than one natural ways to define such pairings and one wants to understand how these different pairings are related. In fact, in the case of locally symmetric spaces it was possible to establish a formula relating two natural pairings which turned out to be a crucial technical tool in establishing an equidistribution result for invariant Ruelle densities,  \cite{GHWb}. In the case of finite graphs the two pairings come about by restricting the tensor product of a resonant and a coresonant state to different naturally chosen subsets of the set of pairs of paths. In the first pairing which we call the \emph{vertex pairing} as it can be expressed as a sum over all vertices, one considers pairs of paths with the same initial vertex (see Definition~\ref{rem:finite graphs}). In the second pairing, the \emph{geodesic pairing}, one considers pairs of paths that can be joined at their ends to give a  non-backtracking bi-infinite path (which, when lifted to the covering tree is a bi-infinite geodesic), see Definition~\ref{rem:pairing1}. Our main result is that these two pairings are constant multiples of each other, the constant depending on the resonance and the degree of the regular graph.

Let us describe our results in more detail.

\subsection*{(Co)resonances and (co)resonant states on finite regular graphs}

For a finite $(q+1)$-regular graph $\mf G$ we denote by ${\mf P}^+$ the space of non-backtracking chains $(\vec e_1,\vec e_2,\ldots)$ of edges $\vec e_j$. It can be endowed with a natural topology (see Definition~\ref{def:ShiftSpace}). The transfer operator $\mathcal L_+$ acts on the space $C^\mathrm{lc}({\mf P}^+)$ of locally constant functions on ${\mf P}^+$ by
$$ \mathcal L_+F(\vec e_1,\vec e_2,\ldots) = \sum_{(\vec e_0,\vec e_1,\vec e_2,\ldots)\in{\mf P}^+}F(\vec e_0,\vec e_1,\vec e_2,\ldots) $$
and is associated to the shift dynamics given by $(\vec e_0,\vec e_1,\vec e_2,\ldots)\mapsto(\vec e_1,\vec e_2,\ldots)$. The dual operator $\mathcal L_+'$ acts on the algebraic dual space $\mc D'({\mf P}^+)$ of $C^\mathrm{lc}({\mf P}^+)$, which we view as a space of distributions on ${\mf P}^+$. Eigenvalues of $\mathcal L_+'$ on $\mc D'({\mf P}^+)$ are called \emph{resonances} and the corresponding eigendistributions are called \emph{resonant states}. Similarly, one can define \emph{coresonant states} as eigendistributions of a transfer operator $\mathcal L_-'$ acting on distributions on the space ${\mf P}^-$ of non-backtracking chains $(\ldots,\vec e_2,\vec e_1)$ of edges $\vec e_j$ in a similar way as $\mathcal L_+$.

\subsection*{The vertex pairing}

Consider the projections $\pi_\pm:{\mf P}^\pm\to\mf X$ onto the set $\mf X$ of vertices, mapping a chain of edges to the initial/terminal vertex. The associated pushforwards $\pi_{\pm,*}:\mc D'({\mf P}^\pm)\to\mathrm{Maps}(\mf X,\C)$ map a distribution on ${\mf P}^\pm$ to the map on $\mf X$ given by averaging over all chains with a fixed initial/terminal vertex. Note that (co)resonance states associated to resonances $z\not\in\{-1,0,1\}$ are uniquely determined by their pushforwards to $\mf X$ (see \cite[Theorem 11.5]{BHW23}). This gives rise to the \emph{vertex pairing} of  two distributions $u_\pm\in \mc D'({\mf P}^\pm)$:
$$ \langle u_+,u_-\rangle_{(\mf X)} = \sum_{x\in\mf X}\pi_{+,*}u_+(x)\cdot\pi_{-,*}u_-(x). $$
We remark that, viewing ${\mf P}^\pm$ as the analog of the sphere bundle of a Riemannian manifold, the function $\pi_{\pm,*}u$ plays the role of the zeroth Fourier mode of $u$.

\subsection*{The geodesic pairing}

On the other hand, we can integrate the tensor product distribution $u_+\otimes u_-\in \mc D'({\mf P}^+)\otimes \mc D'({\mf P}^-) \cong C^\mathrm{lc}({\mf P}^-\times{\mf P}^+)'$ over the subset ${\mf P}\subseteq{\mf P}^-\times{\mf P}^+$ of pairs of chains that share the same final/terminal vertex and can be joined to form a geodesic in $\mf G$. This defines the \emph{geodesic pairing}
$$ \langle u_+,u_-\rangle_\textup{geod} = \int_{{\mf P}} (u_+\otimes u_-)|_{{\mf P}} \coloneqq \langle u_+\otimes u_-,\mathbbm{1}_{{\mf P}}\rangle. $$
It should be noted that the restriction of distributions to the subset ${\mf P}\subseteq{\mf P}^-\times{\mf P}^+$ is possible in this context since ${\mf P}$ is open, which is a lot simpler than the wavefront set conditions one has to guarantee in the context of smooth manifolds.

\subsection*{The pairing formula}

While the vertex pairing and the geodesic pairing are quite different in nature and in general not related to each other, they coincide up to a constant when restricted to (co)resonant states:

\begin{thmmain}[see Theorem~\ref{thm:pairing formula}]
Let $0\neq z\in\C$ be a resonance of the finite $(q+1)$-regular graph $\mf G$ and $u_\pm\in \mc D'({\mf P}^\pm)$ associated (co)resonant states. Then the following pairing formula holds:
$$ (z^2-q)\langle u_+,u_-\rangle_{(\mf X)} = (z^2-1)\langle u_+,u_-\rangle_\textup{geod}  $$
\end{thmmain}

The proof is in fact modeled on the proof of the pairing formula given in \cite{GHWb}. This means in particular that we need to use the universal covering of the graph which is a tree, a number of integral formulas for the corresponding group of tree automorphisms, and a subtle limiting process for the integrals involved. It would, of course, be desirable to have a direct combinatorial proof of the pairing formula.

\subsection*{Outlook}

A graph analog for the equidistribution results built on the pairing formula in \cite{GHWb} would have to be dealing with sequences of graphs in the spirit of \cite{AM15}. Given the extent to which the quantum-classical correspondence, on which the proof of the equidistribution result in \cite{GHWb} also depends, could be transferred to graphs (see \cite{BHW21,BHW23,AFH23}), we expect such an application to be possible. However, as of now one other crucial tool used in \cite{GHWb}, the Patterson-Sullivan distributions, have not yet been established for graphs.

\subsection*{Structure of the paper}

In Section~\ref{sec:DefiningThePairings} we review resonances and (co)resonant states and define several different pairings for them, among them the vertex and the geodesic pairing. Section~\ref{sec:trees} introduces some technical tools for the universal covering of a finite regular graph which is a homogeneous tree, in particular several subgroups of automorphisms of the tree. These tools are used in Section~\ref{sec:coverings} to lift resonances on the finite graph to its universal cover and express the pairings in terms of integrals associated with the homogeneous tree. Finally, we state and prove our main theorem, the pairing formula, in Section~\ref{sec:pairingformula}.

Some, but by no means all, of the needed integral formulas for the group of tree automorphisms we did find in the literature. For the convenience of the readers we collected the formulas together with their proofs in Appendix~\ref{sec:integralformulas} even though we assume that most of them are known to the experts.

\subsection*{Acknowledgements}

The first named author was partially supported by a research grant from the Aarhus University Research Foundation (Grant No. AUFF-E-2022-9-34). The second named author was partially supported by a research grant from the Villum Foundation (Grant No. 00025373). The third named author was supported by the Deutsche Forschungsgemeinschaft (DFG, German Research Foundation) via the grant SFB-TRR 358/1 2023 -- 491392403.

\subsection*{Notation}

For a set $X$ we write $\mathrm{Maps}(X,\C)$ for the vector space of maps $f:X\to\C$ with the pointwise operations. If $X$ carries a topology, then $\mathrm{Maps}_c(X,\C)$ denotes the subspace of maps $f$ with compact support, i.e.\@ the closure of the set $\{x\in X:f(x)\neq0\}$ is compact.

\section{Defining the pairings}\label{sec:DefiningThePairings}

In this section we introduce the pairings of resonant and coresonant states we want to compare. This requires some preparation.

\subsection{Resonances}

Let $\mf G\coloneqq(\mf X,\mf E)$ be a connected \emph{graph} consisting of a set $\mf X$ of \emph{vertices} and a set $\mf E\subseteq \mf X^2$ of \emph{directed edges}. We assume that $\mf G$ is $(q+1)$-regular, i.e.\@ every vertex has $q+1$ \emph{neighbors} (vertices connected to $x$ by an edge).

We will always assume that $\mf E$ is symmetric under the switch of vertices, that $\mf G$ has  \emph{no loops}, i.e.\@  $\mf E\cap \{(x,x)\mid x\in \mf X\}=\emptyset$, and that $\mathfrak{G}$ has no multiple edges. For each directed edge $\vec{e}\coloneqq(a,b)$ we call $a\eqqcolon\iota(\vec{e})$ the \emph{initial} and $b\eqqcolon\tau(\vec{e})$ the \emph{terminal point} of $\vec{e}$. Moreover, let $\eop{e}\coloneqq(b,a)$ denote the \emph{opposite edge of $\vec{e}$}.

\begin{definition}\label{def:ShiftSpace}
A \emph{chain of edges} in $\mf G$ is an infinite sequence $(\vec e_1,\vec e_2,\ldots)$ of directed edges  that are \emph{concatenated}, i.e.\@ $\tau(\vec e_j) = \iota(\vec e_{j+1})$ for all $j\in \N$, and \emph{non-backtracking}, i.e.\@ $\tau(\vec e_{j+1})\neq \iota (\vec e_j)$ for all $j\in \N$.
The \emph{shift space} ${\mf P}^+$ is the set of chains. On the shift space we can introduce the \emph{shift operator}
\begin{gather*}
 \sigma : {\mf P}^+\to {\mf P}^+,\quad (\vec e_1,\vec e_2,\vec e_3,\ldots)\mapsto (\vec e_2,\vec e_3,\ldots).
\end{gather*}
By reversing the orientation we obtain the \emph{opposite shift space} ${\mf P}^-$ which is given by the set of infinite sequences of concatenated edges $(\ldots,\vec{e}_2,\vec{e}_1)$. We equip ${\mf P}^\pm$ with the topology generated by the sets of all chains of edges $(\vec e_1,\vec e_2,\ldots)$ resp. $(\ldots,\vec e_2,\vec e_1)$ such that $(\vec e_1,\ldots,\vec e_n)$ is equal to some fixed tuple of edges and $n\in\N$ (also referred to as district topology in \cite[\S~5]{BHW23}). With respect to this topology, we consider
\begin{gather*}
   C_c^{\mathrm{lc}}({\mf P}^\pm)\coloneqq\{f\in \mathrm{Maps}({\mf P}^\pm,\mathbb{C})\mid f \text{ locally constant with compact support}\}.
\end{gather*}
\end{definition}

If the graph happens to be a tree, chains can be interpreted as geodesic rays, so the shift space ${\mf P}^+$ is one possible graph-analog of the sphere bundle of a Riemannian manifold. If we adopt this interpretation, the shift operator can be viewed as the \emph{geodesic flow} on the tree. Associated to this dynamics we have the following natural \emph{transfer operator}:

\begin{definition}[Transfer operators {cf.\@ \cite[\S~3]{BHW23}}]
\label{def:RuelleTO}
The \emph{(Ruelle) transfer operator} $\mathcal L\coloneqq\mathcal L_+: \mathrm{Maps}({\mf P}^+, \C)\to \mathrm{Maps}({\mf P}^+, \C)$
associated to the shift $\sigma$ is given by
\begin{equation}\label{eq:RuelleTO}
(\mathcal L F)(\vec{\mathbf e})\ \coloneqq\ \sum_{\sigma(\vec{\mathbf e}\,')=\vec{\mathbf e}}  F(\vec{\mathbf e}\,').
\end{equation}
We can rewrite the action of $\mathcal{L}$ as
\[
 (\mathcal LF)(\vec e_1,\vec e_2,\ldots) = \sum_{\vec e_0: \tau(\vec e_0)=\iota(\vec e_1)} F(\vec e_0, \vec e_1, \vec e_2,\ldots).
\]
Similar to the shift dynamics, reversing the orientation leads to the \emph{opposite transfer operator}
\begin{gather*}
\mathcal{L} _{-}\colon \mathrm{Maps}({\mf P}^-,\C)\to\mathrm{Maps}({\mf P}^-,\C),\quad (\mathcal{L} _{-}F)(\ldots,\vec{e}_2,\vec{e}_1)\coloneqq \sum_{\vec{e}_0\colon\tau(\vec{e}_1)=\iota(\vec{e}_0)} F(\ldots,\vec{e}_2,\vec{e}_1,\vec{e}_0).
\end{gather*}
\end{definition}

Note that $\mathcal{L}_\pm$ leaves $C_c^{\mathrm{lc}}({\mf P}^\pm)$ invariant so that we can consider its dual operator
\begin{gather*}
   \mathcal{L}_\pm'\colon \mc D'({\mf P}^\pm)\to \mc D'({\mf P}^\pm),
\end{gather*}
where, motivated by the Archimedean analog, we write $\mc D'({\mf P}^\pm)$ for the algebraic dual $C_c^{\mathrm{lc} }({\mf P}^\pm)'$ of $C_c^{\mathrm{lc}}({\mf P}^\pm)$.

\begin{definition}[Resonant and coresonant states]
   For each $z\in\mathbb{C}$ we define the eigenspaces
\begin{gather*}
 \mathcal E_z({\mathcal L}_\pm';\mc D'({\mf P}^\pm)) \coloneqq \{u \in \mc D'({\mf P}^\pm)\mid {\mathcal L}_{\pm}' u = zu\}.
\end{gather*}
We call non-zero elements of $\mathcal E_z({\mathcal L}_+';\mc D'({\mf P}^+))$ resp.\@ $\mathcal E_z({\mathcal L}_{-}';\mc D'({\mf P}^-))$ \emph{resonant} resp.\@ \emph{coresonant states}. Moreover, we say that $z\in\mathbb{C}$ is a \emph{(co)resonance} if there exists a non-zero (co)resonant state for $z$.
\end{definition}

\subsection{Base projections and the vertex pairing}\label{subsec:base_pr} In order to define pairings of resonant and coresonant states we introduce base projections relating functions on chains to functions on edges and vertices.

\begin{definition}[Base projections]
The \emph{base projections} $\pi_\pm^{\mf E}: {\mf P}^\pm\to \mathfrak{E}$  and $\pi_\pm:=\pi_\pm^{\mf X} : {\mf P}^\pm\to \mathfrak{X}$ are given by
$\pi_+^{\mf E}(\vec e_1,\ldots)= \vec e_1$ and $\pi_-^{\mf E}(\ldots,\vec e_1)= \vec e_1$, respectively $\pi_+(\vec e_1,\ldots)= \iota(\vec e_1)$ and $\pi_-(\ldots,\vec e_1)= \tau(\vec e_1)$, i.e.\@ $\pi_+=\iota\circ \pi_+^{\mf E}$ and $\pi_-=\tau\circ \pi_-^{\mf E}$.
\end{definition}

The maps $\pi_\pm$ and $\pi_\pm^{\mf E}$ induce pullback maps
\[\pi_\pm^*:  \mathrm{Maps}(\mathfrak X,\mathbb C)\to\mathrm{Maps}({\mf P}^\pm,\mathbb C) , \quad f\mapsto f\circ \pi_\pm\]
and
\[\pi_\pm^{\mf E,*}:  \mathrm{Maps}(\mathfrak E,\mathbb C)\to\mathrm{Maps}({\mf P}^\pm,\mathbb C) , \quad f\mapsto f\circ \pi^{\mf E}_\pm\]
satisfying
\[\pi_+^*=\pi_+^{\mf E,*}\circ \iota^*\quad\text{and}\quad
\pi_-^*=\pi_-^{\mf E,*}\circ \tau^*,
\]
where $\iota^*,\tau^*:\mathrm{Maps}(\mathfrak X,\mathbb C)\to\mathrm{Maps}(\mathfrak E,\mathbb C)$ denote the pullback maps induced by $\iota$ and $\tau$.

\begin{proposition}\label{prop:push-forwards_dist_proj}
By dualization, $\pi_\pm^*$ and $\pi_\pm^{\mf E,*}$ induce well-defined pushforwards
\begin{align*}
  \pi_{\pm,*}\coloneqq(\pi_\pm^*)'&\colon  \mc D'({\mf P}^\pm)\to  \mathrm{Maps}_c({\mathfrak X},\C)'\\
   \pi^{\mf E}_{\pm,*}\coloneqq(\pi_\pm^{\mf E,*})'&\colon  \mc D'({\mf P}^\pm)\to  \mathrm{Maps}_c({\mathfrak E},\C)'
\end{align*}
\end{proposition}

  \begin{proof}
  Note that the sets $\pi_\pm^{-1}(x)\subseteq {\mf P}^\pm$ and  $(\pi^{\mf E}_\pm)^{-1}(\vec{e})\subseteq {\mf P}^\pm$
  are open and compact for each $x\in \mf X$ and $\vec{e}\in \mf E$. Therefore, the images of $\pi_\pm^*$ and $\pi_\pm^{\mf E, *}$ are contained in $C^\mathrm{lc}({\mf P}^\pm)$.  Moreover, this shows that functions with compact supports get mapped to functions with compact supports. Thus we have
  $\pi_\pm^{*}: \mathrm{Maps}_c({\mathfrak X},\C)\to C^\mathrm{lc}_c({\mf P}^\pm)$ and
  $\pi_\pm^{\mf E,*}: \mathrm{Maps}_c({\mathfrak X},\C)\to C^\mathrm{lc}_c({\mf P}^\pm)$. Passing to the dual maps shows the claim.
  \end{proof}

\begin{remark}\label{rem:push_forward_restr}
	We identify $\mathrm{Maps}_c(\mf X,\C)'\simeq\mathrm{Maps}(\mf X,\C)$ by means of the non-degenerate bilinear pairing
	$$ \mathrm{Maps}(\mf X,\C)\times\mathrm{Maps}_c(\mf X,\C)\to\C, \quad (f,h)\mapsto\langle f,h\rangle_{\mf X}\coloneqq\sum_{x\in\mf X}f(x)h(x). $$
	Then $u\in\mathrm{Maps}_c(\mf X,\C)'$ is identified with the function $x\mapsto\langle u,\delta_x\rangle$, where $\delta_x$ is the indicator function of $\{x\}\subseteq\mf X$. We also use the analogous identification $\mathrm{Maps}_c(\mf E,\C)'\simeq\mathrm{Maps}(\mf E,\C)$ by means of the pairing
	$$ \mathrm{Maps}(\mf E,\C)\times\mathrm{Maps}_c(\mf E,\C)\to\C, \quad (f,h)\mapsto\langle f,h\rangle_{\mf E}\coloneqq\sum_{\vec e\in\mf E}f(\vec e)h(\vec e). $$
	Under these identifications, the maps in Proposition~\ref{prop:push-forwards_dist_proj} can be written as
	\begin{gather}
		\pi_{\pm,*}u(x)=\langle u,\mathbbm{1}_{\pi_+^{-1}(x)}\rangle \qquad (x\in\mf X,u\in \mc D'({\mf P}^\pm)).\label{eq:ExplicitFormulaPushForwardFromSigmaToX}
	\end{gather}
\end{remark}

For finite graphs, this allows us to define the vertex and edge pairing of distributions $u_\pm\in \mc D'({\mf P}^\pm)$ in terms of their push-forwards $\pi_{\pm,*}u_\pm\in\mathrm{Maps}(\mf X,\C)$ and $\pi_{\pm,*}^{\mf E}u_\pm\in\mathrm{Maps}(\mf E,\C)$.

\begin{definition}[Vertex and edge pairing]\label{rem:finite graphs}
Assume that $\mf G$ is finite. For distributions $u_\pm\in \mc D'({\mf P}^\pm)$ we define their \emph{vertex pairing} by
\begin{gather*}
\langle u_+,u_-\rangle_{(\mf X)}
\coloneqq \langle \pi_{+,*}u_+,\pi_{-,*}u_-\rangle_{\mf X} = \sum_{x\in \mf X}\big(\pi_{+,*}u_+\big)(x) \big(\pi_{-,*}u_-\big)(x)
\end{gather*}
and their \emph{edge pairing} by
\begin{gather*}
\langle u_+,u_-\rangle_{(\mf E)}
\coloneqq \langle \pi^{\mf E}_{+,*}u_+,\pi^{\mf E}_{-,*}u_-\rangle_{\mf E} =  \sum_{\vec{e}\in \mf E}\big(\pi^{\mf E}_{+,*}u_+\big)(\vec{e}) \big(\pi^{\mf E}_{-,*}u_-\big)(\vec{e}).
\end{gather*}
For later use, we also introduce the \emph{modified edge pairing}
\[\langle u_+,u_-\rangle_{(\mathrm{op},\mf E)}\coloneqq\langle \mathrm{op}_*\pi^{\mf E}_{+,*}u_+,\pi^{\mf E}_{-,*}u_-\rangle_{\mf E},
\]
where $\mathrm{op}_*\colon \mathrm{Maps}(\mf E,\mathbb C)'\to \mathrm{Maps}(\mf E,\mathbb C)'$ is dual to the pullback map $\mathrm{op}^*\colon \mathrm{Maps}(\mf E,\mathbb C)\to \mathrm{Maps}(\mf E,\mathbb C)$ with respect to the reversal of orientation $\mathrm{op}\colon \mf E\to \mf E$.
\end{definition}

\subsection{The geodesic pairing}\label{subsec:finite graphs}

In this subsection we define for a finite graph $\mathfrak{G}$ the geodesic pairing of distributions on ${\mf P}^\pm$. This construction uses the tensor product of distributions in $\mc D'({\mf P}^+)$ and $\mc D'({\mf P}^-)$. Since ${\mf P}^\pm$ is compact, the algebraic tensor product $C^\mathrm{lc}({\mf P}^+)\otimes C^\mathrm{lc}({\mf P}^-)$ can be canonically identified with $C^\mathrm{lc}({\mf P}^-\times{\mf P}^+)$. This gives rise to the tensor product $u_+\otimes u_-\in C^\mathrm{lc}({\mf P}^-\times{\mf P}^+)'$ of two distributions $u_\pm\in \mc D'({\mf P}^\pm)$. Note that for $\varphi\in  C^\mathrm{lc}({\mf P}^-\times {\mf P}^+)$ we have
\begin{equation}
	\langle u_+\otimes u_-,\varphi\rangle = \big\langle u_+,c_+ \mapsto \langle u_-,\varphi(\bigcdot,c_+)\rangle\big\rangle.\label{eq:AlternativeDefinitionTensorProduct}
\end{equation}

The tensor product $u_+\otimes u_-$ can be evaluated on indicator functions of open and closed subsets of ${\mf P}^-\times {\mf P}^+$. We interpret this operation as integrating $u_+\otimes u_-$ over this subset. Let us consider two particular cases of this construction.

First, let
\[{\mf P}^{(2)}\coloneqq\{(c_-,c_+)\in {\mf P}^-\times {\mf P}^+\mid \pi_-(c_-)=\pi_+(c_+)\}.\]
Note that the union of $c_-$ and $c_+$ need not be a chain as $c_+$ might backtrack $c_-$ up to a certain vertex or even indefinitely. By contrast, let ${\mf P}$ be the set of bi-infinite chains:
$$ {\mf P} \coloneqq \{(c_-,c_+)\in{\mf P}^{(2)}\mid \pi_-^\mf E(c_-)\neq\pi_+^\mf E(c_+)^\mathrm{op}\}. $$
Both ${\mf P}^{(2)}$ and ${\mf P}$ are open and closed in ${\mf P}^-\times{\mf P}^+$, so their indicator functions are contained in $C^{\mathrm{lc}}({\mf P}^-\times{\mf P}^+)$.

\begin{definition}[The geodesic pairing]\label{rem:pairing1}
We define the \emph{geodesic pairing} of $u_\pm\in \mc D'({\mf P}^\pm)$ by
$$ \langle u_+,u_-\rangle_\mathrm{geod} \coloneqq \langle u_+\otimes u_-,\mathbbm{1}_{{\mf P}}\rangle. $$
\end{definition}

The geodesic pairing behaves nicely with respect to the action of the transfer operators.

\begin{lemma}
	The transfer operators $\mathcal L_+'$ and $\mathcal L_-'$ are dual to each other with respect to the geodesic pairing, i.e.\@ for $u_\pm\in \mc D'({\mf P}^\pm)$ we have
	$$ \langle \mathcal{L}_{+}'u_+,u_-\rangle_\mathrm{geod} = \langle u_+,\mathcal{L}_{-}'u_-\rangle_\mathrm{geod}. $$
	In particular, $\langle u_+,u_-\rangle_\mathrm{geod}=0$ if $u_+\in\mathcal E_{z_1}({\mathcal L}_{+}';\mc D'({\mf P}^+))$ and $u_-\in\mathcal E_{z_2}({\mathcal L}_{-}';\mc D'({\mf P}^-))$ are (co)resonant states for $z_{1}\neq z_{2}$.
\end{lemma}

\begin{proof}
	By Definition~\ref{rem:pairing1} and \eqref{eq:AlternativeDefinitionTensorProduct} we get
	\begin{align*}
		\langle \mathcal{L}_{+}'u_+,u_-\rangle_\mathrm{geod}&=\langle u_+,c_{+} \mapsto \sum_{\vec{e}\colon \tau(\vec{e})=\iota(\pi_+^{\mf E}(c_{+}))}\langle u_-,\mathbbm{1}_{{\mf P}}(\bigcdot,(\vec{e},c_{+}))\rangle\rangle\\
		&=\langle u_+,c_{+} \mapsto \langle u_-,\sum_{\vec{e}\colon\tau(\pi_-^{\mf E}(\bigcdot))=\iota(\vec{e})}\mathbbm{1}_{{\mf P}}((\bigcdot,\vec{e}),c_{+})\rangle\rangle\\
		&=\langle u_+,c_{+} \mapsto \langle \mathcal{L}_{-}'u_-,\mathbbm{1}_{{\mf P}}(\bigcdot,c_{+})\rangle\rangle=\langle u_+,\mathcal{L}_{+}'u_-\rangle_\mathrm{geod}.\qedhere
	\end{align*}
\end{proof}

\begin{remark}[Relations between some pairings]
We can just as well consider the pairing given by
$$ \mc D'({\mf P}^+)\times \mc D'({\mf P}^-)\to\C, \quad (u_+,u_-)\mapsto\langle u_+\otimes u_-,\mathbbm{1}_{{\mf P}^{(2)}}\rangle. $$
In the following we collect some elementary relations between this pairing, the geodesic pairing and the vertex and edge pairings of Definition~\ref{rem:finite graphs}.
\begin{enumerate}
\item[(i)] Since $\mathbbm{1}_{{\mf P}^{(2)}}(\bigcdot, c_+)=\mathbbm{1}_{\pi_-^{-1}(\pi_+(c_+))}$, we obtain with \eqref{eq:ExplicitFormulaPushForwardFromSigmaToX}:
\[
\langle u_-, \mathbbm{1}_{{\mf P}^{(2)}}(\bigcdot, c_+)\rangle
= \langle  u_-,\mathbbm{1}_{\pi_-^{-1}(\pi_+(c_+))}\rangle
=  (\pi_{-,*}u_-)(\pi_+(c_+))
= ((\pi_{-,*}u_-)\circ \pi_+)(c_+).
\]
Thus we find
\begin{align*}
\langle u_+\otimes u_-,\mathbbm{1}_{{\mf P}^{(2)}}\rangle&\overset{\eqref{eq:AlternativeDefinitionTensorProduct}}{=} \langle u_+, \big(\pi_{-,*}u_-\big)\circ \pi_+ \rangle= \langle \pi_{+,*}u_+, \pi_{-,*}u_- \rangle_{\mf X}\overset{\ref{rem:finite graphs}}{=} \langle u_+,u_-\rangle_{(\mf X)}.
\end{align*}

\item[(ii)]
Note that
\begin{align*}
\mathbbm{1}_{{\mf P}}(\bigcdot, c_+)
&=\mathbbm{1}_{\pi_-^{-1}(\pi_+(c_+))} (\mathbbm{1}- \mathbbm{1}_{(\iota\circ \pi_-^{\mf E})^{-1}(\tau\circ \pi_+^{\mf E} (c_+))} )\\
&=\mathbbm{1}_{(\tau\circ \pi_-^{\mf E})^{-1}(\iota\circ \pi_+^{\mf E} (c_+))} (\mathbbm{1}- \mathbbm{1}_{(\iota\circ \pi_-^{\mf E})^{-1}(\tau\circ \pi_+^{\mf E} (c_+))} )
\end{align*}
and
\[\big(\pi^{\mf E}_{\pm,*}u_\pm\big)(\vec{e})
= \langle\pi^{\mf E}_{\pm,*}u_\pm,\delta_{\vec{e}}\rangle
= \langle u_\pm, \delta_{\vec{e}}\circ \pi^{\mf E}_\pm\rangle
= \langle u_\pm, \mathbbm{1}_{(\pi^{\mf E}_\pm)^{-1}(\vec{e})}\rangle.
\]
In view of
\[
\mathbbm{1}_{(\tau\circ \pi_-^{\mf E})^{-1}(\iota\circ \pi_+^{\mf E} (c_+))}
\mathbbm{1}_{(\iota\circ \pi_-^{\mf E})^{-1}(\tau\circ \pi_+^{\mf E} (c_+))}
=\mathbbm{1}_{(\pi^{\mf E}_-)^{-1}(\pi^{\mf E}_+(c_+)^\mathrm{op})},
\]
we find
\begin{gather*}
\langle u_-, \mathbbm{1}_{(\pi^{\mf E}_-)^{-1}(\pi^{\mf E}_+(c_+)^\mathrm{op})}\rangle
=(\pi^{\mf E}_{-,*}u_-)(\pi^{\mf E}_+(c_+)^\mathrm{op})
=((\pi^{\mf E}_{-,*}u_-)\circ\mathrm{op}\circ \pi^{\mf E}_+)(c_+)
\end{gather*}
and
\begin{align*}
\langle u_-, \mathbbm{1}_{{\mf P}}(\bigcdot, c_+)\rangle
&=\langle u_-,\mathbbm{1}_{(\tau\circ \pi_-^{\mf E})^{-1}(\iota\circ \pi_+^{\mf E} (c_+))} (\mathbbm{1}- \mathbbm{1}_{(\iota\circ \pi_-^{\mf E})^{-1}(\tau\circ \pi_+^{\mf E} (c_+))} )\rangle
  \\
&= \langle u_-, \mathbbm{1}_{\pi_-^{-1}(\pi_+(c_+))}\rangle -
\langle u_-,\mathbbm{1}_{(\tau\circ \pi_-^{\mf E})^{-1}(\iota\circ \pi_+^{\mf E} (c_+))}
\mathbbm{1}_{(\iota\circ \pi_-^{\mf E})^{-1}(\tau\circ \pi_+^{\mf E} (c_+))} )\rangle\\
&= \langle u_-, \mathbbm{1}_{\pi_-^{-1}(\pi_+(c_+))}\rangle -
\langle u_-,\mathbbm{1}_{(\pi_-^{\mf E})^{-1}(\pi_+^{\mf E} (c_+)^\mathrm{op})} \rangle\\
&= \langle u_-, \mathbbm{1}_{\pi_-^{-1}(\pi_+(c_+))}\rangle -
((\pi^{\mf E}_{-,*}u_-)\circ\mathrm{op}\circ \pi^{\mf E}_+)(c_+).
\end{align*}
By \eqref{eq:AlternativeDefinitionTensorProduct} and Definition~\ref{rem:finite graphs} this yields
\begin{align*}
\langle u_+,u_-\rangle_\mathrm{geod} &=\langle u_+\otimes u_-,\mathbbm{1}_{{\mf P}}\rangle\\
&=\langle u_+,u_-\rangle_{(\mf X)}-\langle u_+, (\pi^{\mf E}_{-,*}u_-)\circ\mathrm{op}\circ\pi^{\mf E}_+\rangle\\
&=\langle u_+,u_-\rangle_{(\mf X)}-\langle \mathrm{op}_*\pi^{\mf E}_{+,*} u_+,\pi^{\mf E}_{-,*}u_-\rangle\\
&\stackrel{\mathclap{\ref{rem:finite graphs}}}{=}\langle u_+,u_-\rangle_{(\mf X)}-\langle u_+,u_-\rangle_{(\mathrm{op},\mf E)}.
\end{align*}
\end{enumerate}
Items (i) and (ii) yield the pairing formulae
\begin{align}
\langle u_+\otimes u_-,\mathbbm{1}_{{\mf P}^{(2)}}\rangle
&=\langle u_+,u_-\rangle_{(\mf X)}\\
\langle u_+,u_-\rangle_\mathrm{geod} = \langle u_+\otimes u_-,\mathbbm{1}_{{\mf P}}\rangle
&=\langle u_+,u_-\rangle_{(\mf X)}-\langle u_+,u_-\rangle_{(\mathrm{op},\mf E)}\label{eq:pairing_right_edge}.
\end{align}
\end{remark}

\section{Tools available for trees}\label{sec:trees}

In this section we provide background material and technical tools available for trees. It will be used later in the study of the universal covering of finite regular graphs. So for the remainder of this section we assume $\mf G$ to be a homogeneous tree, i.e.\@ that $\mathfrak{G}$ has no cycles.

\subsection{Resonances on the boundary and Poisson transforms}
If $\mf G$ is a tree, two infinite chains (i.e.\@ non-backtracking paths) are called \emph{equivalent} if they are eventually equal, i.e.\@ if there is an index from which the edges on the paths match except for a shift. We call the set of equivalence classes $[\vec e_1,\vec e_2,\ldots)]$ of such chains the \emph{boundary at infinity} $\Omega$ of $\mf G$ (see e.g.\@ \cite[\S 2]{BHW21}). It is endowed with the topology consisting of the sets
\begin{equation}
	\partial_+\vec e \coloneqq \{\omega\in\Omega\mid \exists \text{ chain } \vec e,\vec e_2,\ldots : [\vec e,\vec e_2,\ldots]=\omega\}.\label{eq:DefinitionPartial+}
\end{equation}

We equip $\mf X$ with the metric $d$ which assigns the minimal length of concatenated sequences  $\vec e_1,\ldots,\vec e_\ell$ of directed edges connecting a pair of vertices. In particular, $\mf X$ carries the discrete topology and the canonical isomorphism ${\mf P}^\pm\cong \mathfrak{X}\times\Omega$ is a homeomorphism. Moreover, the algebraic dual $\mathcal{D}'(\Omega):= C^{\mathrm{lc}}(\Omega)'$ of $C^{\mathrm{lc}}(\Omega)$ is naturally isomorphic to the space of finitely additive measures on $\Omega$ (see \cite[Prop.\@ 3.9]{BHW21}).

We now recall from \cite{AFH23} how to realize (co)resonant states in $\mathcal{D}'(\Omega)$. This realization allows us to reduce several problems to the compact space $\Omega$.

\begin{definition}[End point maps]
The \emph{end point map} is given by
\begin{equation}
 \label{eq:B}
 B_+:{\mf P}^+\cong {\mathfrak X}\times \Omega \to \Omega,\quad (\vec{e}_1,\vec{e}_2,\ldots) \mapsto [(\vec{e}_1,\vec{e}_2,\ldots)].
\end{equation}
Similarly, we have an end point map $B_{-}$ for the opposite shift space,
\begin{gather*}
        B_-\colon {\mf P}^-\to\Omega,\quad(\ldots,\vec{e}_2,\vec{e}_1)\to[(\eop{e}_1,\eop{e}_2,\ldots)].
\end{gather*}
We define a (surjective) pushforward by $B_+$ via integration of compactly supported locally constant functions against fibers:
\begin{gather*}
        B_{+,\star}\colon C_c^{\mathrm{lc}}({\mf P}^+)\to C^{\mathrm{lc}}(\Omega),\quad f\mapsto\left(\omega \mapsto \sum_{x\in\mathfrak{X}}f(x,\omega)\right),
\end{gather*}
and, by duality, an (injective) pullback (cf.\@ \cite[Prop.~3.9]{BHW21})
\begin{equation}\label{eq:Bstar'}
B_+^\star \colon \mathcal D'(\Omega) \to  \mc D'({\mf P}^+),\quad B_+^{\star}(\varphi)(f)\coloneqq\varphi(B_{+,\star}(f)).
\end{equation}
Similarly, we obtain the pushforward
\begin{gather*}
        B_{-,\star}\colon C_c^{\mathrm{lc}}({\mf P}^-)\to C^{\mathrm{lc}}(\Omega),\quad f\mapsto\left(\omega \mapsto \sum_{x\in\mathfrak{X}}f(\omega,x)\right),
\end{gather*}
where $(\omega,x)\in {\mf P}^-$ denotes the unique path $\mathbf{e}\coloneqq(\ldots,\vec{e}_2,\vec{e}_1)$ such that $B_-(\mathbf{e})=\omega$ and $\tau(\vec{e}_1)=x$ and the pullback
\begin{gather*}
        B_{-}^{\star}\colon\mathcal{D}'(\Omega)\to \mc D'({\mf P}^-),\quad B_{-}^{\star}(\varphi)(f)\coloneqq \varphi(B _{-,\star}(f)).
\end{gather*}
\end{definition}

In order to realize (co)resonant states on the boundary, we use the Poisson kernel which requires some more notation. Let us fix a base point $o\in\mathfrak X$. For vertices $x,y\in\mf X$ and $\omega_{1}\neq\omega_{2}\in\Omega$ we write $[x,y]$ for the (unique) chain connecting $x$ and $y$, $[x,\omega_{1}[$ for the representative of $\omega_{1}$ starting at $x$ and $]\omega_{1},\omega_{2}[$ for the (unique up to shift) geodesic from $\omega_{1}$ to $\omega_{2}$.
Now let $\omega\in\Omega$ and $x\in\mf X$. Then there exists a unique $y\in\mf X$ such that $[o,\omega[\cap[x,\omega[=[y,\omega[$. We call the map
\begin{gather*}
\langle\bigcdot,\bigcdot\rangle\colon\mf X\times\Omega\to\Z, \quad \langle x,\omega\rangle\coloneqq d(o,y)-d(x,y)
\end{gather*}
the \emph{horocycle bracket}. It is locally constant and fulfills the \emph{horocycle identity}
\begin{gather}\label{eq:horocycle_id}
\langle gx,g\omega\rangle=\langle x,\omega\rangle+\langle go,g\omega\rangle
\end{gather}
for every automorphism $g$ of $\mf G$, see e.g.\@ \cite[(15)]{BHW23}.

\begin{definition}[Poisson kernel]
   For $z\in\mathbb{C}\setminus\{0\}$ we define the \emph{Poisson kernel}
\begin{gather*}
   p_{z,+}\coloneqq p_z\colon {\mf P}^+\to\mathbb{C},\quad \mathbf{e}\coloneqq(\vec{e}_1,\vec{e}_2,\ldots) \mapsto z^{\langle \iota(\vec{e}_1),B_{+}(\mathbf{e})\rangle}
\end{gather*}
and the \emph{opposite Poisson kernel}
\begin{gather*}
    p_{z,-}\colon {\mf P}^-\to \mathbb{C},\quad \mathbf{e}\coloneqq(\ldots,\vec{e}_2,\vec{e}_1) \mapsto z^{\langle\tau(\vec{e}_1),B_-(\mathbf{e})\rangle}.
\end{gather*}
\end{definition}

\begin{proposition}[{\cite[Prop.\@ 2.7, La.\@ 2.22]{AFH23}}]\label{prop:cpt_picture}
   For $z\neq0$ we have isomorphisms
  \begin{gather*}
    p_{z,\pm}B_{\pm}^\star\colon\mathcal{D}'(\Omega)\cong\mathcal E_z({\mathcal L}_{\pm}';\mc D'({\mf P}^\pm)).
\end{gather*}
Moreover, let $N_z(g,\omega)\coloneqq z^{-\langle go,g\omega\rangle}$ for each $g\in G,\ \omega\in \Omega,$ and define a representation $\pi_z$ on $\mathcal{D}'(\Omega)$ by
\begin{gather*}
    \pi_z(g)\mu\coloneqq g_{*}(N_z(g,\bigcdot)\mu).
\end{gather*}
Then $p_{z,\pm}B_{\pm}^\star$ intertwines $\pi_{z}$ with the left regular representation on $\mc D'({\mf P}^\pm)$.
\end{proposition}

From the Poisson kernels one obtains the Poisson transforms by integration.

\begin{definition}[Poisson transforms]
For $z\in\mathbb{C}\setminus\{0\}$ we define the \emph{Poisson transform}
\begin{gather*}
\mc P_{z}\coloneqq\mathcal{P}_{z,+}\colon\mc D'(\Omega)\to\mathrm{Maps}({\mf X};\C),\quad\mu\mapsto\int_\Omega p_{z}(\bigcdot,\omega)\intd\mu(\omega)
\end{gather*}
and the \emph{opposite Poisson transform}
\begin{gather*}
\mc P_{z,-}\colon\mc D'(\Omega)\to\mathrm{Maps}({\mf X};\C),\quad\mu\mapsto\int_\Omega p_{z,-}(\omega,\bigcdot)\intd\mu(\omega).
\end{gather*}
\end{definition}

The Poisson transforms can also be written as the composition of the isomorphism from Proposition~\ref{prop:cpt_picture} with the pushforward $\pi_{\pm,*}$ by the base projection.

\begin{proposition}[{\cite[Prop.\@ 2.11]{AFH23}}]\label{prop:poisson_proj}
        For each $z\in\C\setminus\{0\}$ we have
\begin{gather*}
\mc P_{z,\pm}=\pi_{\pm,*}\circ p_{z,\pm}B_{\pm}^\star.
\end{gather*}
\end{proposition}

\subsection{Tree automorphisms}\label{subsec:tree_autom}
In this section we study various properties of the automorphism group $G\coloneqq \mathrm{Aut}(\mathfrak{X})$ of $\mathfrak{G}$ which leads to integral formulas that play a crucial role in the proof of the pairing formula.
 Fix a base point $o\in\mathfrak{X}$ and a reference geodesic $]\omega_-,\omega_+[{}\subseteq \mathfrak X$ containing $o$. Then, for $K\coloneqq \mathrm{Stab}_G(o)$, we can view $\mathfrak{X}$ as the $G$-homogeneous space $G/K$. 

We first introduce some more notation.

\begin{definition}\label{def:horocycle}
    For $x\in\mathfrak{X}$ and $\omega\in\Omega$ we write
\begin{gather*}
H_\omega(x)\coloneqq\{y\in\mathfrak{X}\mid\langle x,\omega\rangle=\langle y,\omega\rangle\}
\end{gather*}
for the \emph{horocycle} passing through $x$ and $\omega$.
\end{definition}

\begin{figure}
\tikzstyle{circle}=[shape=circle,draw,inner sep=1.5pt]
\usetikzlibrary{shapes.geometric}
\begin{tikzpicture}[scale=2]
\node (0) at (0,0) [label=left:$\omega_-$]{};
\draw (1,0) node[circle,fill=black] (1) {};
\draw[dotted,thick] (0) to (1);
\path (1) ++(30:1) node (2) [regular polygon, regular polygon sides=3,inner sep=1.5pt, fill=black] {};
\path (1) ++(-30:1) node (3) [regular polygon, regular polygon sides=3,inner sep=1.5pt, fill=black] {};
\draw[-,thick] (1) to (2);
\draw[-,thick] (1) to (3);
\path (2) ++(30:0.5) node (4) [regular polygon, regular polygon sides=4,inner sep=1.5pt, fill=black] {};
\path (2) ++(-30:0.5) node (6) [regular polygon, regular polygon sides=4,inner sep=1.5pt, fill=black] {};
\path (1) ++(-30:1.5) node (5) [regular polygon, regular polygon sides=4,inner sep=1.5pt, fill=black] {};
\path (3) ++(30:0.5) node (7) [regular polygon, regular polygon sides=4,inner sep=1.5pt, fill=black] {};
\draw[-,thick] (2) to (6);
\draw[-,thick] (3) to (7);
\draw[-,thick] (2) to (4);
\draw[-,thick] (3) to (5);
\path (4) ++(-30:0.25) node (8) [regular polygon, regular polygon sides=5,inner sep=1.5pt, fill=black] {};
\path (6) ++(30:0.25) node (9) [regular polygon, regular polygon sides=5,inner sep=1.5pt, fill=black] {};
\path (6) ++(-30:0.25) node (10) [regular polygon, regular polygon sides=5,inner sep=1.5pt, fill=black] {};
\path (7) ++(30:0.25) node (11) [regular polygon, regular polygon sides=5,inner sep=1.5pt, fill=black] {};
\path (7) ++(-30:0.25) node (12) [regular polygon, regular polygon sides=5,inner sep=1.5pt, fill=black] {};
\path (5) ++(30:0.25) node (13) [regular polygon, regular polygon sides=5,inner sep=1.5pt, fill=black] {};
\draw[-,thick] (4) to (8);
\draw[-,thick] (6) to (9);
\draw[-,thick] (6) to (10);
\draw[-,thick] (7) to (11);
\draw[-,thick] (7) to (12);
\draw[-,thick] (5) to (13);
\path (1) ++(30:1.75) node (6) [regular polygon, regular polygon sides=5,inner sep=1.5pt, fill=black] {};
\path (1) ++(-30:1.75) node (7) [regular polygon, regular polygon sides=5,inner sep=1.5pt, fill=black] {};
\draw[thick] (4) to (6);
\draw[thick] (5) to (7);
\path (6) ++(-30:0.2) node (14) {};
\path (6) ++(30:0.2) node (15) {};
\draw[-,thick] (6) to (14);
\draw[-,thick] (6) to (15);
\path (8) ++(-30:0.2) node (16) {};
\path (8) ++(30:0.2) node (17) {};
\draw[-,thick] (8) to (16);
\draw[-,thick] (8) to (17);
\path (9) ++(-30:0.2) node (18) {};
\path (9) ++(30:0.2) node (19) {};
\draw[-,thick] (9) to (18);
\draw[-,thick] (9) to (19);
\path (10) ++(-30:0.2) node (20) {};
\path (10) ++(30:0.2) node (21) {};
\draw[-,thick] (10) to (20);
\draw[-,thick] (10) to (21);
\path (11) ++(-30:0.2) node (22) {};
\path (11) ++(30:0.2) node (23) {};
\draw[-,thick] (11) to (22);
\draw[-,thick] (11) to (23);
\path (12) ++(-30:0.2) node (24) {};
\path (12) ++(30:0.2) node (25) {};
\draw[-,thick] (12) to (24);
\draw[-,thick] (12) to (25);
\path (13) ++(-30:0.2) node (26) {};
\path (13) ++(30:0.2) node (27) {};
\path (27) ++(-90:0.04) node (30) {};
\path (30) ++(0:1) node [label=right:$\omega_+$](31) {};
\draw[dotted,thick] (30) to (31);
\draw[-,thick] (13) to (26);
\draw[-,thick] (13) to (27);
\path (7) ++(-30:0.2) node (28) {};
\path (7) ++(30:0.2) node (29) {};
\draw[-,thick] (7) to (28);
\draw[-,thick] (7) to (29);

\end{tikzpicture}
\caption{Horocycles of $\omega_-$ for a $3$-adic tree}
\label{fig:Omega_graduierung}
\end{figure}
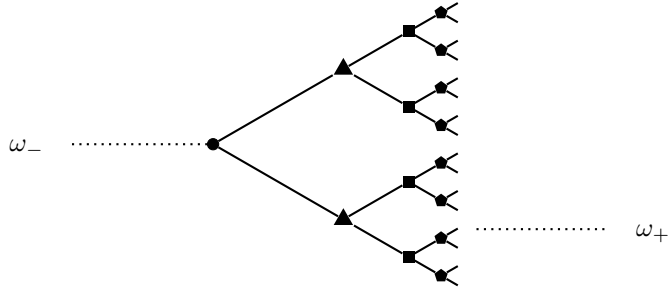

 This also gives a reference horocycle $H_{\omega_-}(o)$. The group
\[M:=\{\gamma\in K\mid \gamma|_{]\omega_-,\omega_+[}=\mathrm{id}\}=\mathrm{Stab}_G(o)\cap \mathrm{Stab}_G(\omega_-)\cap \mathrm{Stab}_G(\omega_+)\]
acts on $H_{\omega_-}(o)$, so that $G$ acts on the associated bundle $G\times_M H_{\omega_-}(o)$.
Following \cite[\S I.8]{FTN91} (see also \cite[\S 3]{Ve02}) we introduce the groups
\[B_{\omega_\pm}:=\{\gamma\in \mathrm{Aut}(\mathfrak X)\mid \gamma(\omega_\pm)=\omega_\pm , \exists\,x\in \mathfrak X: \gamma(x)=x\}.\]
Further we pick a $1$-step shift $\tau$ on $]\omega_-,\omega_+[$, i.e.\@ an automorphism of $\mathfrak X$ which acts on $]\omega_-,\omega_+[$ by mapping each vertex to one of its neighbors (see \cite[p. 157]{Ve02}), and let $\langle \tau\rangle$ be the cyclic group generated by $\tau$. Note that $\tau$ is unique up to inversion and multiplication with elements of $M$.

The following proposition from \cite[Thm.\@ 3.5 and Rem.\@ 3.6]{Ve02} will be used on various occasions. In particular it will give us another useful identity for the horocycle bracket.

\begin{proposition}[Iwasawa ``decomposition‘‘]
$G=B_{\omega_\pm}\langle\tau\rangle K$ so that the multiplication map
\[B_{\omega_\pm}\times \mathbb Z \times K \to G,\quad (u, j, k)\mapsto u\tau^j k\]
is surjective. Further, it satisfies
\[ u\tau^j k=u'\tau^{j'} k' \quad \Longrightarrow\quad j=j',\ k(k')^{-1}\in B_{\omega_\pm}\cap K,\ u(u')^{-1}\in B_{\omega_\pm}\cap K_{\tau^j(o)},\]
where we denoted $K_{\tau^j(o)}\coloneqq \mathrm{Stab}_G(\tau^j(o))=\tau^jK\tau^{-j}$.
\label{prop:Iwasawa}
\end{proposition}

\begin{proof}
The formula $G=B_{\omega_\pm}\langle\tau\rangle K$ is given in \cite[Thm.\@ 3.5]{Ve02} and the conclusions $j=j'$ and $k(k')^{-1}\in B_{\omega_\pm}\cap K$ are given in \cite[Rem.\@ 3.6]{Ve02}. For the conclusion $uB_{\omega_\pm}\cap K_{\tau^j(o)}=u'B_{\omega_\pm}\cap K_{\tau^j(o)}$ we observe that writing $k=bk'$ with $b\in  B_{\omega_\pm}\cap K$ yields $u\tau^j b =u'\tau^j$, whence $u\tau^j b \tau^{-j}=u'$. As $\tau^j b \tau^{-j}$ fixes $\omega_\pm$ and $\tau^j(o)$, we obtain $\tau^j b \tau^{-j}\in B_{\omega_\pm}\cap K_{\tau^j(o)}$, which proves the claim.
\end{proof}

Taking inverses in Proposition \ref{prop:Iwasawa} implies that $G=K\langle\tau\rangle B_{\omega_\pm}$. Moreover, $\langle\tau\rangle$ normalizes $B_{\omega_\pm}$, because if $u\in B_{\omega_\pm}$ fixes $x\in\mathfrak{X}$ then $\tau^ju\tau^{-j}$ fixes $\tau^jx$. Thus, $\langle\tau\rangle B_{\omega_\pm}$ is a group and we may also write $G$ as $\langle\tau\rangle B_{\omega_\pm}K=K B_{\omega_\pm}\langle\tau\rangle$. We now investigate the uniqueness of the latter decomposition.

\begin{corollary}\label{cor:KNA}
    We have $G=KB_{\omega_\pm}\langle\tau\rangle$ and
\[ku\tau^j=k'u'\tau^{j'}\Longrightarrow j=j',\ \exists\, b\in K\cap B_{\omega_\pm}\colon k=k'b^{-1},u=bu'.\]
In particular, the product $ku$ is uniquely defined.
\end{corollary}

\begin{proof}
For $g\in G$, Proposition \ref{prop:Iwasawa} allows us to write $g^{-1}=u\tau^jk$. Thus,
\[g=k^{-1}\tau^{-j}u^{-1}=k^{-1}\tau^{-j}u^{-1}\tau^j\tau^{-j},\]
where $\tau^{-j}u^{-1}\tau^j$ is contained in $B_{\omega_{\pm}}$ by \cite[La.\@ 3.8]{Ve02}. Moreover, if we write $g^{-1}=u'\tau^jk$, Proposition \ref{prop:Iwasawa} shows that there exists some $b\in B_{\omega_{\pm}}$ such that $k=bk'$ and $u\tau^jb\tau^{-j}=u'$. Therefore, $k^{-1}=k'^{-1}b^{-1}$ and $\tau^{-j}u'^{-1}\tau^j=b^{-1}\tau^{-j}u^{-1}\tau^j$ so that in particular
\begin{gather*}
k^{-1}\tau^{-j}u^{-1}\tau^j=k'^{-1}b^{-1}\tau^{-j}u^{-1}\tau^j=k'^{-1}\tau^{-j}u'^{-1}\tau^j.\qedhere
\end{gather*}
\end{proof}

For each $G\ni g=ku\tau^j\in KB_{\omega_{\pm}}\langle\tau\rangle$ we may thus define
\begin{gather*}
    k_{\pm}(g)\coloneqq k(K\cap B_{\omega_{\pm}}),\ u_{\pm}(g)\coloneqq (K\cap B_{\omega_{\pm}})u\quad\text{and}\quad H_{\pm}(g)\coloneqq j.
\end{gather*}
Moreover, we introduce the abbreviations $k(g)\coloneqq k_+(g),\ u(g)\coloneqq u_+(g)$ and $H(g)\coloneqq H_+(g)$. Note that the $\langle\tau\rangle$-projection is closely related to the horocycle bracket in the following way.

\begin{lemma}\label{la:horo_H}
Let $K_{\omega_+}\coloneqq \mathrm{Stab}_K(\omega_+)$ so that $K/K_{\omega_+}\cong\Omega,\ kK_{\omega_+} \mapsto k\omega_+$. Then the horocycle bracket is given by
\begin{gather*}
\langle\cdot,\cdot\rangle\colon G/K\times K/K_{\omega_+}\cong{\mf X}\times\Omega\to\Z,\quad \langle go,k\omega_+\rangle=-H(g^{-1}k).
\end{gather*}
\end{lemma}

\begin{proof}
Let $g\in G$ and $k\in K$. Using the horocycle identity \eqref{eq:horocycle_id} we first obtain
\begin{gather*}
    \langle go,k\omega_+\rangle=\langle k^{-1}go,\omega_+\rangle+\langle o,k\omega_+\rangle=\langle k^{-1}go,\omega_+\rangle.
\end{gather*}
We now write $k^{-1}g=u\tau^jk_1$ as in Proposition \ref{prop:Iwasawa}. Since $g^{-1}k=k_1^{-1}\tau^{-j}u^{-1}$ implies $H(g^{-1}k)=-j$ we have to show $\langle u\tau^jo,\omega_+\rangle=j$. However, for each $x\in\mathfrak{X}$,
\begin{gather*}
    \langle x,\omega_+\rangle=\lim_{i\to\infty} d(o,\tau^io)-d(x,\tau^io)=\lim_{i\to\infty}i-d(x,\tau^io)
\end{gather*}
so that \cite[La.\@ 3.1]{Ve02} implies
\begin{gather*}
 \langle u\tau^jo,\omega_+\rangle=\langle \tau^jo,\omega_+\rangle=\lim_{i\to\infty}i-d(\tau^jo,\tau^io)=\lim_{i\to\infty}i-(i-j)=j.\qedhere
\end{gather*}
\end{proof}

Now we can derive the announced additional identity for the horocycle bracket.

\begin{lemma}\label{la:horo_id_2}
For each $g\in G$ and $\omega\in\Omega$ we have
\begin{gather*}
\langle go,g\omega\rangle=-\langle g^{-1}o,\omega\rangle.
\end{gather*}
\end{lemma}

\begin{proof}
Let $\omega=k\omega_+$ for some $k\in K$. Then, by Lemma \ref{la:horo_H},
\begin{gather*}
 \langle go,g\omega\rangle=-H(g^{-1}k_+(gk))\qquad\text{and}\qquad-\langle g^{-1}o,\omega\rangle=H(gk).
\end{gather*}
Writing $h\coloneqq gk$ we thus have to prove $H(h)\stackrel{!}{=}-H(kh^{-1}k_+(h))=-H(h^{-1}k_+(h))$. But if $h=k_1u_1\tau^{H(h)}$ we obtain
\begin{gather*}
    h^{-1}k_+(h)=\tau^{-H(h)}u_1^{-1}(K\cap B_{\omega_+})
\end{gather*}
which implies $H(h^{-1}k_+(h))=-H(h)$ since $\langle\tau\rangle$ normalizes $B_{\omega_+}$.
\end{proof}

The following propositions will be used to introduce suitable equivariant coordinates on coverings  of finite graphs.

\begin{proposition}\label{prop:G-coord1}
The map
\[\phi: G/M\langle \tau\rangle \to (\Omega\times \Omega)\setminus \mathrm{diag}(\Omega),
\quad gM\langle \tau\rangle\mapsto (g\omega_-,g\omega_+)\]
is a $G$-equivariant bijection, where $G$ acts diagonally on $\Omega\times \Omega$.
\end{proposition}

\begin{proof}
The equivariance is obvious.

To prove surjectivity, note first that $B_{\omega_-}$ acts transitively on $\Omega\setminus \{\omega_-\}$ and $K$ acts transitively on $\Omega$. So, if $(\omega,\omega')\in  (\Omega\times \Omega)\setminus \mathrm{diag}(\Omega)$, find a $k\in K$ such that $k\cdot (\omega,\omega')= (\omega_-,\omega'')$ with $\omega''\not=\omega_-$. Then we find a $b\in B_{\omega_-}$ such that $b\cdot \omega''=\omega_+$. Together we have $bk\cdot  (\omega,\omega')=(\omega_-,\omega_+)$ which implies that $G$ acts transitively on $(\Omega\times \Omega)\setminus \mathrm{diag}(\Omega)$. This together with the equivariance proves that $\phi$ is surjective.

On the other hand, if $g\cdot \omega_-=\omega_-$ and $g\cdot \omega_+=\omega_+$, then the action of $g$ on $]\omega_-,\omega_+[$ agrees with the action of some power $\tau^j$ of $\tau$. But then $g^{-1}\tau^j$ fixes the reference geodesic pointwise, so it is contained in $M$. Thus $g\in M\langle \tau\rangle$ and $\phi$ is injective.
\end{proof}

Note that $B_{\omega_-}$ and $B_{\omega_-}\cap K$ are stable under conjugation by $M$. Thus $M$ acts on $B_{\omega_-}/(B_{\omega_-}\cap K)$ and we have the associated bundle $G\times_M  B_{\omega_-}/(B_{\omega_-}\cap K)$.

\begin{proposition}\label{prop:G-coord2}
The map
\[\psi: G\times_M  B_{\omega_-}/(B_{\omega_-}\cap K) \to G/K \times  G/M\langle \tau\rangle,
\quad [g, \bar{u}(B_{\omega_-}\cap K)] \mapsto (g\bar{u}K,gM\langle \tau\rangle)\]
is a $G$-equivariant bijection, where $G$ acts diagonally on $G/K \times  G/M\langle \tau\rangle$.
\end{proposition}

\begin{proof}
It is immediate to check from the definitions that $\psi$ is well-defined. The equivariance is then clear.

We prove that $\psi$ is bijective by explicitly describing its inverse. Our candidate for the inverse is
\[\check \psi: G/K \times  G/M\langle \tau\rangle  \to   G\times_M  B_{\omega_-}/(B_{\omega_-}\cap K) ,
\quad (hK,gM\langle \tau\rangle)  \mapsto [g\tau^j, \tau^{-j}\bar{u}\tau^{j}(B_{\omega_-}\cap K)],\]
where we use Proposition~\ref{prop:Iwasawa} to write $g^{-1}h=\bar{u}\tau^jk$. First we have to check that $\check \psi$ is well-defined. Suppose that $g^{-1}h=\bar{u}'\tau^jk'$ is another Iwasawa decomposition of $g^{-1}h$. Then there exists some $b\in B_{\omega_-}\cap K$ such that $k=bk'$ and $\bar{u}=\bar{u}'\tau^jb^{-1}\tau^{-j}$ so that
\[[g\tau^j,\tau^{-j}\bar{u}\tau^j(B_{\omega_-}\cap K) ]=[g\tau^j,\tau^{-j}\bar{u}'\tau^jb^{-1}\tau^{-j}\tau^j(B_{\omega_-}\cap K)]=[g\tau^j,\tau^{-j}\bar{u}'\tau^{j}(B_{\omega_-}\cap K) ].\]
Now, if $h'=hk'$ and $g'=gm\tau^\ell$ with $k'\in K$, $m\in M$ and $\ell\in \mathbb Z$, then
\[(g')^{-1}h'=\tau^{-\ell}m^{-1}g^{-1}hk'=(\tau^{-\ell}m^{-1}\bar{u}\tau^\ell)\tau^{j-\ell}(kk')\]
is an Iwasawa decomposition of $(g')^{-1}h'$. Thus
\begin{align*}
[g'\tau^{j-\ell}, \tau^{\ell-j}(\tau^{-\ell}m^{-1}\bar{u}\tau^\ell)\tau^{j-\ell} (B_{\omega_-}\cap K)]
&=[gm\tau^j, \tau^{-j}m^{-1}\bar{u}\tau^j (B_{\omega_-}\cap K)]\\
&=[g\tau^j(\tau^{-j}m\tau^j), (\tau^{-j}m^{-1}\tau^j)\tau^{-j}\bar{u}\tau^j (B_{\omega_-}\cap K)]\\
&=[g\tau^j, \tau^{-j}\bar{u}\tau^j (B_{\omega_-}\cap K)],
\end{align*}
so $\check \psi$ is indeed well-defined.

Next we verify that $\check \psi$ is $G$-equivariant. Fix $\gamma\in G$ and write $g^{-1}h=\bar{u}\tau^jk$ as before. Then
$\gamma\cdot (hK, gM\langle \tau\rangle)=(\gamma hK, \gamma gM\langle \tau\rangle)$
and $(\gamma g)^{-1}(\gamma h)=g^{-1}h=\bar{u}\tau^j k$ imply
\[\check \psi(\gamma hK, \gamma gM\langle \tau\rangle)
= [\gamma g \tau^j,\tau^{-j}\bar{u}\tau^j(B_{\omega_-}\cap K)]
= \gamma \cdot [g \tau^j,\tau^{-j}\bar{u}\tau^j(B_{\omega_-}\cap K)]
= \gamma \cdot \check \psi(hK, gM\langle \tau\rangle),
\]
i.e.\@ the $G$-equivariance of $\check \psi$.

In view of $\psi([\mathbf 1,\bar{u}M])=(\bar{u}K,M\langle \tau\rangle)$ and $\check\psi(\bar{u}K,M\langle \tau\rangle)=[\mathbf 1,\bar{u}M]$ the $G$-equivariance of $\psi$ and $\check \psi$ shows that $\check \psi=\psi^{-1}$, proving the proposition.
\end{proof}

\section{Coverings and deck transformations}\label{sec:coverings}

In order to use the results from Section~\ref{sec:trees}, we first relate $\mathfrak{G}$ to a tree, its \emph{universal cover} $\widetilde{\mathfrak{G}}=(\widetilde{\mathfrak{X}},\widetilde{\mathfrak{E}})$, by which we mean the graph obtained from taking the simply connected covering of the undirected version of $\mf G$ (identifying each edge with its opposite) and then replacing each unoriented edge by two opposite oriented edges. Then $\widetilde{\mathfrak{G}}$ is a $(q+1)$-regular tree. In the following we denote objects on the universal cover with a tilde. We write $\Gamma\leq G=\mathrm{Aut}(\widetilde{\mathfrak{X}})$ for the group of deck transformations and denote the canonical projection from $\widetilde{\mathfrak{X}}$ to $\mathfrak{X}$ by $\pi_{\mathfrak{X}}$. Note that the action of $G$ on $\widetilde{\mathfrak{X}}$ extends to actions on $\widetilde{\mathfrak{E}}$ and $\widetilde{{\mf P}}^{\pm}$ by acting on each vertex contained in the edge, respectively chain. Similarly, $\pi_{\mathfrak{X}}$ induces a projection $\pi_{\mathfrak{E}}\colon\widetilde{\mathfrak{E}}\to \mathfrak{E}$ such that $\iota\circ\pi_{\mathfrak{E}}=\pi_{\mathfrak{X}}\circ\tilde{\iota}$ and $\tau\circ\pi_{\mathfrak{E}}=\pi_{\mathfrak{X}}\circ\tilde{\tau}$, where $\iota,\tilde{\iota}$ resp.\@ $\tau,\tilde{\tau}$ denote the initial resp.\@ end point projections of $\mathfrak{G}$ and $\widetilde{\mathfrak{G}}$. Moreover,
\begin{gather*}
    \pi_{{\mf P}^+}(\vec{e}_1,\vec{e}_2,\ldots)\coloneqq(\pi_{\mathfrak{E}}(\vec{e}_1),\pi_{\mathfrak{E}}(\vec{e}_2),\ldots))\quad\text{and}\quad
    \pi_{{\mf P}^-}(\ldots,\vec{e}_2,\vec{e}_1)\coloneqq(\ldots,\pi_{\mathfrak{E}}(\vec{e}_2),\pi_{\mathfrak{E}}(\vec{e}_1))
\end{gather*}
define projections $\pi_{{\mf P}^{\pm}}\colon\widetilde{{\mf P}}^{\pm}\to {\mf P}^{\pm}$ as we do not allow multiple edges in $\mathfrak{G}$. Note that $\pi_{{\mf P}^{\pm}}$ is continuous with respect to the topology on the shift spaces specified in Definition~\ref{def:ShiftSpace}. Finally, we let $G$ act on $\mathrm{Maps}(\widetilde{{\mf P}}^{\pm},\mathbb{C})$ by the left regular representation and on its algebraic dual $\mathrm{Maps}(\widetilde{{\mf P}}^{\pm},\mathbb{C})'$ by duality.

For technical purposes, we also fix a fundamental domain $\mathcal{F}\subseteq\widetilde{\mathfrak{X}}$ for the $\Gamma$-action on $\widetilde{\mathfrak{X}}$ and denote the associated section, which maps $x$ to the unique vertex in $\pi_{\mathfrak{X}}^{-1}(\{x\})\cap \mathcal{F}$, by $\sigma_{\mathcal{F}}\colon \mathfrak{X}\to\widetilde{\mathfrak{X}}$. We remark that all computations are independent of the chosen fundamental domain.

In this section we will always assume that $\mf G$ is a finite graph.

\subsection{The lifting of pairings}
We now describe a connection between (co)resonant states on the universal cover $\widetilde{\mathfrak{G}}$ and those on $\mathfrak{G}$. This leads to a relation between the pairings on $\widetilde{\mathfrak{G}}$ and those on $\mathfrak{G}$ and in turn allows us to use the tools we developed in Section \ref{sec:trees}.

We start by identifying the (co)resonant states. Note first that for functions the pullback $\pi_{{\mf P}^{\pm}}^*\colon \mathrm{Maps}({\mf P}^{\pm},\mathbb{C})\to\mathrm{Maps}(\widetilde{{\mf P}}^{\pm},\mathbb{C})^{\Gamma}$ is bijective and maps $C^{\mathrm{lc}}({\mf P}^{\pm})$ onto $C^{\mathrm{lc}}(\widetilde{{\mf P}}^{\pm})^{\Gamma}$ since $\pi_{{\mf P}^{\pm}}$ is continuous and preserves districts. For the dual version of this isomorphism we define the $\Gamma$-invariant pushforward
\begin{gather*}
 \pi_{{\mf P}^{\pm},\star}\colon\left\{\begin{array}{ccc}
               C_c^{\mathrm{lc}}(\widetilde{\mf P}^{\pm}, \mathbb C) &\to & C^{\mathrm{lc}}({\mf P}^\pm, \mathbb C)\\
               f & \mapsto & (\pi_{{\mf P}^{\pm}}^*)^{-1}\left( \sum_{\gamma\in\Gamma} \gamma\cdot f\right)\\
              \end{array}\right..
\end{gather*}

The following result ensures that (co)resonance states of the finite graph $\mathfrak{G}$ correspond to $\Gamma$-invariant (co)resonant states of $\widetilde{\mathfrak{G}}$.

\begin{proposition}[{See \cite[La.\@ 3.1, Prop.\@ 3.2]{AFH23}}]\label{prop:res_upstairs_downstairs}
The pullback map
\begin{gather*}
 \pi_{{\mf P}^{\pm}}^\star\colon \mc D'({\mf P}^\pm) \to (\mc D'(\widetilde{\mf P}^\pm))^\Gamma,\quad u \mapsto u\circ \pi_{{\mf P}^{\pm},\star}
\end{gather*}
is a well-defined bijection with inverse
\begin{gather*}
  (\pi_{{\mf P}^{\pm}}^\star)^{-1}\colon \mc D'(\widetilde{\mf P}^\pm)^\Gamma \to \mc D'({\mf P}^\pm),\quad \langle(\pi_{{\mf P}^{\pm}}^\star)^{-1}(U),F\rangle=\langle U,(F\circ \pi_{{\mf P}^{+}})\cdot(\mathbbm{1}_{\mathcal{F}}\circ \pi_{\pm})\rangle.
\end{gather*}
Moreover, we have
\begin{gather*}
    \pi_{{\mf P}^{\pm}}^*\circ\mathcal{L}_{\pm}=\tilde{\mathcal{L}}_{\pm}\circ\pi_{{\mf P}^{\pm}}^*,\quad \mathcal{L}_{\pm}\circ\pi_{{\mf P}^{\pm},\star}=\pi_{{\mf P}^{\pm},\star}\circ\tilde{\mathcal{L}}_{\pm}\quad\text{and}\quad\pi_{{\mf P}^{\pm}}^\star\circ\mathcal{L}_{\pm}'=\tilde{\mathcal{L}}_{\pm}'\circ\pi_{{\mf P}^{\pm}}^\star.
\end{gather*}
In particular, $\pi_{{\mf P}^{\pm}}^\star$ induces an isomorphism $\mathcal E_z({\mathcal L}_\pm';\mc D'({\mf P}^\pm))\cong\mathcal E_z({\tilde{\mathcal L}}_\pm';\mc D'(\widetilde{\mf P}^\pm)^{\Gamma})$ for every $z\in\C$.
\end{proposition}

The isomorphism $\pi_{{\mf P}^{\pm}}^\star$ from Proposition \ref{prop:res_upstairs_downstairs} is compatible with the maps $\pi_{\pm,*}:\mc D'({\mf P}^\pm)\to  \mathrm{Maps}({\mathfrak X},\C)'$ and $\tilde{\pi}_{\pm,*}:\mc D'(\widetilde{\mf P}^\pm)\to  \mathrm{Maps}_c(\tilde{\mathfrak X},\C)'$ from Proposition~\ref{prop:push-forwards_dist_proj} in the following sense.

\begin{lemma}\label{la:pi_star_upstairs_downstairs}
Let $u \in \mc D'({\mf P}^\pm)$ and $\tilde u\coloneqq \pi_{{\mf P}^{\pm}}^\star(u )$. Then, for each $\varphi\in \mathrm{Maps}(\mathfrak{X},\mathbb{C})$,
\begin{gather*}
    \langle \pi_{\pm,*}u ,\varphi\rangle=\langle\tilde{\pi}_{\pm,*}\tilde u,\mathbbm{1}_{\mathcal{F}}\cdot (\varphi\circ \pi_{\mathfrak{X}})\rangle.
\end{gather*}
In particular, the associated functions $\pi_{\pm,*}u \in\mathrm{Maps}(\mf X,\C)$ and $\tilde{\pi}_{\pm,*}\tilde u\in\mathrm{Maps}(\tilde{\mf X},\C)$ defined by Remark~\ref{rem:push_forward_restr} fulfill
\begin{gather*}
    \pi_{\pm,*}u (\pi_{\mathfrak{X}}(x))=\tilde{\pi}_{\pm,*}\tilde u(x) \qquad \mbox{for all }x\in\widetilde{\mf X}.
\end{gather*}
\end{lemma}

\begin{proof}
Note first that
\begin{gather*}
    \langle\tilde{\pi}_{\pm,*}\tilde u,\mathbbm{1}_{\mathcal{F}}\cdot (\varphi\circ \pi_{\mathfrak{X}})\rangle
=\langle \tilde u,(\mathbbm{1}_{\mathcal{F}}\cdot (\varphi\circ \pi_{\mathfrak{X}}))\circ\tilde{\pi}_{\pm}\rangle
=\langle u ,\pi_{{\mf P}^{\pm},\star}\left((\mathbbm{1}_{\mathcal{F}}\cdot (\varphi\circ \pi_{\mathfrak{X}}))\circ\tilde{\pi}_{\pm}\right)\rangle.
\end{gather*}
On the other hand, $\langle \pi_{\pm,*}u ,\varphi\rangle=\langle u ,\varphi\circ\pi_{\pm}\rangle$. By definition of $\pi_{{\mf P}^{\pm},\star}$ we thus have to show
\begin{gather*}
    \pi_{{\mf P}^{\pm}}^*(\varphi\circ\pi_{\pm})=\sum_{\gamma\in\Gamma}\gamma\cdot\left((\mathbbm{1}_{\mathcal{F}}\cdot (\varphi\circ \pi_{\mathfrak{X}}))\circ\tilde{\pi}_{\pm}\right).
\end{gather*}
Let us restrict our attention to the case of ${\mf P}^{+}$ and let $(\vec{e}_1,\vec{e}_2,\ldots)\in{\mf P}^{+}$. Then
\begin{align*}
\sum_{\gamma\in\Gamma}\gamma\cdot\left((\mathbbm{1}_{\mathcal{F}}\cdot (\varphi\circ \pi_{\mathfrak{X}}))\circ\tilde{\pi}_{\pm}\right)(\vec{e}_1,\vec{e}_2,\ldots)
&=\sum_{\gamma\in\Gamma}\left((\mathbbm{1}_{\mathcal{F}}\cdot (\varphi\circ \pi_{\mathfrak{X}}))\circ\tilde{\pi}_{\pm}\right)(\gamma^{-1}\vec{e}_1,\gamma^{-1}\vec{e}_2,\ldots)\\
&=\sum_{\gamma\in\Gamma}\mathbbm{1}_{\mathcal{F}}(\gamma^{-1}\tilde{\iota}(\vec{e}_1))(\varphi\circ\pi_{\mathfrak{X}}\circ\tilde{\iota})(\gamma^{-1}\vec{e}_1)\\
&=(\varphi\circ\pi_{\mathfrak{X}}\circ\tilde{\iota})(\gamma_{0}^{-1}\vec{e}_1),
\end{align*}
where $\gamma_0\in\Gamma$ denotes the unique element with $\gamma_{0}^{-1}\tilde{\iota}(\vec{e}_1)\in\mathcal{F}$. Since $\pi_{\mathfrak{X}}$ is $\Gamma$-invariant we further obtain
\begin{gather*}
   (\varphi\circ\pi_{\mathfrak{X}}\circ\tilde{\iota})(\gamma_{0}^{-1}\vec{e}_1)=(\varphi\circ\pi_{\mathfrak{X}}\circ\tilde{\iota})(\vec{e}_1)=(\varphi\circ\iota\circ\pi_{\mathfrak{E}})(\vec{e}_1)=(\varphi\circ\pi_{+}\circ\pi_{{\mf P}^{+}})(\vec{e}_1,\vec{e}_2,\ldots)
\end{gather*}
which is given by $\pi_{{\mf P}^{+}}^*(\varphi\circ\pi_{+})$.

Now consider $\pi_{\pm,*}u $ and $\tilde{\pi}_{\pm,*}\tilde u$ as functions. Then
\begin{align*}
\langle\tilde{\pi}_{\pm,*}\tilde u,\mathbbm{1}_{\mathcal{F}}\cdot (\varphi\circ \pi_{\mathfrak{X}})\rangle&=\sum_{x\in\widetilde{\mathfrak{X}}}\tilde{\pi}_{\pm,*}\tilde u(x)\mathbbm{1}_{\mathcal{F}}(x)\varphi(\pi_{\mathfrak{X}}(x))\\
&=\sum_{x\in\widetilde{\mathfrak{X}}}\tilde{\pi}_{\pm,*}\tilde u(\sigma_{\mathcal{F}}(\pi_{\mathfrak{X}}(x)))\mathbbm{1}_{\mathcal{F}}(x)\varphi(\pi_{\mathfrak{X}}(x))\\
&=\sum_{x\in\mathfrak{X}}\tilde{\pi}_{\pm,*}\tilde u(\sigma_{\mathcal{F}}(x))\varphi(x)=\langle\tilde{\pi}_{\pm,*}\tilde u\circ\sigma_{\mathcal{F}},\varphi\rangle_{\mathfrak{X}}
\end{align*}
so that $\pi_{\pm,*}u (x)=\tilde{\pi}_{\pm,*}\tilde u(\sigma_{\mathcal{F}}(x))$ follows from the non-degeneracy of the pairing $\langle\cdot,\cdot\rangle_{\mathfrak{X}}$.
\end{proof}

The reasoning from Lemma \ref{la:pi_star_upstairs_downstairs} also allows to prove a similar result for the pushforwards $\pi^{\mathfrak{E}}_{\pm,*}$ and $\pi^{\widetilde{\mathfrak{E}}}_{\pm,*}$ associated with edges (see Proposition \ref{prop:push-forwards_dist_proj}).

\begin{lemma}\label{la:pi_E_star_upstairs_downstairs}
Let $u \in \mc D'({\mf P}^\pm)$ and $\tilde u\coloneqq \pi_{{\mf P}^{\pm}}^\ast(u )$. Then, for each $\varphi\in \mathrm{Maps}(\mathfrak{E},\mathbb{C})$,
\begin{gather*}
    \langle \pi^{\mathfrak{E}}_{\pm,*}u ,\varphi\rangle=\langle\pi^{\widetilde{\mathfrak{E}}}_{\pm,*}\tilde u,(\mathbbm{1}_{\mathcal{F}}\circ\tilde{\iota})\cdot (\varphi\circ \pi_{\mathfrak{E}})\rangle,
\end{gather*}
In particular, the associated functions $\pi_{\pm,*}^{\mathfrak{E}}u \in\mathrm{Maps}(\mf E,\C)$ and $\pi^{\widetilde{\mathfrak{E}}}_{\pm,*}\tilde u\in\mathrm{Maps}(\widetilde{\mf E},\C)$ defined by Remark~\ref{rem:push_forward_restr} fulfill
\begin{gather*}
    \pi^{\mathfrak{E}}_{\pm,*}u (\pi_{\mathfrak{E}}(\vec{e}))=\pi^{\widetilde{\mathfrak{E}}}_{\pm,*}\tilde u(\vec{e}) \qquad \mbox{for all }\vec{e}\in\widetilde{\mathfrak{E}}.
\end{gather*}
\end{lemma}

\begin{proof}
This follows along the lines of the proof of Lemma \ref{la:pi_star_upstairs_downstairs}.
\end{proof}

Note that Lemma \ref{la:pi_star_upstairs_downstairs} resp.\@ \ref{la:pi_E_star_upstairs_downstairs} in particular show how the vertex resp.\@ edge pairing behaves under lifting. The following lemma proves a similar connection for the geodesic pairing.

\begin{lemma}\label{la:pairing_right_up_and_down}
Let $u_\pm \in \mc D'({\mf P}^\pm)$ and $\tilde u_{\pm}\coloneqq \pi_{{\mf P}^{\pm}}^\ast(u_\pm )$ and consider
\begin{gather*}
\mathcal{F}_{{\mf P}^{\pm}}\coloneqq\{\mathbf{e}\in\widetilde{{\mf P}}^{\pm}\mid \pi_{\pm}(\mathbf{e})\in\mathcal{F}\}\quad\text{and}\quad
\mathcal{F}_{{\mf P}}\coloneqq\{(\mathbf{k},\mathbf{e})\in\widetilde{{\mf P}}^{-}\times\widetilde{{\mf P}}^{+}\mid\mathbf{k}\in\mathcal{F}_{{\mf P}^{-}}\text{ and } \mathbf{e}\in\mathcal{F}_{{\mf P}^{+}}\}.
\end{gather*}
Then, for each $\varphi\in C^{\mathrm{lc}}({\mf P})\subseteq C^{\mathrm{lc}}({\mf P}^{-}\times{\mf P}^{+})$
\begin{gather*}
    \langle u_+\otimes u_-,\varphi\rangle=\langle \tilde u_{+}\otimes \tilde u_{-},\mathbbm{1}_{\mathcal{F}_{{\mf P}}}\cdot(\varphi\circ(\pi_{{\mf P}^{-}}\otimes\pi_{{\mf P}^{+}}))\rangle.
\end{gather*}
\end{lemma}

\begin{proof}
We calculate
\begin{align*}
\langle \tilde u_{+}\otimes \tilde u_{-},(\varphi\circ(\pi_{{\mf P}^{-}}\otimes\pi_{{\mf P}^{+}}))\cdot\mathbbm{1}_{\mathcal{F}_{{\mf P}}}\rangle
=\tilde u_{+}(\mathbf{e} \mapsto \tilde u_{-}(\mathbf{k} \mapsto ((\varphi\circ(\pi_{{\mf P}^{-}}\otimes\pi_{{\mf P}^{+}}))\cdot\mathbbm{1}_{\mathcal{F}_{{\mf P}}})(\mathbf{k},\mathbf{e})).
\end{align*}
For fixed $\mathbf{e}\in\widetilde{{\mf P}}^{+}$ we have
\begin{gather*}
    (\mathbf{k} \mapsto ((\varphi\circ(\pi_{{\mf P}^{-}}\otimes\pi_{{\mf P}^{+}}))\cdot\mathbbm{1}_{\mathcal{F}_{{\mf P}}})(\mathbf{k},\mathbf{e}))=((\mathbbm{1}_{\mathcal{F}_{{\mf P}^{+}}}(\mathbf{e})\cdot\varphi\circ(\mathrm{id}\otimes\pi_{{\mf P}^{+}}(\mathbf{e})))\circ\pi_{{\mf P}^{-}})\mathbbm{1}_{\mathcal{F}_{{\mf P}^{-}}}
\end{gather*}
so that the formula for $(\pi_{{\mf P}^{\pm}}^\star)^{-1}$ from Proposition \ref{prop:res_upstairs_downstairs} implies
\begin{gather*}
    \tilde u_-(\mathbf{k} \mapsto ((\varphi\circ(\pi_{{\mf P}^{-}}\otimes\pi_{{\mf P}^{+}}))\cdot\mathbbm{1}_{\mathcal{F}_{{\mf P}}})(\mathbf{k},\mathbf{e}))=\langle u_-,(\mathbbm{1}_{\mathcal{F}_{{\mf P}^{+}}}(\mathbf{e})\cdot\varphi\circ(\mathrm{id}\otimes\pi_{{\mf P}^{+}}(\mathbf{e})))\rangle.
\end{gather*}
Thus,
\begin{align*}
    \langle \tilde u_{+}\otimes \tilde u_{-},(\varphi\circ(\pi_{{\mf P}^{-}}\otimes\pi_{{\mf P}^{+}}))\cdot\mathbbm{1}_{\mathcal{F}_{{\mf P}}}\rangle
    &=\tilde u_{+}(\mathbf{e} \mapsto \langle u_-,(\mathbbm{1}_{\mathcal{F}_{{\mf P}^{+}}}(\mathbf{e})\cdot\varphi\circ(\mathrm{id}\otimes\pi_{{\mf P}^{+}}(\mathbf{e})))\rangle)\\
&=\langle \tilde u_{+},(\mathbf{e} \mapsto \langle u_-,(\varphi\circ(\mathrm{id}\otimes\cdot ))\rangle\circ\pi_{{\mf P}^{+}}(\mathbf{e})\mathbbm{1}_{\mathcal{F}_{{\mf P}^{+}}}(\mathbf{e}))\rangle\\
&=\langle \tilde u_{+},(\langle u_-,\varphi\circ(\mathrm{id}\otimes\cdot)\rangle\circ\pi_{{\mf P}^{+}})\mathbbm{1}_{\mathcal{F}_{{\mf P}^{+}}}\rangle
\end{align*}
so that another application of Proposition \ref{prop:res_upstairs_downstairs} yields
\begin{gather*}
\langle \tilde u_{+}\otimes \tilde u_{-},(\varphi\circ(\pi_{{\mf P}^{-}}\otimes\pi_{{\mf P}^{+}}))\cdot\mathbbm{1}_{\mathcal{F}_{{\mf P}}}\rangle=\langle u_+,\langle u_-,\varphi\circ(\mathrm{id}\otimes\cdot)\rangle\rangle=\langle u_+\otimes u_-,\varphi\rangle.\qedhere
\end{gather*}
\end{proof}

\subsection{An integral formula for the vertex pairing}
Let $u_\pm \in\mathcal E_z({\mathcal L}_\pm';\mc D'({\mf P}^\pm))$ be (co)res\-o\-nant states on the finite graph $\mathfrak{G}$ and denote their $\Gamma$-invariant lifts $\pi_{{\mf P}^{\pm}}^\star(u_\pm )$ by $\tilde u_{\pm}$. We want to express the vertex pairing $\langle u_+,u_-\rangle_{(\mathfrak{X})}$ from Definition~\ref{rem:finite graphs} in terms of $\tilde u_\pm$. For this we first define measures on quotient spaces that are compatible with the pairings on the finite graph $\mathfrak{G}$.

\begin{lemma}\label{la:quotient_measures}
Let $Y$ be a locally compact Hausdorff space with a continuous $G$-action and $p\colon Y\to \widetilde{\mathfrak{X}}$ a continuous $G$-equivariant projection such that $p^{-1}(\mathcal{F})$ is a fundamental domain for the $\Gamma$-action on $Y$. Then, if $\intd y$ is a $G$-invariant Radon measure on $Y$, there is a well-defined Radon measure $\intd \Gamma y$ on $\Gamma\backslash Y$ which is characterized by
\begin{gather*}
\int_{\Gamma\backslash Y}\sum_{\gamma\in\Gamma}f(\gamma y)\intd\Gamma y=\int_Y f(y)\intd y
\end{gather*}
for each $f\in C_c(Y)$.
\end{lemma}

\begin{proof}
Let us consider the map
\begin{gather*}
\Phi\colon C_c(Y)\to {}^{\Gamma}C(Y),\quad \Phi(f)(y)\coloneqq\sum_{\gamma\in\Gamma}f(\gamma y).
\end{gather*}
First note that the sum is finite since $\gamma y\in \mathrm{supp}(f)$ implies that $\gamma p(y)=p(\gamma y)$ is contained in the finite set $p(\mathrm{supp}(f))\subseteq\widetilde{\mathfrak{X}}$ and $\Gamma$ acts freely on $\widetilde{\mathfrak{X}}$. Moreover, $\Phi$ is surjective: Indeed, if $\varphi\in {}^{\Gamma}C(Y)$ the function $f\in C_c(Y)$ defined by $f\coloneqq \mathbbm{1}_{p^{-1}(\mathcal{F})}\varphi$ is an inverse image of $\varphi$ under $\Phi$. For the well-definedness we claim that each $f\in \mathrm{ker}(\Phi)$ fulfills $\int_Yf(y)\intd y=0$. For this (using the surjectivity of $\Phi$) let $F\in C_c(Y)$ be such that $\Phi(F)$ is the constant function 1 on $Y$. Then
\begin{align*}
0&=\int_YF(y)\Phi(f)(y)\intd y=\int_YF(y)\sum_{\gamma\in\Gamma}f(\gamma y)\intd y=\sum_{\gamma\in\Gamma}\int_YF(y)f(\gamma y)\intd y\\
&=\sum_{\gamma\in\Gamma}\int_YF(\gamma^{-1}y)f(y)\intd y=\int_Y\Phi(F)(y)f(y)\intd y=\int_Yf(y)\intd y.\qedhere
\end{align*}
\end{proof}

We can now show the desired integration formula for the vertex pairing $\langle u_+,u_-\rangle_{(\mathfrak{X})}$. By Remark~\ref{rem:push_forward_restr} we know that $\tilde{\pi}_{\pm,*}(u_{\pm})$ are given by functions so that Lemma \ref{la:pi_star_upstairs_downstairs} implies
\begin{align}\label{eq:pairing_left}
    \nonumber\langle u_+,u_-\rangle_{(\mathfrak{X})}&=\sum_{x\in \mf X}\big(\pi_{+,*}u_+\big)(x) \big(\pi_{-,*}u_-\big)(x)\\
\nonumber&=\sum_{x\in \mf X}\big(\tilde{\pi}_{+,*}\tilde u_{+}\big)(\sigma_{\mathcal{F}}(x)) \big(\tilde{\pi}_{-,*}\tilde u_{-}\big)(\sigma_{\mathcal{F}}(x))\\
\nonumber&=\sum_{x\in\mathfrak{X}}\langle \tilde u_+,\mathbbm{1}_{\tilde{\pi}_+^{-1}(\sigma_{\mathcal{F}}(x))}\rangle\langle \tilde u_-,\mathbbm{1}_{\tilde{\pi}_-^{-1}(\sigma_{\mathcal{F}}(x))}\rangle\\
\nonumber&=\sum_{x\in\widetilde{\mathfrak{X}}}\mathbbm{1}_{\mathcal{F}}(x)\langle \tilde u_+,\mathbbm{1}_{\tilde{\pi}_+^{-1}(x)}\rangle\langle \tilde u_-,\mathbbm{1}_{\tilde{\pi}_-^{-1}(x)}\rangle\\
&=\sum_{\Gamma x\in\Gamma\backslash\widetilde{\mathfrak{X}}}\langle \tilde u_+,\mathbbm{1}_{\tilde{\pi}_+^{-1}(x)}\rangle\langle \tilde u_-,\mathbbm{1}_{\tilde{\pi}_-^{-1}(x)}\rangle.
\end{align}
By Proposition \ref{prop:cpt_picture} we can write $\tilde u_{\pm}=p_{z,\pm}B_{\pm}^\star(\mu_{\pm})$ for some measures $\mu_{\pm}\in\mathcal{D}'(\Omega)$. This allows us to write
\begin{align*}
   \langle \tilde u_{\pm},\mathbbm{1}_{\tilde{\pi}_\pm^{-1}(x)}\rangle&=\langle p_{z,\pm}B_{\pm}^{\star}(\mu_{\pm}),\mathbbm{1}_{\tilde{\pi}_\pm^{-1}(x)}\rangle\\
&=\mu_{\pm}\left(\omega \mapsto \sum_{y\in\tilde{\mathfrak{X}}}p_{z,\pm}(y,\omega)\mathbbm{1}_{\tilde{\pi}_\pm^{-1}(x)}(y,\omega)=p_{z,\pm}(x,\omega)\right)\\
&=\int_{\Omega}z^{\langle x,\omega\rangle}\intd\mu_{\pm}(\omega)
\end{align*}
so that Equation \eqref{eq:pairing_left} implies
\begin{gather}\label{eq:pairing_left_triple_int_distr}
    \langle u_+,u_-\rangle_{(\mathfrak{X})}=\sum_{\Gamma x\in\Gamma\backslash\widetilde{\mathfrak{X}}}\quad \int_{\Omega}\int_{\Omega}z^{\langle x,\omega_{1}\rangle}z^{\langle x,\omega_{2}\rangle}\intd \mu_{-}(\omega_{1})\intd \mu_{+}(\omega_{2}).
\end{gather}

We now define a distribution $T$ on $\widetilde{\mathfrak{X}}\times\Omega\times\Omega\cong\widetilde{{\mf P}}^{(2)}$ by
\begin{gather*}
    T\colon C_c^{\mathrm{lc}}(\widetilde{{\mf P}}^{(2)})\ni\varphi \mapsto \sum_{x\in\widetilde{\mathfrak{X}}}\int_{\Omega}\int_{\Omega}\varphi(x,\omega_{1},\omega_{2})z^{\langle x,\omega_{1}\rangle}z^{\langle x,\omega_{2}\rangle}\intd \mu_{-}(\omega_{1})\intd \mu_{+}(\omega_{2})\in\mathbb{C}.
\end{gather*}
Note that by Proposition \ref{prop:cpt_picture} the $\Gamma$-invariance of $\tilde u_{\pm}$ implies the $\Gamma$-invariance of $\mu_{\pm}$ with respect to $\pi_{z}$. This leads to the following lemma.

\begin{lemma}\label{la:T_Gamma_inv}
The distribution $T$ is $\Gamma$-invariant, i.e.\@
\begin{gather*}
T(\varphi)=T(\gamma\varphi)
\end{gather*}
for each $\varphi\in C_c^{\mathrm{lc}}(\widetilde{{\mf P}}^{(2)})$ and $\gamma\in\Gamma$.
\end{lemma}

\begin{proof}
For each $\gamma\in \Gamma$ we have
\begin{align*}
    &\hphantom{{}={}}\sum_{x\in\widetilde{\mathfrak{X}}}\int_{\Omega}\int_{\Omega}\varphi(\gamma x,\gamma\omega_{1},\gamma\omega_{2})z^{\langle x,\omega_{1}\rangle}z^{\langle x,\omega_{2}\rangle}\intd \mu_{-}(\omega_{1})\intd \mu_{+}(\omega_{2})\\
&=\sum_{x\in\widetilde{\mathfrak{X}}}\int_{\Omega}\int_{\Omega}\varphi(\gamma x,\gamma\omega_{1},\gamma\omega_{2})z^{\langle x,\omega_{1}\rangle}z^{\langle \gamma o,\gamma\omega_{1}\rangle}z^{-\langle \gamma o,\gamma\omega_{1}\rangle}z^{\langle x,\omega_{2}\rangle}\intd \mu_{-}(\omega_{1})\intd \mu_{+}(\omega_{2})\\
&=\sum_{x\in\widetilde{\mathfrak{X}}}\int_{\Omega}\int_{\Omega}\varphi(\gamma x,\omega_{1},\gamma\omega_{2})z^{\langle x,\gamma^{-1}\omega_{1}\rangle}z^{\langle \gamma o,\omega_{1}\rangle}z^{\langle x,\omega_{2}\rangle}\intd (\pi_{z}(\gamma)\mu_{-})(\omega_{1})\intd \mu_{+}(\omega_{2})\\
&=\sum_{x\in\widetilde{\mathfrak{X}}}\int_{\Omega}\int_{\Omega}\varphi(\gamma x,\omega_{1},\omega_{2})z^{\langle x,\gamma^{-1}\omega_{1}\rangle}z^{\langle \gamma o,\omega_{1}\rangle}z^{\langle x,\gamma^{-1}\omega_{2}\rangle}z^{\langle \gamma o,\omega_{2}\rangle}\intd (\pi_{z}(\gamma)\mu_{-})(\omega_{1})\intd (\pi_{z}(\gamma)\mu_{+})(\omega_{2})\\
&=\sum_{x\in\widetilde{\mathfrak{X}}}\int_{\Omega}\int_{\Omega}\varphi(\gamma x,\omega_{1},\omega_{2})z^{\langle x,\gamma^{-1}\omega_{1}\rangle}z^{\langle \gamma o,\omega_{1}\rangle}z^{\langle x,\gamma^{-1}\omega_{2}\rangle}z^{\langle \gamma o,\omega_{2}\rangle}\intd \mu_{-}(\omega_{1})\intd \mu_{+}(\omega_{2})\\
&=\sum_{x\in\widetilde{\mathfrak{X}}}\int_{\Omega}\int_{\Omega}\varphi(x,\omega_{1},\omega_{2})z^{\langle \gamma^{-1}x,\gamma^{-1}\omega_{1}\rangle}z^{\langle \gamma o,\omega_{1}\rangle}z^{\langle \gamma^{-1}x,\gamma^{-1}\omega_{2}\rangle}z^{\langle \gamma o,\omega_{2}\rangle}\intd \mu_{-}(\omega_{1})\intd \mu_{+}(\omega_{2}).
\end{align*}
But now the horocycle identity \eqref{eq:horocycle_id} and Lemma \ref{la:horo_id_2} imply that for each $\omega\in\Omega$
\begin{gather*}
    \langle \gamma^{-1}x,\gamma^{-1}\omega\rangle+\langle \gamma o,\omega\rangle=\langle x,\omega\rangle+\langle\gamma^{-1}o,\gamma^{-1}\omega\rangle+\langle \gamma o,\omega\rangle=\langle x,\omega\rangle.\qedhere
\end{gather*}
\end{proof}

Because of its $\Gamma$-invariance, $T$ factors through the quotient $\Gamma\backslash\widetilde{{\mf P}}^{(2)}$ and we may define a corresponding distribution $T_{\Gamma}$ in the following way.

\begin{lemma}
The map
\begin{gather*}
    \Phi\colon C_c^{\mathrm{lc}}(\widetilde{{\mf P}}^{(2)})\to C^{\mathrm{lc}}(\widetilde{{\mf P}}^{(2)})^{\Gamma},\quad \varphi \mapsto \sum_{\gamma\in\Gamma}\gamma\varphi
\end{gather*}
is surjective. Moreover, the map $T_{\Gamma}(\Phi(\varphi))\coloneqq T(\varphi)$ is well-defined.
\end{lemma}

\begin{proof}
This follows as in Lemma \ref{la:quotient_measures}.
\end{proof}

We can now express the vertex pairing $\langle u_+,u_-\rangle_{(\mathfrak{X})}$ from Equation \eqref{eq:pairing_left_triple_int_distr} in terms of $T_{\Gamma}$. Indeed, we have
\begin{gather}\label{eq:pairing_right_T_Gamma}
    T_{\Gamma}(\mathbbm{1})=\langle u_+,u_-\rangle_{(\mathfrak{X})}
\end{gather}
since
\begin{gather*}
    T_{\Gamma}(\mathbbm{1})=T(\mathbbm{1}_{\mathcal{F}})=\sum_{x\in\widetilde{\mathfrak{X}}}\mathbbm{1}_{\mathcal{F}}(x)\int_{\Omega}\int_{\Omega}z^{\langle x,\omega_{1}\rangle}z^{\langle x,\omega_{2}\rangle}\intd \mu_{-}(\omega_{1})\intd \mu_{+}(\omega_{2}),
\end{gather*}
which equals the right hand side of Equation \eqref{eq:pairing_left_triple_int_distr}.

The distribution $T_\Gamma$ can also be used to describe the geodesic pairing from Definition~\ref{rem:pairing1}.

\begin{lemma}\label{lem:TGamma-product pairing}
    $T_{\Gamma}(\mathbbm{1}_{\widetilde{{\mf P}}})=\langle u_+,u_-\rangle_\mathrm{geod}$.
\end{lemma}

\begin{proof}
By Lemma \ref{la:pairing_right_up_and_down} we have
\begin{gather*}
    \langle u_+,u_-\rangle_\mathrm{geod}
    =\langle \tilde u_{+}\otimes \tilde u_{-},(\mathbbm{1}_{{\mf P}}\circ(\pi_{{\mf P}^{-}}\otimes\pi_{{\mf P}^{+}}))\cdot\mathbbm{1}_{\mathcal{F}_{{\mf P}}}\rangle
    =\langle \tilde u_{+}\otimes \tilde u_{-},\mathbbm{1}_{\widetilde{{\mf P}}\cap\mathcal{F}_{{\mf P}}}\rangle.
\end{gather*}
By definition of $u_{\pm}$ the latter pairing is given by
\begin{gather*}
\int_{\Omega}\int_{\Omega}\sum_{x\in\widetilde{\mathfrak{X}}}\sum_{y\in\widetilde{\mathfrak{X}}}z^{\langle x,\omega_{1}\rangle}z^{\langle y,\omega_{2}\rangle}\mathbbm{1}_{\widetilde{{\mf P}}\cap\mathcal{F}_{{\mf P}}}((\omega_{1},x),(y,\omega_{2}))\intd\mu_{+}(\omega_{2})\intd\mu_{-}(\omega_{1})\\
=\int_{\Omega}\int_{\Omega}\sum_{x\in\widetilde{\mathfrak{X}}}z^{\langle x,\omega_{1}\rangle}z^{\langle y,\omega_{2}\rangle}\mathbbm{1}_{\widetilde{{\mf P}}}((\omega_{1},x),(x,\omega_{2}))\mathbbm{1}_{\mathcal{F}}(x)\intd\mu_{+}(\omega_{2})\intd\mu_{-}(\omega_{1})=T(\mathbbm{1}_{\widetilde{{\mf P}}}\mathbbm{1}_{\mathcal{F}}).
\end{gather*}
But now note that for each $(x,\omega_{1},\omega_{2})\in\widetilde{\mathfrak{X}}\times\Omega\times\Omega$
\begin{gather*}
    \sum_{\gamma\in\Gamma}(\mathbbm{1}_{\widetilde{{\mf P}}}\mathbbm{1}_{\mathcal{F}})(\gamma x,\gamma\omega_{1},\gamma\omega_{2})=\mathbbm{1}_{\widetilde{{\mf P}}}(x,\omega_{1},\omega_{2})\sum_{\gamma\in\Gamma}\mathbbm{1}_{\mathcal{F}}(\gamma x)=\mathbbm{1}_{\widetilde{{\mf P}}}(x,\omega_{1},\omega_{2})
\end{gather*}
so that $T(\mathbbm{1}_{\widetilde{{\mf P}}}\mathbbm{1}_{\mathcal{F}})=T_{\Gamma}(\mathbbm{1}_{\widetilde{{\mf P}}})$.
\end{proof}

\section{The pairing formula}\label{sec:pairingformula}

The key to our proof of the pairing formula expressing the vertex pairing of (co)resonant states in terms of the geodesic pairing is a specific cut-off that allows us to split $T(\varphi)$ up into two parts, which can be treated separately.

\begin{definition}\label{def:cutoff}
For $n\in\mathbb{N}$ we consider the sets
\begin{gather*}
    S_{n}\coloneqq\{(x,\omega_{1},\omega_{2})\in\widetilde{\mathfrak{X}}\times\Omega\times\Omega\mid d(x,]\omega_1,\omega_2[)\leq n\},
\end{gather*}
where $d(x,]\omega_1,\omega_2[)$ denotes the distance from $x$ to $]\omega_{1},\omega_{2}[$. Note that each $S_{n}$ is $G$-invariant so that the associated cutoff functions $\mathbbm{1}_{S_{n}}$ and $\mathbbm{1}_{S_{n}^{c}}$ are also $G$-invariant. Moreover, both functions are locally constant.
\end{definition}

We now set
\begin{gather*}
\mathrm{I}_c(\varphi,n)\coloneqq T(\mathbbm{1}_{S_{n}}\varphi)\text{ and }\mathrm{I}_r(\varphi,n)\coloneqq T(\mathbbm{1}_{S_{n}^{c}}\varphi)
\end{gather*}
so that $T(\varphi)=\mathrm{I}_c(\varphi,n)+\mathrm{I}_r(\varphi,n)$. Similarly, for $\tilde\varphi\in C^{\mathrm{lc}}(\widetilde{{\mf P}}^{(2)})^{\Gamma}$, we set
\begin{gather*}
\mathrm{I}_c^{\Gamma}(\tilde\varphi,n)\coloneqq T_{\Gamma}(\mathbbm{1}_{S_{n}}\tilde\varphi)\text{ and }\mathrm{I}_r^{\Gamma}(\tilde\varphi,n)\coloneqq T_{\Gamma}(\mathbbm{1}_{S_{n}^{c}}\tilde\varphi)
\end{gather*}
so that $T_{\Gamma}(\tilde\varphi)=\mathrm{I}_c^{\Gamma}(\tilde\varphi,n)+\mathrm{I}_r^{\Gamma}(\tilde\varphi,n)$ and in particular, by Equation \eqref{eq:pairing_right_T_Gamma},
\begin{gather}\label{eq:decomp_left}
    \langle u_+,u_-\rangle_{(\mathfrak{X})}=\mathrm{I}_c^{\Gamma}(\mathbbm{1},n)+\mathrm{I}_r^{\Gamma}(\mathbbm{1},n).
\end{gather}

\subsection{Equivariant coordinates and the calculation of \texorpdfstring{$\mathrm{I}_c^{\Gamma}(\mathbbm{1},n)$}{IcGamma(1,n)}}
\label{sec:coord}

Combining Propositions~\ref{prop:G-coord1} and \ref{prop:G-coord2} we obtain a $G$-equivariant bijection
\[\widetilde{\mathcal{A}}\coloneqq\psi^{-1}\circ (\mathrm{id},\phi^{-1})\colon \widetilde{\mathfrak X}\times(\Omega\times \Omega)\setminus \mathrm{diag}(\Omega)\to G\times_M  B_{\omega_-}/(B_{\omega_-}\cap K),\]
where we identify $G/K\cong\widetilde{\mathfrak{X}}$ by $gK \mapsto go$. Restricting $\psi$ to $G/M\cong G\times_{M}\{1\}$ also gives rise to coordinates on $\widetilde{{\mf P}}$:
\begin{gather*}
    (\mathrm{id},\phi)\circ\psi|_{G\times_M\{1\}} \colon G/M\to\widetilde{{\mf P}},\qquad gM \mapsto (go,g\omega_{-},g\omega_{+}).
\end{gather*}
 The coordinates also fit well with the cutoff functions $\mathbbm{1}_{S_{n}}$ from Definition \ref{def:cutoff}:

\begin{lemma}\label{la:cutoff_pushforward}
Write $]\omega_{-},o]:=(\ldots,x_{-1},x_{0})$ and let $\chi_{n}\coloneqq\mathbbm{1}_{G\times_MB_{\omega_{-},n}/(B_{\omega_{-},n}\cap K)}$ with $B_{\omega_{-},n}\coloneqq \mathrm{Stab}_{B_{\omega_{-}}}(x_{-n})$. Then $\widetilde{\mathcal{A}}_{*}(\mathbbm{1}_{S_{n}})\coloneqq\mathbbm{1}_{S_{n}}\circ\widetilde{\mathcal{A}}^{-1}=\chi_n$.
\end{lemma}

\begin{proof}
Note first that $\widetilde{\mathcal{A}}^{-1}([g,\bar{u}(B_{\omega_{-}}\cap K)])=(g\bar{u}o,g\omega_{-},g\omega_{+})$ and, writing $h\coloneqq g\bar{u}$,
\begin{gather*}
    (g\bar{u}o,g\omega_{-},g\omega_{+})=(ho,h\omega_{-},h\bar{u}^{-1}\omega_{+})=h.(o,\omega_{-},\bar{u}^{-1}\omega_{+}),
\end{gather*}
since $\bar{u}$ stabilizes $\omega_{-}$. We now claim that $\bar{u}^{-1}\omega_{+}\in\partial_{+}(x_{-(n+1)},x_{-n})$ (with $\partial_+\vec e$ as in \eqref{eq:DefinitionPartial+}) if and only if $\bar{u}\in B_{\omega_{-},n}$. Indeed, if $\bar{u}\in B_{\omega_{-},n}$ we have
\begin{gather*}
    \bar{u}^{-1}]\omega_{-},\omega_{+}[=(\ldots,x_{-(n+1)},x_{-n},\bar{u}^{-1}x_{-n+1},\bar{u}^{-1}x_{-n+2},\ldots).
\end{gather*}
Now note that $\bar{u}^{-1}x_{-n+1}$ is a neighbor of $x_{-n}$ different from $x_{-(n+1)}=\bar{u}^{-1}x_{-(n+1)}$. Thus, $\bar{u}^{-1}\omega_{+}\in\partial_{+}(x_{-(n+1)},x_{-n})$.

On the other hand suppose that $\bar{u}\not\in B_{\omega_{-},n}$. Then $\bar{u}\in B_{\omega_{-},m}$ for some $m>n$. Then, denoting $y_{\ell}\coloneqq\bar{u}^{-1}x_{\ell}$ we obtain $y_{\ell}=x_{\ell}$ for $\ell\leq -m$ and for $j\in\{n,n+1\}$
\begin{gather*}
    d(x_{-m},y_{-j})=d(y_{-m},y_{-j})=d(x_{-m},x_{-j})=m-j.
\end{gather*}
Thus, $(x_{-(n+1)},x_{-n})$ and $(y_{-(n+1)},y_{-n})$ are two edges whose initial points have distance $m-n-1$ to $x_{-m}$, but which point in two different directions as $y_{-n}\neq x_{-n}$. Thus, $\partial_{+}(x_{-(n+1)},x_{-n})\cap\partial_{+}(y_{-(n+1)},y_{-n})=\emptyset$. Since $\bar{u}^{-1}\omega_{+}\in\partial_{+}(y_{-(n+1)},y_{-n})$ this proves the claim and thus the lemma since the $G$-action preserves distances.
\end{proof}

We want to apply the coordinates to $\mathrm{I}_{c}(\varphi,n)$. For this we first choose an appropriate, canonical measure and approximate the measures $\mu_{\pm}$ by functions with respect to it. Then we calculate the coordinate transform $\widetilde{\mathcal{A}}$ in terms of the fixed measure and in turn obtain the transform of $\mathrm{I}_{c}(\varphi,n)$.

We first describe the canonical measure which allows us to embed functions into distributions. In this step we need the integration formula for the Iwasawa decomposition provided in the appendix.

\begin{proposition}\label{prop:meas_Sigma2}
For $f\in C_c(\widetilde{{\mf P}}^{(2)})$ we define
\begin{gather*}
    \int_{\widetilde{{\mf P}}^{(2)}}f\intd\mu_{\widetilde{{\mf P}}^{(2)}}\coloneqq \sum_{x\in\widetilde{\mathfrak{X}}}\quad\iint\limits_{(K/K\cap B_{\omega_{+}})^2}f(x,k_{-}\omega_{+},k_{+}\omega_{+})q^{\langle x,k_{+}\omega_{+}\rangle+\langle x,k_{-}\omega_{+}\rangle}\intd k_{-}\intd k_{+}.
\end{gather*}
Then, for each $g\in G$,
\begin{gather}\label{eq:meas_Sigma2_inv}
    \int_{\widetilde{{\mf P}}^{(2)}}g.f\intd\mu_{\widetilde{{\mf P}}^{(2)}}=\int_{\widetilde{{\mf P}}^{(2)}}f\intd\mu_{\widetilde{{\mf P}}^{(2)}}.
\end{gather}
Moreover, the measure $\mu_{\widetilde{{\mf P}}^{(2)}}$ is independent of the choice of $o$ and $\omega_{+}$.
\end{proposition}

\begin{proof}
We first prove the $G$-invariance. Using $k(g^{-1}k(gk))=k(K\cap B_{\omega_{+}})$ and Lemma \ref{la:integral_K_shift} we infer
\begin{align*}
&\hphantom{{}={}}\sum_{x\in\widetilde{\mathfrak{X}}}\quad \iint\limits_{(K/K\cap B_{\omega_{+}})^2}\!f(gx,gk_{-}\omega_{+},gk_{+}\omega_{+})q^{\langle x,k_{+}\omega_{+}\rangle+\langle x,k_{-}\omega_{+}\rangle}\intd k_{-}\intd k_{+}\\
&=\sum_{x\in\widetilde{\mathfrak{X}}}\quad\iint\limits_{(K/K\cap B_{\omega_{+}})^2}\!f(gx,k(gk_{-})\omega_{+},k(gk_{+})\omega_{+})q^{\langle x,k(g^{-1}k(gk_{+}))\omega_{+}\rangle+\langle x,k(g^{-1}k(gk_{-}))\omega_{+}\rangle}\intd k_{-}\intd k_{+}\\
&=\sum_{x\in\widetilde{\mathfrak{X}}}\quad\iint\limits_{(K/K\cap B_{\omega_{+}})^2}\!f(gx,k_{-}\omega_{+},k_{+}\omega_{+})q_{x,g}(k_{-})q_{x,g}(k_{+})\intd k_{-}\intd k_{+},
\end{align*}
where $q_{x,g}(k_{\pm})\coloneqq q^{\langle x,k(g^{-1}k_{\pm})\omega_{+}\rangle-H(g^{-1}k_{\pm})}$.
But now Lemma \ref{la:horo_H} and the horocycle identity imply
\begin{gather*}
    \langle g^{-1}x,k(g^{-1}k_{\pm})\omega_{+}\rangle-H(g^{-1}k_{\pm})=\langle g^{-1}x,g^{-1}k_{\pm}\omega_{+}\rangle+\langle go,k_{\pm}\omega_{+}\rangle=\langle x,k_{\pm}\omega_{+}\rangle,
\end{gather*}
proving \eqref{eq:meas_Sigma2_inv}. For the independence we first consider another choice $k_0\omega_{+},\ k_0\in K,$ instead of $\omega_{+}$. Then we have
\begin{gather*}
    \iint\limits_{(K/K\cap B_{k_{0}\omega_{+}})^2}f(x,k_{-}k_{0}\omega_{+},k_{+}k_{0}\omega_{+})q^{\langle x,k_{+}k_{0}\omega_{+}\rangle+\langle x,k_{-}k_{0}\omega_{+}\rangle}\intd k_{-}\intd k_{+}=\\
    \iint\limits_{(K/K\cap B_{\omega_{+}})^2}f(x,k_{-}\omega_{+},k_{+}\omega_{+})q^{\langle x,k_{+}\omega_{+}\rangle+\langle x,k_{-}\omega_{+}\rangle}\intd k_{-}\intd k_{+}
\end{gather*}
by the $K$-invariance of the Haar measure on $K$ and since $B_{k_{0}\omega_{+}}=k_{0}B_{\omega_{+}}k_{0}^{-1}$ so that $K\cap B_{k_{0}\omega_{+}}$ and $K\cap B_{\omega_{+}}$ have the same measure. Another choice $\tilde{o}=go$ of $o$ not only changes the group $K=\mathrm{Stab}_{G}(o)$ to $K_{go}=gKg^{-1}$ but also changes the horocycle bracket. We denote the horocycle bracket attached to $\tilde{o}$ by $\langle\cdot,\cdot\rangle_{\tilde{o}}$ and obtain $\langle gx,g\omega\rangle_{\tilde{o}}=\langle x,\omega\rangle$ for each $x\in\widetilde{\mathfrak{X}}$ and $\omega\in \Omega$. The Haar measure on $K_{go}$ with full measure one is given by
\begin{gather*}
    \int_{K_{go}}\varphi(k)\intd k=\int_K\varphi(gkg^{-1})\intd k.
\end{gather*}
We thus infer that the volume of $K_{go}\cap B_{\omega_{+}}$ equals the volume of $K\cap g^{-1}B_{\omega_{+}}g=K\cap B_{g^{-1}\omega_{+}}$. Therefore we obtain
\begin{align*}
&\hphantom{{}={}}\sum_{x\in\widetilde{\mathfrak{X}}}\quad\iint\limits_{(K_{go}/K_{go}\cap B_{\omega_{+}})^2}f(x,k_{-}\omega_{+},k_{+}\omega_{+})q^{\langle x,k_{+}\omega_{+}\rangle_{go}+\langle x,k_{-}\omega_{+}\rangle_{go}}\intd k_{-}\intd k_{+}\\
&=\sum_{x\in\widetilde{\mathfrak{X}}}\quad\iint\limits_{(K/K\cap B_{g^{-1}\omega_{+}})^2}f(x,gk_{-}g^{-1}\omega_{+},gk_{+}g^{-1}\omega_{+})q^{\langle x,gk_{+}g^{-1}\omega_{+}\rangle_{\tilde{o}}+\langle x,gk_{-}g^{-1}\omega_{+}\rangle_{\tilde{o}}}\intd k_{-}\intd k_{+}\\
&=\sum_{x\in\widetilde{\mathfrak{X}}}\quad\iint\limits_{(K/K\cap B_{g^{-1}\omega_{+}})^2}f(x,gk_{-}g^{-1}\omega_{+},gk_{+}g^{-1}\omega_{+})q^{\langle g^{-1}x,k_{+}g^{-1}\omega_{+}\rangle+\langle g^{-1}x,k_{-}g^{-1}\omega_{+}\rangle}\intd k_{-}\intd k_{+}\\
&=\sum_{x\in\widetilde{\mathfrak{X}}}\quad\iint\limits_{(K/K\cap B_{g^{-1}\omega_{+}})^2}f(gx,gk_{-}g^{-1}\omega_{+},gk_{+}g^{-1}\omega_{+})q^{\langle x,k_{+}g^{-1}\omega_{+}\rangle+\langle x,k_{-}g^{-1}\omega_{+}\rangle}\intd k_{-}\intd k_{+}\\
&=\sum_{x\in\widetilde{\mathfrak{X}}}\quad\iint\limits_{(K/K\cap B_{\omega_{+}})^2}f(gx,gk_{-}\omega_{+},gk_{+}\omega_{+})q^{\langle x,k_{+}\omega_{+}\rangle+\langle x,k_{-}\omega_{+}\rangle}\intd k_{-}\intd k_{+}=\int_{\widetilde{{\mf P}}^{(2)}}f\intd\mu_{\widetilde{{\mf P}}^{(2)}},
\end{align*}
where we used the independence of the choice of $\omega_{+}$ in the penultimate step.
\end{proof}

The following result allows us to weakly approximate the measures $\mu_{\pm}\in\mathcal{D}'(\Omega)$ by locally constant functions multiplied with the Haar measure $\intd k$ on $K$.

\begin{lemma}[{see \cite[Lemma 5.7]{AFH23}}]\label{la:meas_approx}
For every $\mu\in\mathcal{M}_{\textup{fa}}(\Omega)$ there exists a sequence $\{\chi_n\}_n\subseteq C^{\mathrm{lc}}(\Omega)$ such that
$$ \int_Kf(k\omega_0)\chi_n(k\omega_0)\intd k \to \int_\Omega f\intd\mu \qquad \mbox{for all }f\in C^{\mathrm{lc}}(\mathcal{O}),\omega_0\in\Omega. $$
\end{lemma}

For the transformation of the measure $\mu_{\widetilde{{\mf P}}^{(2)}}$ we have the following result which uses the computation of pushforward measures and in particular the integration formula for the Bruhat decomposition provided in the appendix.

\begin{lemma}\label{la:pushforward_meas}
There is a constant $c>0$ such that
\begin{gather*}
    \widetilde{\mathcal{A}}_*(q^{\langle x,k_+\omega_+\rangle}\intd k_+q^{\langle x,k_-\omega_+\rangle}\intd k_-\intd x)=c\cdot q^{\langle\bar{u}o,\omega_+\rangle}(\widetilde{\mathrm{pr}}_M)_*(\intd g\intd\bar{u}),
\end{gather*}
i.e.\@, for each $\varphi\in C_c^{\mathrm{lc}}(\widetilde{{\mf P}}^{(2)})$,
\begin{gather*}
\int_{\widetilde{{\mf P}}^{(2)}}\varphi\intd\mu_{\widetilde{{\mf P}}^{(2)}}=c\cdot\int_{G}\int_{B_{\omega_{-}}/B_{\omega_{-}}\cap K}(\varphi\circ\widetilde{\mathcal{A}}^{-1})\cdot q^{\langle\bar{u}o,\omega_+\rangle}\intd\bar{u}\intd g.
\end{gather*}
Moreover, if $\varphi\in C_c^{\mathrm{lc}}(\widetilde{{\mf P}})\subseteq C_c^{\mathrm{lc}}(\widetilde{{\mf P}}^{(2)})$,
\begin{gather*}
\int_{\widetilde{{\mf P}}^{(2)}}\varphi\intd\mu_{\widetilde{{\mf P}}^{(2)}}=c\cdot\int_{G/M}(\varphi\circ\widetilde{\mathcal{A}}^{-1})([g,e])\intd gM=c\cdot\int_{G/M}\varphi(go,g\omega_{-},g\omega_{+})\intd gM.
\end{gather*}
\end{lemma}

\begin{proof}
By Proposition \ref{prop:pushforward_phi} we first obtain
\begin{gather*}
    c_\varphi(\varphi^{-1})_* (\intd k_+(K\cap B_{\omega_+})\otimes \intd k_-(K\cap B_{\omega_+}))|_{(\Omega\times\Omega)\setminus \mathrm{diag}(\Omega)}=d(g\omega_+,g\omega_-)^{-1}\intd(gM\langle\tau\rangle)
\end{gather*}
such that
\begin{gather*}
c_{\varphi}(\mathrm{id},\varphi^{-1})_*(\intd x\,q^{\langle x,k_+\omega_+\rangle}\intd k_+\,q^{\langle x,k_-\omega_+\rangle}\intd k_-)=d(g\omega_+,g\omega_-)^{-1}q^{\langle x,g\omega_+\rangle}q^{\langle x,g\omega_-\rangle}\intd x\intd(gM\langle\tau\rangle)
\end{gather*}
off the diagonal in $\Omega$. Moreover, Proposition \ref{prop:pushforward_psi} implies that
\begin{gather*}
    (\psi^{-1})_*(d(g\omega_+,g\omega_-)^{-1}q^{\langle x,g\omega_+\rangle}q^{\langle x,g\omega_-\rangle}\intd x\intd(gM\langle\tau\rangle))=\\
c_{\psi}d(g\omega_+,g\omega_-)^{-1}q^{\langle g\bar{u}o,g\omega_+\rangle}q^{\langle g\bar{u}o,g\omega_-\rangle}\intd g \intd \bar{u}.
\end{gather*}
By definition of $d$ (see Definition \ref{def:d}) and Lemma \ref{la:horo_H} the latter expression is given by
\begin{gather*}
    q^{-H(g)-H(gs)+\langle g\bar{u}o,g\omega_+\rangle+\langle g\bar{u}o,g\omega_-\rangle}=q^{\langle g^{-1}o,\omega_+\rangle+\langle g^{-1}o,s\omega_+\rangle+\langle g\bar{u}o,g\omega_+\rangle+\langle g\bar{u}o,g\omega_-\rangle}.
\end{gather*}
Note that $s\omega_+=\omega_-$ makes this formula symmetric in $\pm$. We claim that $\langle g\bar{u}o,g\omega_{\pm}\rangle+\langle g^{-1}o,\omega_{\pm}\rangle=\langle\bar{u}o,\omega_{\pm}\rangle$. Indeed, Lemma \ref{la:horo_id_2} implies $\langle g^{-1}o,\omega_{\pm}\rangle=-\langle go,g\omega_{\pm}\rangle$ so that the horocycle identity \eqref{eq:horocycle_id} yields
\begin{gather*}
    \langle g\bar{u}o,g\omega_{\pm}\rangle+\langle g^{-1}o,\omega_{\pm}\rangle=\langle\bar{u}o,\omega_{\pm}\rangle+\langle go,g\omega_{\pm}\rangle-\langle go,g\omega_{\pm}\rangle=\langle\bar{u}o,\omega_{\pm}\rangle.
\end{gather*}
Finally, note that $\langle\bar{u}o,\omega_-\rangle=\langle o,\omega_-\rangle=0$ (see e.g.\@ \cite[La.\@ 3.1]{Ve02}).

If $\varphi\in C_c^{\mathrm{lc}}(\widetilde{{\mf P}})$ the right hand side of our integral formula simplifies since
\begin{align*}
&\hphantom{{}={}}\int_{G}\int_{B_{\omega_{-}}/B_{\omega_{-}}\cap K}(\varphi\circ\widetilde{\mathcal{A}}^{-1})\cdot q^{\langle\bar{u}o,\omega_+\rangle}\intd\bar{u}\intd g\\
&=\int_{G}\int_{B_{\omega_{-}}/B_{\omega_{-}}\cap K}\varphi(g\bar{u}o,g\omega_{-},g\omega_{+})\cdot q^{\langle\bar{u}o,\omega_+\rangle}\intd\bar{u}\intd g\\
&=\int_{G}\int_{B_{\omega_{-}}}\varphi(g\bar{u}o,g\omega_{-},g\omega_{+})\cdot q^{\langle\bar{u}o,\omega_+\rangle}\intd\bar{u}\intd g\\
&=\int_{G}\int_{B_{\omega_{-}}\cap K}\varphi(g\bar{u}o,g\omega_{-},g\omega_{+})\cdot q^{\langle\bar{u}o,\omega_+\rangle}\intd\bar{u}\intd g\\
&=\int_{G}\varphi(go,g\omega_{-},g\omega_{+})\intd\bar{u}\intd g,
\end{align*}
where we used $\intd\bar{u}(B_{\omega_{-}}\cap K)=q^{\langle o,\omega_{-}\rangle}=1$ (see Subsection \ref{subsec:int_formulas}) in the second and last step, and the fact that $\varphi\in C_c^{\mathrm{lc}}(\widetilde{{\mf P}})$ implies $\varphi(g\bar{u}o,g\omega_{-},g\omega_{+})=0$ if $\bar{u}o\neq o$ (if and only if $\bar{u}\not\in K$) in the penultimate step.
\end{proof}

We need to know how the integrand in $\mathrm{I}_{c}(\varphi,n)$ behaves under the coordinate change $\widetilde{\mathcal{A}}$.

\begin{lemma}\label{la:transform_function}
Let $s\in\C$ and, for $\varphi_{\pm}\in C^{\mathrm{lc}}(\Omega)$, define the function
\begin{gather*}
   f_{\varphi_{-},\varphi_{+},s}\colon \widetilde{{\mf P}}^{(2)}\to\mathbb{C},\quad (x,\omega_{1},\omega_{2}) \mapsto q^{s\langle x,\omega_{1}\rangle+s\langle x,\omega_{2}\rangle}\varphi_{+}(\omega_{2})\varphi_{-}(\omega_{1}).
\end{gather*}
Then, for all $g\in G$ and $\bar{u}\in B_{\omega_{-}}$,
\begin{gather*}
(f_{\varphi_{-},\varphi_{+},s}\circ\widetilde{\mathcal{A}}^{-1})([g,\bar{u}(B_{\omega_{-}}\cap K)])=(f_{\varphi_{-},\varphi_{+},s}\circ\widetilde{\mathcal{A}}^{-1})([g,e(B_{\omega_{-}}\cap K)])\cdot q^{s\langle\bar{u}o,\omega_{+}\rangle}.
\end{gather*}
\end{lemma}

\begin{proof}
We calculate
\begin{align*}
(f_{\varphi_{-},\varphi_{+},s}\circ\widetilde{\mathcal{A}}^{-1})([g,\bar{u}(B_{\omega_{-}}\cap K)])&=q^{s\langle g\bar{u}o,g\omega_{-}\rangle+s\langle g\bar{u}o,g\omega_{+}\rangle}\varphi_{+}(g\omega_{+})\varphi_{-}(g\omega_{-})\\
&=q^{s\langle go,g\omega_{-}\rangle+s\langle go,g\omega_{+}\rangle}\varphi_{+}(g\omega_{+})\varphi_{-}(g\omega_{-})q^{s\langle\bar{u}o,\omega_{-}\rangle+s\langle\bar{u}o,\omega_{+}\rangle}\\
&=(f_{\varphi_{-},\varphi_{+},s}\circ\widetilde{\mathcal{A}}^{-1})([g,e(B_{\omega_{-}}\cap K)])q^{s\langle\bar{u}o,\omega_{-}\rangle+s\langle\bar{u}o,\omega_{+}\rangle}\\
&=(f_{\varphi_{-},\varphi_{+},s}\circ\widetilde{\mathcal{A}}^{-1})([g,e(B_{\omega_{-}}\cap K)])q^{s\langle\bar{u}o,\omega_{+}\rangle},
\end{align*}
where we used the horocycle identity and the equality $\langle\bar{u}o,\omega_{-}\rangle=\langle o,\omega_{-}\rangle=0$ (see e.g.\@ \cite[La.\@ 3.1]{Ve02}).
\end{proof}

Before calculating $\mathrm{I}_c(\varphi,n)$ we introduce the projection
\begin{gather*}
    \pi\colon G\times_M  B_{\omega_-}/(B_{\omega_-}\cap K)\to G/M,\quad [g,\bar{u}]\mapsto gM.
\end{gather*}

\begin{proposition}\label{prop:coordinate_change}
    For each $\varphi\in C_c^{\mathrm{lc}}(\widetilde{{\mf P}}^{(2)})$ and each $n\in\mathbb{N}$ we have
\begin{gather*}
    \mathrm{I}_c(\varphi,n)=T(\mathbbm{1}_{S_{n}}\varphi)=\int_{B_{\omega_{-},n}/B_{\omega_{-},n}\cap K}T(\psi_{\varphi,\bar{u}})z^{\langle\bar{u}o,\omega_{+}\rangle}\intd\bar{u},
\end{gather*}
where $\psi_{\varphi,\bar{u}}$ denotes the locally constant function
\begin{gather*}
    \psi_{\varphi,\bar{u}}(x,\omega_{1},\omega_{2})\coloneqq \mathbbm{1}_{\widetilde{{\mf P}}}(x,\omega_{1},\omega_{2})\int_M\varphi(\pi(\widetilde{\mathcal{A}}(x,\omega_{1},\omega_{2}))m\bar{u}o,\omega_{1},\omega_{2})\intd m.
\end{gather*}
\end{proposition}

\begin{proof}
Let $\{\chi_{j}^{\pm}\}_{j}\subseteq C^{\mathrm{lc}}(\Omega)$ be sequences approximating $\mu_{\pm}$ with respect to the Haar measure $\intd k$ as in Lemma~\ref{la:meas_approx}. Then we have, for $j,\ell\to\infty$,
\begin{gather*}
   \sum_{x\in\widetilde{\mathfrak{X}}}\quad\iint\limits_{(K/K\cap B_{\omega_{+}})^2}\mathbbm{1}_{S_{n}}\varphi(x,k_{-}\omega_{+},k_{+}\omega_{+})z^{\langle x,k_{-}\omega_{+}\rangle}z^{\langle x,k_{+}\omega_{+}\rangle}\chi_{j}^{+}(k_{+}\omega_{+})\intd k_{+}\chi_{\ell}^{-}(k_{-}\omega_{+})\intd k_{-}\\
\to T(\mathbbm{1}_{S_{n}}\varphi)=\mathrm{I}_c(\varphi,n).
\end{gather*}
Let $z=q^{s+1}$ and fix $j,\ell$. We use Lemma \ref{la:pushforward_meas}, \ref{la:cutoff_pushforward} and \ref{la:transform_function} to rewrite the left hand side as
\begin{align*}
&\int_{\widetilde{{\mf P}}^{(2)}}\mathbbm{1}_{S_{n}}\varphi(x,k_{-}\omega_{+},k_{+}\omega_{+})q^{s\langle x,k_{-}\omega_{+}\rangle}q^{s\langle x,k_{+}\omega_{+}\rangle}\chi_{j}^{+}(k_{+}\omega_{+})\chi_{\ell}^{-}(k_{-}\omega_{+})\intd\mu_{\widetilde{{\mf P}}^{(2)}}\\
&=c\int_{G}\int_{B_{\omega_{-}}/B_{\omega_{-}}\cap K}((\mathbbm{1}_{S_{n}}\varphi)\circ\widetilde{\mathcal{A}}^{-1})(f_{\chi_{\ell}^{-},\chi_{j}^{+},s}\circ\widetilde{\mathcal{A}}^{-1})\cdot q^{\langle\bar{u}o,\omega_+\rangle}\intd\bar{u}\intd g\\
&=c\int_{G}\int_{B_{\omega_{-},n}/B_{\omega_{-},n}\cap K}(\varphi\circ\widetilde{\mathcal{A}}^{-1})(f_{\chi_{\ell}^{-},\chi_{j}^{+},s}\circ\widetilde{\mathcal{A}}^{-1})\cdot q^{\langle\bar{u}o,\omega_+\rangle}\intd\bar{u}\intd g\\
&=c\int_{G}\int_{B_{\omega_{-},n}/B_{\omega_{-},n}\cap K}(\varphi\circ\widetilde{\mathcal{A}}^{-1})([g,\bar{u}])(f_{\chi_{\ell}^{-},\chi_{j}^{+},s}\circ\widetilde{\mathcal{A}}^{-1})([g,e])\cdot z^{\langle\bar{u}o,\omega_+\rangle}\intd\bar{u}\intd g\\
&=c\int_{B_{\omega_{-},n}/B_{\omega_{-},n}\cap K}z^{\langle\bar{u}o,\omega_+\rangle}\int_{G}(\varphi\circ\widetilde{\mathcal{A}}^{-1})([g,\bar{u}])(f_{\chi_{\ell}^{-},\chi_{j}^{+},s}\circ\widetilde{\mathcal{A}}^{-1})([g,e])\intd g\intd\bar{u}\\
&=c\int_{B_{\omega_{-},n}/B_{\omega_{-},n}\cap K}z^{\langle\bar{u}o,\omega_+\rangle}\int_{G/M}\int_M(\varphi\circ\widetilde{\mathcal{A}}^{-1})([gm,\bar{u}])\intd m\,(f_{\chi_{\ell}^{-},\chi_{j}^{+},s}\circ\widetilde{\mathcal{A}}^{-1})([g,e])\intd gM\intd\bar{u}.
\end{align*}
In this expression we now apply the second integral formula from Lemma \ref{la:pushforward_meas} to
\begin{gather*}
    \Phi(gM)\coloneqq\int_M(\varphi\circ\widetilde{\mathcal{A}}^{-1})([gm,\bar{u}])\intd m\cdot (f_{\chi_{\ell}^{-},\chi_{j}^{+},s}\circ\widetilde{\mathcal{A}}^{-1})([g,e])
\end{gather*}
and claim that we end up with
\begin{gather}\label{eq:proof_Ic}
\int_{B_{\omega_{-},n}/B_{\omega_{-},n}\cap K}z^{\langle\bar{u}o,\omega_+\rangle}\int_{\widetilde{{\mf P}}^{(2)}}(\psi_{\varphi,\bar{u}}\cdot f_{\chi_{\ell}^{-},\chi_{j}^{+},s})(x,k_{-}\omega_{+},k_{+}\omega_{+})\intd\mu_{\widetilde{{\mf P}}^{(2)}}\intd\bar{u}.
\end{gather}
Indeed, we obtain the following identity for $\psi\coloneqq\psi_{\varphi,\bar{u}}\cdot f_{\chi_{\ell}^{-},\chi_{j}^{+},s}$:
\begin{align*}
    \psi(go,g\omega_{-},g\omega_{+})&=\int_M\varphi(gm\bar{u}o,g\omega_{-},g\omega_{+})\intd m\cdot f_{\chi_{\ell}^{-},\chi_{j}^{+},s}(go,g\omega_{-},g\omega_{+})\\
&=\int_M(\varphi\circ\widetilde{\mathcal{A}}^{-1})([gm,\bar{u}])\intd m\cdot (f_{\chi_{\ell}^{-},\chi_{j}^{+},s}\circ\widetilde{\mathcal{A}}^{-1})([g,e])=\Phi(gM).
\end{align*}
Now note that, for $j,\ell\to\infty$, we have
\begin{gather*}
\int_{\widetilde{{\mf P}}^{(2)}}(\psi_{\varphi,\bar{u}}\cdot f_{\chi_{\ell}^{-},\chi_{j}^{+},s})(x,k_{-}\omega_{+},k_{+}\omega_{+})\intd\mu_{\widetilde{{\mf P}}^{(2)}}\to T(\psi_{\varphi,\bar{u}})
\end{gather*}
so that taking the limits in the integral from \eqref{eq:proof_Ic} yields the desired result.
\end{proof}

\begin{proposition}\label{prop:Ic_Gamma}
    \begin{gather*}
        \mathrm{I}_c^{\Gamma}(\mathbbm{1},n)=T_{\Gamma}(\mathbbm{1}_{\widetilde{{\mf P}}})\int_{B_{\omega_{-},n}/B_{\omega_{-},n}\cap K}z^{\langle\bar{u}o,\omega_{+}\rangle}\intd\bar{u}.
    \end{gather*}
\end{proposition}

\begin{proof}
Let $\varphi\in C_c^{\mathrm{lc}}(\widetilde{{\mf P}}^{(2)})$ be such that $\sum_{\gamma\in\Gamma}\gamma\varphi=\mathbbm{1}$. Then, for the function $\psi_{\varphi,\bar{u}}$ from Proposition \ref{prop:coordinate_change} we have $\sum_{\gamma\in\Gamma}\gamma\psi_{\varphi,\bar{u}}=\mathbbm{1}_{\widetilde{{\mf P}}}$ since
\begin{align*}
    &\hphantom{{}={}}\sum_{\gamma\in\Gamma}\int_M\varphi(\pi(\widetilde{\mathcal{A}}(\gamma x,\gamma\omega_{1},\gamma\omega_{2}))m\bar{u}o,\gamma\omega_{1},\gamma\omega_{2})\intd m\\
&=\sum_{\gamma\in\Gamma}\int_M\varphi(\gamma\pi(\widetilde{\mathcal{A}}(x,\omega_{1},\omega_{2}))m\bar{u}o,\gamma\omega_{1},\gamma\omega_{2})\intd m=\int_M\mathbbm{1}\intd m=1.\qedhere
\end{align*}
\end{proof}

The right hand side in Proposition \ref{prop:Ic_Gamma} is closely related to the geodesic pairing of $u_+$ and $u_-$. In fact, Lemma~\ref{lem:TGamma-product pairing} immediately implies the following corollary of Proposition \ref{prop:Ic_Gamma}:

\begin{corollary}\label{cor:Ic_Gamma}
   \begin{gather*}
    \mathrm{I}_c^{\Gamma}(\mathbbm{1},n)=\langle u_+,u_-\rangle_\mathrm{geod}\cdot\int_{B_{\omega_{-},n}/(B_{\omega_{-},n}\cap K)}z^{\langle\bar{u}o,\omega_{+}\rangle}\intd\bar{u}.
   \end{gather*}
\end{corollary}

Finally, we can also explicitly compute the integral in Corollary \ref{cor:Ic_Gamma}:

\begin{lemma}\label{la:N_bar_int}
For each $n\in\mathbb{N}$ we have
\begin{gather*}
    \int_{B_{\omega_{-},n}/(B_{\omega_{-},n}\cap K)}z^{\langle\bar{u}o,\omega_{+}\rangle}\intd\bar{u}=\int_{B_{\omega_-,n}}z^{\langle\bar{u}o,\omega_+\rangle}\intd\bar{u}=(1-z^{-2})\sum_{j=0}^{n-1}\left(\frac{q}{z^2}\right)^j+\left(\frac{q}{z^2}\right)^n,
\end{gather*}
where $\intd\bar{u}$ is normalized such that $\intd\bar{u}(\mathrm{Stab}_{B_{\omega_-}}(x))=q^{\langle x,\omega_-\rangle}$ (see Appendix \ref{subsec:int_formulas}).

\noindent In particular, we obtain
\begin{gather*}
\int_{B_{\omega_-}}z^{\langle\bar{u}o,\omega_+\rangle}\intd\bar{u}=\frac{z^2-1}{z^2-q}
\end{gather*}
for $\abs{z}>\sqrt{q}$.
\end{lemma}

\begin{proof}
Recall that $]\omega_-,\omega_+[=(\ldots,x_{-1},x_0,x_1,\ldots)$ with $x_0=o$ and $B_{\omega_-,n}=\mathrm{Stab}_{B_{\omega_-}}(x_{-n})$. Then we may write $B_{\omega_-}=\bigcup_{n\in\N_0}B_{\omega_-,n}$. Since each $\bar{u}\in B_{\omega_-,n}$ fixes the whole one-sided geodesic $(\omega_-,x_{-n}]=(\ldots,x_{-n-1},x_{-n})$ we obtain $B_{\omega_-,n}\subseteq B_{\omega_-,n+1}$ for each $n\in\N_0$. We now claim that
\begin{gather*}
    \int_{B_{\omega_-,n}}z^{\langle\bar{u}o,\omega_+\rangle}\intd\bar{u}=(1-z^{-2})\sum_{j=0}^{n-1}\left(\frac{q}{z^2}\right)^j+\left(\frac{q}{z^2}\right)^n.
\end{gather*}
Indeed, for $n=0$ we have
\begin{gather*}
    \int_{B_{\omega_-,0}}z^{\langle\bar{u}o,\omega_+\rangle}\intd\bar{u}=\int_{B_{\omega_-,0}}1\intd\bar{u}=q^{\langle o,\omega_-\rangle}=1=\left(\frac{q}{z^2}\right)^0.
\end{gather*}
For the induction step we write $B_{\omega_-,n}=B_{\omega_-,n}\setminus B_{\omega_-,n-1}\cup B_{\omega_-,n-1}$. Now note that for each $\bar{u}\in B_{\omega_-,n}\setminus B_{\omega_-,n-1}$ we have (see Figure \ref{fig:horo_distance})
\begin{gather*}
    \langle\bar{u}o,\omega_+\rangle=d(o,o)-d(\bar{u}o,o)=-d(\bar{u}o,o)=-2n.
\end{gather*}

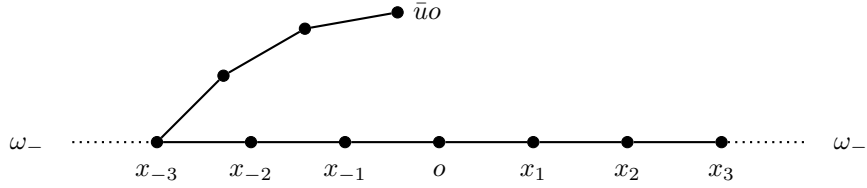
\begin{figure}
\tikzstyle{circle}=[shape=circle,draw,inner sep=1.5pt]
\begin{tikzpicture}[scale=2.5]
\node (0) at (0,0) [label=left:$\omega_-$]{};
\draw (0.5,0) node[circle,fill=black,label=below:$x_{-3}$] (1) {};
\draw[dotted,thick] (0) to (1);
\path (1) ++(0:0.5) node (2) [circle,inner sep=1.5pt, fill=black,label=below:$x_{-2}$] {};
\path (2) ++(0:0.5) node (3) [circle,inner sep=1.5pt, fill=black,label=below:$x_{-1}$] {};
\path (3) ++(0:0.5) node (4) [circle,inner sep=1.5pt, fill=black,label=below:$o$] {};
\path (4) ++(0:0.5) node (5) [circle,inner sep=1.5pt, fill=black,label=below:$x_{1}$] {};
\path (5) ++(0:0.5) node (6) [circle,inner sep=1.5pt, fill=black,label=below:$x_{2}$] {};
\path (6) ++(0:0.5) node (7) [circle,inner sep=1.5pt, fill=black,label=below:$x_{3}$] {};
\draw[-,thick] (1) to (2);
\draw[-,thick] (2) to (3);
\draw[-,thick] (3) to (4);
\draw[-,thick] (4) to (5);
\draw[-,thick] (5) to (6);
\draw[-,thick] (6) to (7);
\path (1) ++(45:0.5) node (9) [circle,inner sep=1.5pt, fill=black] {};
\path (9) ++(30:0.5) node (10) [circle,inner sep=1.5pt, fill=black] {};
\path (10) ++(10:0.5) node (11) [circle,inner sep=1.5pt, fill=black,label=right:$\bar{u}o$] {};
\draw[-,thick] (1) to (9);
\draw[-,thick] (9) to (10);
\draw[-,thick] (10) to (11);
%

\node (8) at (4,0) [label=right:$\omega_-$]{};
\draw[dotted,thick] (7) to (8);
\end{tikzpicture}
\caption{Action of $\bar{u}\in B_{\omega_-,3}\setminus B_{\omega_-,2}$ on the reference geodesic}
\label{fig:horo_distance}
\end{figure}

Thus, using the induction hypothesis,
\begin{align*}
    \int_{B_{\omega_-,n}}z^{\langle\bar{u}o,\omega_+\rangle}\intd\bar{u}&=\int_{B_{\omega_-,n}\setminus B_{\omega_-,n-1}}z^{-2n}\intd\bar{u}+(1-z^{-2})\sum_{j=0}^{n-2}\left(\frac{q}{z^2}\right)^j+\left(\frac{q}{z^2}\right)^{n-1}\\
&=(q^{\langle x_{-n},\omega_-\rangle}-q^{\langle x_{-n+1},\omega_-\rangle})z^{-2n}+(1-z^{-2})\sum_{j=0}^{n-2}\left(\frac{q}{z^2}\right)^j+\left(\frac{q}{z^2}\right)^{n-1}\\
&=(q^{n}-q^{n-1})z^{-2n}+(1-z^{-2})\sum_{j=0}^{n-2}\left(\frac{q}{z^2}\right)^j+\left(\frac{q}{z^2}\right)^{n-1}\\
&=\left(\frac{q}{z^2}\right)^n-z^{-2}\left(\frac{q}{z^2}\right)^{n-1}+(1-z^{-2})\sum_{j=0}^{n-2}\left(\frac{q}{z^2}\right)^j+\left(\frac{q}{z^2}\right)^{n-1}\\
&=(1-z^{-2})\sum_{j=0}^{n-1}\left(\frac{q}{z^2}\right)^j+\left(\frac{q}{z^2}\right)^n.
\end{align*}
For $\abs{z}>\sqrt{q}$ this sequence converges to
\begin{gather*}
    \int_{B_{\omega_-}}z^{\langle\bar{u}o,\omega_+\rangle}\intd\bar{u}=(1-z^{-2})\frac{1}{1-\frac{q}{z^2}}=\frac{z^2-1}{z^2-q}.\qedhere
\end{gather*}
\end{proof}

We conclude this subsection with the final expression we obtained for $\mathrm{I}_c^{\Gamma}(\mathbbm{1},n)$.
\begin{proposition}\label{prop:pairing_first_part}
    \begin{gather*}
        \mathrm{I}_c^{\Gamma}(\mathbbm{1},n)=\langle u_+,u_-\rangle_\mathrm{geod}\cdot\left((1-z^{-2})\sum_{j=0}^{n-1}\left(\frac{q}{z^2}\right)^j+\left(\frac{q}{z^2}\right)^n\right).
    \end{gather*}
\end{proposition}

\subsection{The integral \texorpdfstring{$\mathrm{I}_r^{\Gamma}(\mathbbm{1},n)$}{Ir(1,n)}}
The goal of this section is to compute $\mathrm{I}_r^{\Gamma}(\mathbbm{1},n)$ and relate it to $\langle u_+,u_-\rangle_\mathrm{geod}$. For this we first prove that $\mathrm{I}_r^{\Gamma}(\mathbbm{1},n)=\left(\frac{q}{z^2}\right)^{n}\langle u_+,u_-\rangle_{(\mathrm{op},\mf E)}$ and then use the pairing formula from Equation~\eqref{eq:pairing_right_edge} to establish the connection to $\langle u_+,u_-\rangle_\mathrm{geod}$.

We start by relating the sets $S_{n}^c$ to paths. For this, let $\mathcal{P}_{n}$ be the set of all finite edge chains $\mathbf{p}=(\vec{e}_{0},\ldots,\vec{e}_{n})$ of length $n+1$ in $\widetilde{\mathfrak{G}}$. Moreover, we denote the vertices on $\mathbf{p}\in\mathcal{P}_{n}$ by $p_0\coloneqq\iota(\vec{e}_{0}),\ldots,p_{n}\coloneqq\iota(\vec{e}_{n})$ and $p_{n+1}\coloneqq\tau(\vec{e}_{n})$. Now note that
\begin{gather}\label{eq:Snc_paths}
    (x,\omega_{1},\omega_{2})\in S_{n}^{c}\Leftrightarrow \exists!\,\mathbf{p}\in\mathcal{P}_{n}\colon p_0=x\text{ and }\omega_{1},\omega_{2}\in\partial_{+}(p_{n},p_{n+1})
\end{gather}
so that $\mathbbm{1}_{S_{n}^{c}}=\sum_{\mathbf{p}\in\mathcal{P}_{n}}\mathbbm{1}_{\mathbf{p}}$ with
\begin{gather*}
    \mathbbm{1}_{\mathbf{p}}(x,\omega_{1},\omega_{2})\coloneqq
\begin{cases}
    1\quad&\colon p_0=x\text{ and }\omega_{1},\omega_{2}\in\partial_{+}(p_{n},p_{n+1})\\
    0\quad&\colon\text{else}
\end{cases}.
\end{gather*}

This observation allows us to rewrite $\mathrm{I}_r(\varphi,n)$, for $\varphi\in C_c^{\mathrm{lc}}(\widetilde{{\mf P}}^{(2)})$, in terms of sums over paths in the following way.

\begin{lemma}\label{la:Ir_first_form}For each $\varphi\in C_c^{\mathrm{lc}}(\widetilde{{\mf P}}^{(2)})$ we have
\begin{gather*}
    \mathrm{I}_r(\varphi,n)=z^{-2n}\sum_{\mathbf{p}\in\mathcal{P}_{n}}\int_{\partial_{+}(p_{n},p_{n+1})}\int_{\partial_{+}(p_{n},p_{n+1})}\varphi(p_{0},\omega_{1},\omega_{2})z^{\langle p_{n},\omega_{1}\rangle}z^{\langle p_{n},\omega_{2}\rangle}\intd\mu_{-}(\omega_{1})\intd\mu_{+}(\omega_{2}).
\end{gather*}
\end{lemma}

\begin{proof}
By definition we have $\mathrm{I}_r(\varphi,n)=T(\mathbbm{1}_{S_{n}^{c}}\varphi)$. But Equation \eqref{eq:Snc_paths} implies that
\begin{align*}
T(\mathbbm{1}_{S_{n}^{c}}\varphi)&=\sum_{x\in\widetilde{\mathfrak{X}}}\int_{\Omega}\int_{\Omega}(\mathbbm{1}_{S_{n}^{c}}\varphi)(x,\omega_{1},\omega_{2})z^{\langle x,\omega_{1}\rangle}z^{\langle x,\omega_{2}\rangle}\intd \mu_{-}(\omega_{1})\intd \mu_{+}(\omega_{2})\\
&=\sum_{\mathbf{p}\in\mathcal{P}_{n}}\sum_{x\in\widetilde{\mathfrak{X}}}\int_{\Omega}\int_{\Omega}(\mathbbm{1}_{\mathbf{p}}\varphi)(x,\omega_{1},\omega_{2})z^{\langle x,\omega_{1}\rangle}z^{\langle x,\omega_{2}\rangle}\intd \mu_{-}(\omega_{1})\intd \mu_{+}(\omega_{2})\\
&=\sum_{\mathbf{p}\in\mathcal{P}_{n}}\int_{\Omega}\int_{\Omega}(\mathbbm{1}_{\mathbf{p}}\varphi)(p_{0},\omega_{1},\omega_{2})z^{\langle p_{0},\omega_{1}\rangle}z^{\langle p_{0},\omega_{2}\rangle}\intd \mu_{-}(\omega_{1})\intd \mu_{+}(\omega_{2})\\
&=\sum_{\mathbf{p}\in\mathcal{P}_{n}}\int_{\partial_{+}(p_{n},p_{n+1})}\int_{\partial_{+}(p_{n},p_{n+1})}\varphi(p_{0},\omega_{1},\omega_{2})z^{\langle p_{0},\omega_{1}\rangle}z^{\langle p_{0},\omega_{2}\rangle}\intd \mu_{-}(\omega_{1})\intd \mu_{+}(\omega_{2}).
\end{align*}
We finally claim that $\langle p_{0},\omega\rangle=\langle p_{n},\omega\rangle-n$, for each $\mathbf{p}\in\mathcal{P}_{n}$ and $\omega\in\partial_{+}(p_{n},p_{n+1})$, which in turn finishes the proof. Indeed, we calculate that for $\xi\in]o,\omega[\cap]p_{n},\omega[$
\begin{gather*}
    \langle p_{0},\omega\rangle=d(o,\xi)-d(p_{0},\xi)=d(o,\xi)-(d(p_{0},p_{n})+d(p_{n},\xi))=\langle p_{n},\omega\rangle-d(p_{0},p_{n}).\qedhere
\end{gather*}
\end{proof}

Lemma \ref{la:Ir_first_form} immediately implies the following identity for $\mathrm{I}_r^{\Gamma}(\mathbbm{1},n)$.

\begin{lemma}\label{la:Ir_second_form}
\begin{gather*}
\mathrm{I}_r^{\Gamma}(\mathbbm{1},n)=z^{-2n}\sum_{\Gamma\mathbf{p}\in\Gamma\backslash\mathcal{P}_{n}}\int_{\partial_{+}(p_{n},p_{n+1})}z^{\langle p_{n},\omega_{1}\rangle}\intd\mu_{-}(\omega_{1})\cdot\int_{\partial_{+}(p_{n},p_{n+1})}z^{\langle p_{n},\omega_{2}\rangle}\intd\mu_{+}(\omega_{2}).
\end{gather*}
\end{lemma}

\begin{proof}
Recall that $\mathrm{I}_r^{\Gamma}(\mathbbm{1},n)=T_{\Gamma}(\mathbbm{1}_{S_{n}^{c}})$. Now let $\mathcal{F}$ denote a fundamental domain for the $\Gamma$-action on $\widetilde{\mathfrak{X}}$ and observe that $\mathrm{I}_r(\mathbbm{1}_{\mathcal{F}},n)=T(\mathbbm{1}_{S_{n}^{c}}\mathbbm{1}_{\mathcal{F}})=T_{\Gamma}(\mathbbm{1}_{S_{n}^{c}})$. Thus, Lemma \ref{la:Ir_first_form} implies
\begin{align*}
    \mathrm{I}_r^{\Gamma}(\mathbbm{1},n)&=z^{-2n}\sum_{\mathbf{p}\in\mathcal{P}_{n}}\mathbbm{1}_{\mathcal{F}}(p_{0})\int_{\partial_{+}(p_{n},p_{n+1})}\int_{\partial_{+}(p_{n},p_{n+1})}z^{\langle p_{n},\omega_{1}\rangle}z^{\langle p_{n},\omega_{2}\rangle}\intd\mu_{-}(\omega_{1})\intd\mu_{+}(\omega_{2})\\
&=z^{-2n}\sum_{\Gamma\mathbf{p}\in\Gamma\backslash\mathcal{P}_{n}}\int_{\partial_{+}(p_{n},p_{n+1})}\int_{\partial_{+}(p_{n},p_{n+1})}z^{\langle p_{n},\omega_{1}\rangle}z^{\langle p_{n},\omega_{2}\rangle}\intd\mu_{-}(\omega_{1})\intd\mu_{+}(\omega_{2}),
\end{align*}
where we obtain the $\Gamma$-invariance in $\mathbf{p}$ by the same reasoning as in Lemma \ref{la:T_Gamma_inv}.
\end{proof}

We now observe that the summand in Lemma \ref{la:Ir_second_form} only depends on the last two vertices $(p_{n},p_{n+1})$ on the path $\mathbf{p}\in\mathcal{P}_{n}$. This leads to the following simplification.

\begin{lemma}\label{la:Ir_third_form}
\begin{gather*}
\mathrm{I}_r^{\Gamma}(\mathbbm{1},n)=\left(\frac{q}{z^2}\right)^{n}\sum_{\Gamma\vec{e}\in\Gamma\backslash\widetilde{\mathfrak{E}}}\int_{\partial_{+}\vec{e}}z^{\langle \iota(\vec{e}),\omega_{1}\rangle}\intd\mu_{-}(\omega_{1})\cdot\int_{\partial_{+}\vec{e}}z^{\langle \iota(\vec{e}),\omega_{2}\rangle}\intd\mu_{+}(\omega_{2}).
\end{gather*}
\end{lemma}

\begin{proof}
Let $\mathcal{F}$ denote a fundamental domain for the $\Gamma$-action on $\widetilde{\mathfrak{X}}$. Then, by Lemma \ref{la:Ir_second_form},
\begin{align*}
\mathrm{I}_r^{\Gamma}(\mathbbm{1},n)&=z^{-2n}\sum_{\Gamma\mathbf{p}\in\Gamma\backslash\mathcal{P}_{n}}\int_{\partial_{+}(p_{n},p_{n+1})}z^{\langle p_{n},\omega_{1}\rangle}\intd\mu_{-}(\omega_{1})\cdot\int_{\partial_{+}(p_{n},p_{n+1})}z^{\langle p_{n},\omega_{2}\rangle}\intd\mu_{+}(\omega_{2})\\
&=z^{-2n}\sum_{\mathbf{p}\in\mathcal{P}_{n}}\mathbbm{1}_{\mathcal{F}}(p_{n})\int_{\partial_{+}(p_{n},p_{n+1})}z^{\langle p_{n},\omega_{1}\rangle}\intd\mu_{-}(\omega_{1})\cdot\int_{\partial_{+}(p_{n},p_{n+1})}z^{\langle p_{n},\omega_{2}\rangle}\intd\mu_{+}(\omega_{2}).
\end{align*}
Now note that, for a fixed edge $\vec{e}\in\widetilde{\mathfrak{E}}$, there exist $q^{n}$ paths $\mathbf{p}\in\mathcal{P}_{n}$ with $(p_{n},p_{n+1})=\vec{e}$. Therefore, $\mathrm{I}_r^{\Gamma}(\mathbbm{1},n)$ equals
\begin{align*}
&z^{-2n}\sum_{\vec{e}\in\widetilde{\mathfrak{E}}}\sum_{\substack{\mathbf{p}\in\mathcal{P}_{n}\\ (p_{n},p_{n+1})=\vec{e}}}\mathbbm{1}_{\mathcal{F}}(p_{n})\int_{\partial_{+}(p_{n},p_{n+1})}z^{\langle p_{n},\omega_{1}\rangle}\intd\mu_{-}(\omega_{1})\cdot\int_{\partial_{+}(p_{n},p_{n+1})}z^{\langle p_{n},\omega_{2}\rangle}\intd\mu_{+}(\omega_{2})\\
&=z^{-2n}q^{n}\sum_{\vec{e}\in\widetilde{\mathfrak{E}}}\mathbbm{1}_{\mathcal{F}}(\iota(\vec{e}))\int_{\partial_{+}\vec{e}}z^{\langle \iota(\vec{e}),\omega_{1}\rangle}\intd\mu_{-}(\omega_{1})\cdot\int_{\partial_{+}\vec{e}}z^{\langle \iota(\vec{e}),\omega_{2}\rangle}\intd\mu_{+}(\omega_{2})\\
&=\left(\frac{q}{z^2}\right)^{n}\sum_{\Gamma\vec{e}\in\Gamma\backslash\widetilde{\mathfrak{E}}}\int_{\partial_{+}\vec{e}}z^{\langle \iota(\vec{e}),\omega_{1}\rangle}\intd\mu_{-}(\omega_{1})\cdot\int_{\partial_{+}\vec{e}}z^{\langle \iota(\vec{e}),\omega_{2}\rangle}\intd\mu_{+}(\omega_{2}).\qedhere
\end{align*}
\end{proof}

Recall from \cite[Def.\@ 2.1]{AFH23} that the function
\begin{gather*}
    \widetilde{\mathfrak{E}}\to\mathbb{C},\quad \vec{e} \mapsto \int_{\partial_{+}\vec{e}}z^{\langle\iota(\vec{e}),\omega\rangle}\intd\mu_{+}(\omega)
\end{gather*}
is called the \emph{edge Poisson transform} $\mathcal{P}_{z}^{\mathrm{e}}(\mu_{+})$ of $\mu_{+}$. Moreover, by \cite[Prop.\@~2.11]{AFH23},
\begin{gather*}
    \mathcal{P}_{z}^{\mathrm{e}}(\mu_{+})=(\pi^{\widetilde{\mf E}}_{+,*}\circ p_{z,+}B_{+}^{\star})(\mu_{+})=\pi^{\widetilde{\mf E}}_{+,*}(\tilde u_{+}).
\end{gather*}
Similarly, we obtain
\begin{gather*}
    \int_{\partial_{-}\vec{e}}z^{\langle\tau(\vec{e}),\omega\rangle}\intd\mu_{-}(\omega)=\pi^{\widetilde{\mf E}}_{-,*}(\tilde u_{-})(\vec{e})
\end{gather*}
so that Lemma \ref{la:Ir_third_form} implies
\begin{gather*}
    \mathrm{I}_r^{\Gamma}(\mathbbm{1},n)=\left(\frac{q}{z^{2}}\right)^{n}\sum_{\Gamma\vec{e}\in\Gamma\backslash\widetilde{\mathfrak{E}}}\pi^{\widetilde{\mf E}}_{+,*}(\tilde u_{+})(\vec{e})\pi^{\widetilde{\mf E}}_{-,*}(\tilde u_{-})(\eop{e}).
\end{gather*}
But now Lemma \ref{la:pi_E_star_upstairs_downstairs} yields
\begin{proposition}\label{prop:pairing_second_part}
\begin{gather*}
    \mathrm{I}_r^{\Gamma}(\mathbbm{1},n)=\left(\frac{q}{z^{2}}\right)^{n}\sum_{\Gamma\vec{e}\in\Gamma\backslash\widetilde{\mathfrak{E}}}\pi^{\mathfrak{E}}_{+,*}u_+(\pi_{\mathfrak{E}}(\vec{e}))\pi^{\mathfrak{E}}_{-,*}u_-(\pi_{\mathfrak{E}}(\eop{e}))=\left(\frac{q}{z^{2}}\right)^{n}\langle u_+,u_-\rangle_{(\mathrm{op},\mf E)}.
\end{gather*}
\end{proposition}

\subsection{The pairing formula}
We can now combine Propositions \ref{prop:pairing_first_part} and \ref{prop:pairing_second_part} to prove the pairing formula.
\begin{theorem}\label{thm:pairing formula}
Let $0\neq z\in\mathbb{C}$ and $u_\pm \in\mathcal E_z({\mathcal L}_\pm';\mc D'({\mf P}^\pm))$ be (co)resonant states on the finite graph $\mathfrak{G}$. Then
\begin{gather*}
    (z^{2}-q)\langle u_+,u_-\rangle_{(\mathfrak{X})}=(z^{2}-1)\langle u_+,u_-\rangle_\mathrm{geod}.
\end{gather*}
\end{theorem}
\begin{proof}
By Equation \eqref{eq:decomp_left} we have $\langle u_+,u_-\rangle_{(\mathfrak{X})}=\mathrm{I}_c^{\Gamma}(\mathbbm{1},n)+\mathrm{I}_r^{\Gamma}(\mathbbm{1},n)$. Thus, Proposition~\ref{prop:pairing_first_part} and Proposition~\ref{prop:pairing_second_part} imply
\begin{align*}
\langle u_+,u_-\rangle_{(\mathfrak{X})}&=\mathrm{I}_c^{\Gamma}(\mathbbm{1},n)+\mathrm{I}_r^{\Gamma}(\mathbbm{1},n)\\
&=\langle u_+,u_-\rangle_\mathrm{geod}\cdot\left((1-z^{-2})\sum_{j=0}^{n-1}\left(\frac{q}{z^2}\right)^j+\left(\frac{q}{z^2}\right)^n\right)+\\
&\qquad+\left(\frac{q}{z^{2}}\right)^{n}\langle u_+,u_-\rangle_{(\mathrm{op},\mf E)}
\end{align*}
so that Equation \eqref{eq:pairing_right_edge} implies
\begin{align*}
 \langle u_+,u_-\rangle_{(\mathfrak{X})}&=\langle u_+,u_-\rangle_\mathrm{geod}\cdot\left((1-z^{-2})\sum_{j=0}^{n-1}\left(\frac{q}{z^2}\right)^j+\left(\frac{q}{z^2}\right)^n\right)\\
&\qquad+\left(\frac{q}{z^{2}}\right)^{n}(\langle u_+,u_-\rangle_{(\mathfrak{X})}-\langle u_+,u_-\rangle_\mathrm{geod})\\
&=\langle u_+,u_-\rangle_\mathrm{geod}\cdot\left((1-z^{-2})\sum_{j=0}^{n-1}\left(\frac{q}{z^2}\right)^j\right)+\left(\frac{q}{z^{2}}\right)^{n}\langle u_+,u_-\rangle_{(\mathfrak{X})}.
\end{align*}
For $z^{2}=q$ we infer
\begin{gather*}
  0=\langle u_+,u_-\rangle_\mathrm{geod}\cdot(1-z^{-2})n,
\end{gather*}
so that the theorem holds for this case. On the other hand, for $q\neq z^{2}$, we obtain
\begin{align*}
  \left(1-\left(\frac{q}{z^{2}}\right)^{n}\right)\langle u_+,u_-\rangle_{(\mathfrak{X})}=\langle u_+,u_-\rangle_\mathrm{geod}\cdot (1-z^{-2})\frac{1-\left(\frac{q}{z^2}\right)^n}{1-\frac{q}{z^2}},
\end{align*}
which implies the theorem after multiplying by $\frac{z^2-q}{1-\left(\frac{q}{z^2}\right)^{n}}$.
\end{proof}

\begin{remark}
\begin{enumerate}
\item[(i)] The factor in Theorem \ref{thm:pairing formula} is closely related to the $\mathrm{c}$-function: Indeed, by \cite[Prop.\@~2.4]{FTN91}, this function, in dependence of the principal series parameter $s$, is given by
\begin{gather*}
\frac{1}{q+1}(q^{1-s}-q^{s-1})(q^{-s}-q^{s-1})^{-1}.
\end{gather*}
In our parametrization we have $z=q^{-s+\frac{1}{2}}$ and thus, for $z^{2}\not\in\{0,1\}$,
\begin{gather*}
        \mathrm{c}(z)=\frac{1}{q+1}(q^{\frac{1}{2}}z-q^{-\frac{1}{2}}z^{-1})(q^{-\frac{1}{2}}z-q^{-\frac{1}{2}}z^{-1})^{-1}=\frac{1}{q+1}\frac{qz-z^{-1}}{z-z^{-1}}=\frac{1}{q+1}\frac{z^{-2}-q}{z^{-2}-1}.
\end{gather*}
Using this expression, Theorem \ref{thm:pairing formula} takes the form
\begin{gather*}
    (q+1)\mathrm{c}(z^{-1})\langle u_+,u_-\rangle_{(\mathfrak{X})}=\langle u_+,u_-\rangle_\mathrm{geod}.
\end{gather*}
\item[(ii)] Note that, by Proposition \ref{prop:pairing_first_part}, Lemma \ref{la:N_bar_int} and Theorem \ref{thm:pairing formula} we have
\begin{gather*}
    \lim_{n\to\infty}\mathrm{I}_c^{\Gamma}(\mathbbm{1},n)=\langle u_+,u_-\rangle_\mathrm{geod} \frac{z^2-1}{z^2-q}=\langle u_+,u_-\rangle_{(\mathfrak{X})}
\end{gather*}
whenever $\abs{z}>\sqrt{q}$. Since $\langle u_+,u_-\rangle_{(\mathfrak{X})}=\mathrm{I}_c^{\Gamma}(\mathbbm{1},n)+\mathrm{I}_r^{\Gamma}(\mathbbm{1},n)$ this also follows from Proposition \ref{prop:pairing_second_part}. For $\abs{z}<\sqrt{q}$ neither $\mathrm{I}_c^{\Gamma}(\mathbbm{1},n)$ nor $\mathrm{I}_r^{\Gamma}(\mathbbm{1},n)$ converges.
\item[(iii)] Theorem \ref{thm:pairing formula} in particular implies that for $z^{2}\not\in\{0,1\}$ the geodesic pairing only depends on the base projections $\pi_{\pm,*}u_\pm $ of $u_\pm $.
\end{enumerate}
\end{remark}

\appendix

\section{Integral formulas and computations of pushforward measures}\label{sec:integralformulas}

Note that this appendix is used in the main text only from Section~\ref{sec:coverings} on. It is therefore not a problem that we use results from the earlier sections here.

\subsection{Integral formulas}\label{subsec:int_formulas}
Decompositions of $G$ naturally lead to integral formulas for the corresponding Haar measures. In order to state them, first note that $G$ is a locally compact unimodular group (see e.g.\@ \cite[p.\@~155]{Ve02}) so that we may fix a bi-invariant measure $\intd g$ on $G$ which will be normalized later. Moreover, by \cite[La.\@~3.3]{Ve02}, $B_{\omega_+}$ (resp.\@ $B_{\omega_-}$) is unimodular and its invariant measure $\intd u$ (resp.\@ $\intd\bar{u}$) can be normalized such that $\intd u(B_{\omega_+}\cap \mathrm{Stab}_G(x))=q^{\langle x,\omega_+\rangle}$ (resp.\@ $\intd\bar{u}(B_{\omega_-}\cap \mathrm{Stab}_G(x))=q^{\langle x,\omega_-\rangle}$) for each $x\in\mathfrak{X}$. Let $s\in K$ be an element with $s^2=\mathrm{id}$ and $s\tau^js^{-1}=\tau^{-j}$ for each $j\in\mathbb{Z}$. Then the map
\begin{gather*}
    B_{\omega_-}\to B_{\omega_+},\qquad \bar{u}\mapsto s\bar{u}s^{-1}
\end{gather*}
is measure-preserving: Indeed, if $\xi\in[o,\omega_+[\cap[x,\omega_+]$ we have $s\xi\in[o,\omega_-[\cap[sx,\omega_-[$ and thus
\begin{gather*}
    \langle x,\omega_+\rangle=d(o,\xi)-d(x,\xi)=d(so,s\xi)-d(sx,s\xi)=d(o,s\xi)-d(sx,s\xi)=\langle sx,\omega_-\rangle.
\end{gather*}
Finally, the compact group $K$ is unimodular and we denote a Haar measure on it by $\intd k$.

We can now describe the integral formulas corresponding to the Iwasawa decompositions from Proposition \ref{prop:Iwasawa} and Corollary \ref{cor:KNA}.

\begin{proposition}[{See \cite[Thm.\@ 3.5]{Ve02}}]\label{prop:NAK_integral}
   The Haar measure $\intd g$ on $G$ can be normalized such that for each $f\in C_c(G)$
\begin{gather*}
    \int_Gf(g)\intd g=\int_{B_{\omega_+}}\sum_{j\in\mathbb{Z}}\int_Kf(u\tau^jk)q^{-j}\intd k\intd u.
\end{gather*}
\end{proposition}

\begin{proposition}\label{prop:KNA_integral}
   For each $f\in C_c(G)$ we have
\begin{gather*}
    \int_Gf(g)\intd g=\int_K\int_{B_{\omega_+}}\sum_{j\in\mathbb{Z}}f(ku\tau^j)\intd u\intd k.
\end{gather*}
\end{proposition}

\begin{proof}
Since $G$ is unimodular we have $\int_G f(g)\intd g=\int_G f(g^{-1})\intd g$. Thus, Proposition \ref{prop:NAK_integral} allows us to write
\begin{gather*}
    \int_G f(g)\intd g=\int_{B_{\omega_+}}\sum_{j\in\mathbb{Z}}\int_K f(k^{-1}\tau^{-j}u^{-1})q^{-j}\intd k\intd u=\int_{B_{\omega_+}}\sum_{j\in\mathbb{Z}}\int_K f(k\tau^{j}u)q^{j}\intd k\intd u
\end{gather*}
since $K$ and $B_{\omega_+}$ are unimodular. Now \cite[La.\@~3.8]{Ve02} allows us to rewrite this integral as
\begin{gather*}
    \int_K\int_{B_{\omega_+}}\sum_{j\in\mathbb{Z}}f(k\tau^j\tau^{-j}u\tau^j)\intd u\intd k=\int_K\int_{B_{\omega_+}}\sum_{j\in\mathbb{Z}}f(ku\tau^j)\intd u\intd k.\qedhere
\end{gather*}
\end{proof}

We now normalize the Haar measure on the compact group $M$ by $\mathrm{vol}(M)=1$ and denote by $d(gM\langle\tau\rangle)$ the unique $G$-invariant measure on $G/M\langle\tau\rangle$ such that
\begin{gather}\label{eq:measure_GMt}
   \forall f\in C_c(G)\colon \int_G f(g)\intd g=\int_{G/M\langle\tau\rangle}\int_M\sum_{j\in\mathbb{Z}}f(gm\tau^j)\intd m\intd (gM\langle\tau\rangle).
\end{gather}
Then we can express the measure on $G/M\langle\tau\rangle$ in terms of the Iwasawa decomposition in the following way.

\begin{lemma}\label{la:int_G_KN}
For all $f\in C_c(G/M\langle\tau\rangle)$ we have
\begin{gather*}
    \int_{G/M\langle\tau\rangle} f(gM\langle\tau\rangle)\intd(gM\langle\tau\rangle)=\int_K\int_{B_{\omega_+}/M}f(kuM\langle\tau\rangle)\intd u\intd k.
\end{gather*}
\end{lemma}

\begin{proof}
Let $\chi\in C_c(\langle\tau\rangle)$ such that $\sum_{j\in\mathbb{Z}}\chi(\tau^j)=1$. By Corollary \ref{cor:KNA} each $g\in G$ can be written as $ku\tau^j$, where $\mathrm{pr}_{KB_{\omega_+}}(g)\coloneqq ku$ and $\mathrm{pr}_{\langle\tau\rangle}(g)\coloneqq \tau^j$ are uniquely defined. Therefore, there exists a right $M$-invariant function $\tilde{f}\in C_c(KB_{\omega_+})^M$ such that $f(gM\langle\tau\rangle)=\tilde{f}(\mathrm{pr}_{KB_{\omega_+}}(g))$ for all $g\in G$. Then we calculate
\begin{align*}
\int_{G/M\langle\tau\rangle}f(gM\langle\tau\rangle)&=\int_{G/M\langle\tau\rangle}\tilde{f}(\mathrm{pr}_{KB_{\omega_+}}(g))\intd (gM\langle\tau\rangle)\\
&=\int_{G/M\langle\tau\rangle}\int_M\sum_{j\in\mathbb{Z}}\tilde{f}(\mathrm{pr}_{KB_{\omega_+}}(g))\chi(\mathrm{pr}_{\langle\tau\rangle}(g)\tau^j)\intd m\intd (gM\langle\tau\rangle)\\
&=\int_{G/M\langle\tau\rangle}\int_M\sum_{j\in\mathbb{Z}}\tilde{f}(\mathrm{pr}_{KB_{\omega_+}}(gm\tau^j))\chi(\mathrm{pr}_{\langle\tau\rangle}(gm\tau^j))\intd m\intd (gM\langle\tau\rangle),
\end{align*}
where the last step follows from writing $g=ku\tau^{j_1}$ and observing that
\begin{gather*}
    gm\tau^j=ku\tau^{j_1}m\tau^j=ku\tau^{j_1}m\tau^{-j_1}\tau^{j_1}\tau^j
\end{gather*}
with $\tau^{j_1}m\tau^{-j_1}\in M\subseteq B_{\omega_+}$ so that $\mathrm{pr}_{KB_{\omega_+}}(gm\tau^j)\equiv\mathrm{pr}_{KB_{\omega_+}}(g)\,\mathrm{mod}\,M$ and $\mathrm{pr}_{\langle\tau\rangle}(gm\tau^j)=\mathrm{pr}_{\langle\tau\rangle}(g)\tau^j$. By \eqref{eq:measure_GMt} and Proposition \ref{prop:KNA_integral} we infer
\begin{align*}
\int_{G/M\langle\tau\rangle}f(gM\langle\tau\rangle)&=\int_G\tilde{f}(\mathrm{pr}_{KB_{\omega_+}}(g))\chi(\mathrm{pr}_{\langle\tau\rangle}(g))\intd g\\
&=\int_K\int_{B_{\omega_+}}\sum_{j\in\mathbb{Z}}\tilde{f}(ku)\chi(\tau^j)\intd u\intd k\\
&=\int_K\int_{B_{\omega_+}}\tilde{f}(ku)\intd u\intd k\\
&=\int_K\int_{B_{\omega_+}}f(kuM\langle\tau\rangle)\intd u\intd k.\qedhere
\end{align*}
\end{proof}

We can also use the Iwasawa decomposition to analyze the $G$-action on the boundary $G/\mathrm{Stab}_G(\omega_+)\cong K/K\cap \mathrm{Stab}_G(\omega_+)=K/K\cap B_{\omega_+}\cong\Omega$ of the tree.

\begin{lemma}\label{la:integral_K_shift}
Let $x\in G$. Then, for each $F\in C(K/K\cap B_{\omega_+})$,
\begin{gather*}
    \int_{K/K\cap B_{\omega_+}}F(k(x^{-1}k))\intd k=\int_{K/K\cap B_{\omega_+}}F(k)q^{-H(xk)}\intd k.
\end{gather*}
\end{lemma}

\begin{proof}
We first pick some left $(B_{\omega_+}\cap K)$-invariant function $F_1\in C_c(B_{\omega_+})$ and $\chi\in C_c(\langle\tau\rangle)$ which both integrate to $1$. Then, writing $F(k)$ for $F(k(K\cap B_{\omega_+}))$,
\begin{gather*}
    f(ku\tau^j)\coloneqq F(k)F_1(u)\chi(\tau^j)
\end{gather*}
defines a compactly supported function on $G$. Moreover, Proposition \ref{prop:KNA_integral} implies
\begin{align*}
\int_{K}F(k)\intd k&=\int_K\sum_{j\in\mathbb{Z}}F(k)\chi(\tau^j)\intd k=\int_K\int_{B_{\omega_+}}\sum_{j\in\mathbb{Z}}F(k)\chi(\tau^j)F_1(u)\intd u\intd k\\
&=\int_Gf(g)\intd g=\int_Gf(xg)\intd g.
\end{align*}
Now we write $g=ku\tau^j$ and
\begin{gather*}
    xg=xku\tau^j=k(xk)u_+(xk)\tau^{H(xk)}u\tau^j=k(xk)u_+(xk)\tau^{H(xk)}u\tau^{-H(xk)}\tau^{H(xk)+j}
\end{gather*}
with $\tau^{H(xk)}u\tau^{-H(xk)}\in B_{\omega_+}$. Thus, Proposition \ref{prop:KNA_integral} implies
\begin{align*}
\int_Gf(xg)\intd g&=\int_K\int_{B_{\omega_+}}\sum_{j\in\mathbb{Z}}f(xku\tau^j)\intd u\intd k\\
&=\int_K\int_{B_{\omega_+}}\sum_{j\in\mathbb{Z}}F(k(xk))F_1(u_+(xk)\tau^{H(xk)}u\tau^{-H(xk)})\chi(\tau^{H(xk)+j})\intd u\intd k\\
&=\int_K\int_{B_{\omega_+}}F(k(xk))F_1(u_+(xk)\tau^{H(xk)}u\tau^{-H(xk)})\intd u\intd k.
\end{align*}
Using now \cite[La.\@ 3.8]{Ve02} we find
\begin{align*}
\int_Gf(xg)\intd g&=\int_K\int_{B_{\omega_+}}F(k(xk))F_1(u_+(xk)u)q^{-H(xk)}\intd u\intd k\\
&=\int_KF(k(xk))q^{-H(xk)}\intd k\int_{B_{\omega_+}}F_1(u)\intd u\\
&=\int_KF(k(xk))q^{-H(xk)}\intd k.
\end{align*}
Altogether we arrive at
\begin{gather}\label{eq:pf_int_K}
    \int_KF(k(xk))q^{-H(xk)}\intd k=\int_KF(k)\intd k.
\end{gather}
Note that $k(x^{-1}k(xk))=k(kk^{-1}x^{-1}k(xk))=k$ since $(xk)^{-1}k(xk)\in B_{\omega_+}\langle\tau\rangle$. Thus, applying \eqref{eq:pf_int_K} to $\tilde{F}(k)\coloneqq F(k(x^{-1}k))$ yields
\begin{gather*}
    \int_KF(k)q^{-H(xk)}\intd k=\int_K\tilde{F}(k)\intd k=\int_KF(k(x^{-1}k))\intd k.\qedhere
\end{gather*}
\end{proof}

We can also decompose $G$ using both of the subgroups $B_{\omega_\pm}$. In this case we however get a disjoint union of two sets.

\begin{proposition}[Bruhat ``decomposition‘‘]\label{prop:Bruhat}
   We can write
\begin{gather*}
    G=B_{\omega_\pm}B_{\omega_\mp}\langle\tau\rangle\sqcup s \mathrm{Stab}_G(\omega_\pm).
\end{gather*}
Moreover, $u_1\tilde u_1\tau^{j_1}=u_2\tilde u_2\tau^{j_2}$ implies
\begin{gather*}
 j_1=j_2\qquad\text{ and }\qquad\exists\, m\in M\colon\quad u_1=u_2m,\ m\tilde u_1=\tilde u_2.
\end{gather*}
\end{proposition}

\begin{proof}
Note first that $B_{\omega_+} \mathrm{Stab}_G(\omega_-)=\{g\in G\mid g\omega_-\neq\omega_+\}$: For $g\in G$ with $g\omega_-\neq\omega_+$ there exists some $u\in B_{\omega_+}$ such that $ug\omega_-=\omega_-$ and thus $ug\in \mathrm{Stab}_G(\omega_-)$. On the other hand $B_{\omega_+}\mathrm{Stab}_G(\omega_-)\omega_-=B_{\omega_+}\omega_-=\Omega\setminus\{\omega_+\}$. Moreover, we have $\mathrm{Stab}_G(\omega_-)=B_{\omega_-}\langle\tau\rangle$. Indeed, let $g\in\mathrm{Stab}_G(\omega_-)$ and write $]\omega_-,\omega_+[=(\ldots,x_{-1},x_0,x_1,\ldots)$. Then there exists some $i\in\mathbb{Z}$ such that $gx_{-j}=x_{-j+i}$ for large enough $j$. Thus, $g\tau^{-i}$ is contained in $B_{\omega_-}$. In particular we obtain $B_{\omega_-}B_{\omega_+}\langle\tau\rangle=\{g\in G\mid g\omega_-\neq\omega_+\}$ and thus $G\setminus B_{\omega_-}B_{\omega_+}\langle\tau\rangle=s \mathrm{Stab}_G(\omega_-)$.

Let now $u_1\tilde u_1\tau^{j_1}=u_2\tilde u_2\tau^{j_2}\in B_{\omega_\pm}B_{\omega_\mp}\langle\tau\rangle$. Then $B_{\omega_\pm}\ni u_2^{-1}u_1=\tilde u_2\tau^{j_2-j_1}\tilde u_1^{-1}$ where the left hand side fixes $\omega_{\pm}$ and some $x\in\mathfrak{X}$, whereas the right hand side fixes $\omega_{\mp}$. Thus, $u_2^{-1}u_1$ fixes $]\omega_-,\omega_+[$ pointwise and there exists $m\in M$ such that $u_1=u_2m$. We thus infer $m\tilde u_1\tau^{j_1}=\tilde u_2\tau^{j_2}$. But $\tilde u_2^{-1}m\tilde u_1=\tau^{j_2-j_1}\in B_{\omega_\pm}\cap\langle\tau\rangle=\{\mathrm{id}\}$ implies $j_1=j_2$ and $m\tilde u_1=\tilde u_2$.
\end{proof}

The Bruhat decomposition from Proposition \ref{prop:Bruhat} also yields integration formulas. We start by describing the Haar measure on $B_{\omega_-}\langle\tau\rangle$.

\begin{lemma}\label{la:measure_NA}
The Haar measure on $B_{\omega_-}\langle\tau\rangle$ can be normalized such that
\begin{gather*}
    \forall f\in C_c(B_{\omega_-}\langle\tau\rangle)\colon\quad\int_{B_{\omega_-}\langle\tau\rangle} f(u\tau^j)\intd(u\tau^j)=\int_{B_{\omega_-}}\sum_{j\in\mathbb{Z}}f(u\tau^j)q^{-j}\intd u.
\end{gather*}
\end{lemma}

\begin{proof}
Let $f\in C_c(B_{\omega_-}\langle\tau\rangle),\ u_1\in B_{\omega_-}$ and $\tau^i\in\langle\tau\rangle$. Then we have
\begin{align*}
\int_{B_{\omega_-}}\sum_{j\in\mathbb{Z}} f(u_1\tau^\ell u\tau^j)q^{-j}\intd u&=\int_{B_{\omega_-}}\sum_{j\in\mathbb{Z}}f(u_1\tau^\ell u\tau^{-\ell}\tau^{j+\ell})q^{-j}\intd u\\
&=\int_{B_{\omega_-}}\sum_{j\in\mathbb{Z}}f(u_1\tau^\ell u\tau^{-\ell}\tau^{j})q^{\ell-j}\intd u\\
&=\int_{B_{\omega_-}}\sum_{j\in\mathbb{Z}}f(u_1u\tau^{j})q^{-j}\intd u\\
&=\int_{B_{\omega_-}}\sum_{j\in\mathbb{Z}}f(u\tau^{j})q^{-j}\intd u,
\end{align*}
where we used \cite[La.\@ 3.8]{Ve02} in the penultimate and the right-invariance of $\intd u$ in the last step. Thus, we obtain a left-invariant Haar measure on $B_{\omega_-}\langle\tau\rangle$ and the lemma follows from the uniqueness of Haar measures.
\end{proof}

\begin{lemma}\label{la:modular_fct_NA}
The modular function of $B_{\omega_-}\langle\tau\rangle$ is given by $\Delta_{B_{\omega_-}\langle\tau\rangle}(\bar{u}\tau^j)=q^{j}$ so that $\intd \bar{u}\intd\tau^j$, where $\intd\tau^j$ denotes the counting measure, defines a right-invariant Haar measure.
\end{lemma}

\begin{proof}
Let $u_1\in B_{\omega_-}$ and $\ell\in\mathbb{Z}$. Then Lemma \ref{la:measure_NA} implies
\begin{align*}
\int_{B_{\omega_-}\langle\tau\rangle} f(u\tau^ju_1\tau^\ell)\intd(u\tau^j)&=\int_{B_{\omega_-}}\sum_{j\in\mathbb{Z}}f(u\tau^ju_1\tau^\ell)q^{-j}\intd u\\
&=\int_{B_{\omega_-}}\sum_{j\in\mathbb{Z}}f(u\tau^ju_1\tau^{-j}\tau^{j+\ell})q^{-j}\intd u\\
&=\int_{B_{\omega_-}}\sum_{j\in\mathbb{Z}}f(u\tau^{j+\ell})q^{-j}\intd u\\
&=q^\ell\int_{B_{\omega_-}}\sum_{j\in\mathbb{Z}}f(u\tau^{j})q^{-j}\intd u.\qedhere
\end{align*}
\end{proof}

Recall that the Bruhat decomposition is built up by two disjoint sets or \emph{cells}. However, one of these cells carries the full measure while the other one has measure zero:
\begin{lemma}\label{la:Bruhat_measure_zero}
   The Bruhat cell $G\setminus B_{\omega_-}B_{\omega_+}\langle\tau\rangle=s \mathrm{Stab}_G(\omega_-)$ has $\intd g$-measure zero.
\end{lemma}

\begin{proof}
By the left invariance of the Haar measure, $s\mathrm{Stab}_G(\omega_-)$ has the same measure as $\mathrm{Stab}_G(\omega_-)$. If $\mathrm{Stab}_G(\omega_-)$ had positive measure, \cite[p.\@ 50]{W42} would imply that it contains an interior point $g_0$. However, (see \cite[I.4]{FTN91}) a basis of open neighborhoods of $g_0$ in $G$ is given by the sets
\begin{gather*}
    U_F(g_0)\coloneqq\{h\in G\mid\forall x\in F\colon g_0x=hx\},
\end{gather*}
where $F$ denotes a finite subset of $\mathfrak{X}$. By definition, if $\omega_-=[(x_0,x_1,\ldots)]$, there exists some $j$ such that $g_0x_{i}=x_{i+j}$ for large $i\geq n_0$. Now let $F\subseteq\mathfrak{X}$ be finite and denote by $\vec{e}$ an edge with $\omega_-\in\partial_+\vec{e}$ and such that $\eop{e}$ points towards each $x\in F$. Pick a geodesic ray from $\tau(\vec{e})$ to some $\omega\neq\omega_-$ and an automorphism $y$ that interchanges this ray with the ray from $\tau(\vec{e})$ to $\omega_-$ and leaves all points that can be reached from $\eop{e}$ fixed. Then $g_0y$ is in $U_F(g_0)$ but not in $\mathrm{Stab}_G(\omega_-)$ so that $g_0$ was not an interior point.
\end{proof}

Now we have all tools at hand to prove the integral formula for the Bruhat decomposition.

\begin{proposition}\label{prop:Bruhat_integral}
    The Haar measure on $G$ can be normalized such that
\begin{gather*}
    \forall f\in C_c(G)\colon\quad \int_Gf(g)\intd g=\int_{B_{\omega_+}}\int_{B_{\omega_-}}\sum_{j\in\mathbb{Z}}f(u\bar{u}\tau^j)\intd\bar{u}\intd u.
\end{gather*}
\end{proposition}

\begin{proof}
Since $G\setminus B_{\omega_-}B_{\omega_+}\langle\tau\rangle=s \mathrm{Stab}_G(\omega_-)$ has measure zero by Lemma \ref{la:Bruhat_measure_zero} it is enough to integrate over $B_{\omega_+}B_{\omega_-}\langle\tau\rangle$ on the left hand side. Now set, for $f\in C_c(B_{\omega_+}\times B_{\omega_-}\langle\tau\rangle)$,
\begin{gather*}
    \int_{B_{\omega_+}B_{\omega_-}\langle\tau\rangle}f\intd\tilde\mu\coloneqq\int_{B_{\omega_+}\langle\tau\rangle B_{\omega_-}}\int_{M}f(n(g)m,m^{-1}a(g)\bar{n}(g))\intd m\intd g,
\end{gather*}
where $g=n(g)a(g)\bar{n}(g)\in B_{\omega_+}\langle\tau\rangle B_{\omega_-}$ denotes a fixed decomposition of $g\in B_{\omega_+}\langle\tau\rangle B_{\omega_-}$. Note that this integral does not depend on the chosen decomposition: If $n_1a_1\bar{n}_1=n_2a_2\bar{n}_2$ we infer that $n_2^{-1}n_1$ is contained in $B_{\omega_+}$ and $\langle\tau\rangle B_{\omega_-}$. Then $n_2^{-1}n_1$ fixes $\omega_+,\ \omega_-$ and some $x\in\widetilde{\mathfrak{X}}$. Therefore, it fixes the geodesic rays $]\omega_-,x]$ and $[x,\omega_+[$ pointwise and thus $n_2^{-1}n_1|_{]\omega_-,\omega_+[}=\mathrm{id}$ so that $n_2^{-1}n_1\eqqcolon m_0\in M$. In particular, $(n_1,a_1\bar{n}_1)=(n_2m_0,m_0^{-1}a_2\bar{n}_2)$ so that
\begin{gather*}
    \int_Mf(n_1m,m^{-1}a_1\bar{n}_1)\intd m=\int_Mf(n_2m_0m,m^{-1}m_0^{-1}a_2\bar{n}_2)\intd m=\int_Mf(n_2m,m^{-1}a_2\bar{n}_2)\intd m.
\end{gather*}
Let $E\subseteq B_{\omega_+}$ and $F\subseteq \langle\tau\rangle B_{\omega_-}$ be measurable. Then, for $\bar{u}\in B_{\omega_-}$,
\begin{align*}
    \int_{B_{\omega_+}B_{\omega_-}\langle\tau\rangle}\mathbbm{1}_{E\times F\bar{u}^{-1}}\intd\tilde\mu&=\int_{B_{\omega_+}\langle\tau\rangle B_{\omega_-}}\int_{M}\mathbbm{1}_{E}(n(g)m)\mathbbm{1}_{F}(m^{-1}a(g)\bar{n}(g)\bar{u})\intd m\intd g\\
&=\int_{B_{\omega_+}\langle\tau\rangle B_{\omega_-}}\int_{M}\mathbbm{1}_{E}(n(g\bar{u})m)\mathbbm{1}_{F}(m^{-1}a(g\bar{u})\bar{n}(g\bar{u}))\intd m\intd g\\
&=\int_{B_{\omega_+}\langle\tau\rangle B_{\omega_-}}\int_{M}\mathbbm{1}_{E}(n(g)m)\mathbbm{1}_{F}(m^{-1}a(g)\bar{n}(g))\intd m\intd g\\
&=\int_{B_{\omega_+}B_{\omega_-}\langle\tau\rangle}\mathbbm{1}_{E\times F}\intd\tilde\mu,
\end{align*}
where we used that $n(g\bar{u})=n(g),\, a(g\bar{u})=a(g)$ and $\bar{n}(g\bar{u})=\bar{n}(g)\bar{u}$ is a possible decomposition of $g\bar{u}$ in the second step and the right invariance of $\intd g$ in the third step. For each $a\in\langle\tau\rangle$ we may similarly chose $n(ga)=n(g),\, a(ga)=a(g)a$ and $\bar{n}(ga)=a^{-1}\bar{n}(g)a$ so that $a(g)\bar{n}(g)a=a(ga)\bar{n}(ga)$ and we infer that $\tilde\mu(E\times Fa)=\tilde\mu(E\times F)$. Thus, for fixed $E$, $\tilde\mu(E\times\bigcdot)$ is a right $B_{\omega_-}\langle\tau\rangle$-invariant measure so that Lemma \ref{la:modular_fct_NA} implies
\begin{gather*}
    \exists\ c(E)>0\colon\qquad\tilde\mu(E\times F)=c(E)\cdot\intd\bar{u}\intd\tau^j(F).
\end{gather*}
On the other hand, we have $c(uE)\cdot\intd\bar{u}\intd\tau^j(F)=\tilde\mu(uE\times F)=\tilde\mu(E\times F)=c(E)\cdot\intd\bar{u}\intd\tau^j(F)$ for each $u\in B_{\omega_+}$ so that $c(\bigcdot)$ defines a left-invariant measure on $B_{\omega_+}$. Thus, there exists some constant $c$ with $c(E)=c\cdot\intd u(E)$. Altogether we obtain $\tilde\mu(E\times F)=c\cdot\intd u(E)\intd\bar{u}\intd\tau^j(F)$ and since $B_{\omega_\pm}=\bigcup_{x\in\mathfrak{X}}B_{\omega_\pm}\cap \mathrm{Stab}_G(x)$ are $\sigma$-finite by \cite[La.\@ 3.3]{Ve02} this uniquely extends to a measure on the product space $B_{\omega_+}\times B_{\omega_-}\langle\tau\rangle$ so that $\tilde\mu=c\cdot\intd u\intd\bar{u}\intd\tau^j$. Since $B_{\omega_+}\langle\tau\rangle B_{\omega_-}\cong B_{\omega_+}\times_M \langle\tau\rangle B_{\omega_-}$, where $m.(u,\tau^j\bar{u})\coloneqq (um,m^{-1}\tau^j\bar{u})$, functions on $B_{\omega_+}\langle\tau\rangle B_{\omega_-}$ correspond to $M$-invariant functions on $B_{\omega_+}\times B_{\omega_-}\langle\tau\rangle$. However, for such functions we have
\begin{gather*}
    \int_{B_{\omega_+}B_{\omega_-}\langle\tau\rangle}f\intd\tilde\mu=\mathrm{vol}(M)\int_{B_{\omega_+}\langle\tau\rangle B_{\omega_-}}f(n(g),a(g)\bar{n}(g))\intd g=\mathrm{vol}(M)\int_{B_{\omega_+}\langle\tau\rangle B_{\omega_-}}f(g)\intd g
\end{gather*}
which finishes the proof.
\end{proof}

Similar to the case of the Iwasawa decomposition we can also express the measure on $G/M\langle\tau\rangle$ in terms of the Bruhat decomposition.

\begin{lemma}\label{la:int_GMNN_-}
There exists a constant $c>0$ such that for each $f\in C_c(G/M\langle\tau\rangle)$
\begin{gather*}
\int_{G/M\langle\tau\rangle}f(gM\langle\tau\rangle)\intd(gM\langle\tau\rangle)=c\int_{B_{\omega_+}}\int_{B_{\omega_-}}f(u\bar{u}M\langle\tau\rangle)\intd\bar{u}\intd u.
\end{gather*}
\end{lemma}

\begin{proof}
Since $B_{\omega_-}B_{\omega_+}\langle\tau\rangle$ is dense in $G$ with complement of measure zero, it suffices to prove the lemma for $f\in C_c(B_{\omega_-}B_{\omega_+}\langle\tau\rangle/M\langle\tau\rangle)$. For such $f$ we define $\tilde{f}\in C_c(B_{\omega_-}B_{\omega_+})^M$ by $\tilde{f}(\bar{u}u)\coloneqq f(\bar{u}uM\langle\tau\rangle)$. Moreover, for some $\chi\in C_c(\langle\tau\rangle)$ we set $F(\bar{u}u\tau^j)\coloneqq \tilde{f}(\bar{u}u)\chi(\tau^j)$ on $B_{\omega_-}B_{\omega_+}\langle\tau\rangle$ and $F(g)\coloneqq 0$ on $G\setminus B_{\omega_-}B_{\omega_+}\langle\tau\rangle$. Note that $F$ is well-defined: If $\bar{u}u\tau^j=\bar{u}_1u_1\tau^{j_1}$,
\begin{gather*}
   \bar{u}^{-1}_1\bar{u}=u_1\tau^{j_1-j}u^{-1}\in B_{\omega_-}\cap \mathrm{Stab}_G(\omega_+)\subseteq M
\end{gather*}
implies the existence of some $m\in M$ such that $\bar{u}=\bar{u}_1m$. Thus, $mu\tau^j=u_1\tau^{j_1}$ so that Corollary \ref{cor:KNA} implies $j=j_1$ and in particular $\bar{u}_1u_1=\bar{u}m^{-1}mu=\bar{u}u$. Finally we obtain
\begingroup
\allowdisplaybreaks
\begin{align*}
\int_{G/M\langle\tau\rangle}f(gM\langle\tau\rangle)\intd(gM\langle\tau\rangle)&=
\int_{B_{\omega_-}B_{\omega_+}\langle\tau\rangle/M\langle\tau\rangle}\tilde{f}(\bar{u}u)\intd(\bar{u}uM\langle\tau\rangle)\\
&=\int_{B_{\omega_-}B_{\omega_+}\langle\tau\rangle/M\langle\tau\rangle}\int_M\sum_{j\in\mathbb{Z}}\tilde{f}(\bar{u}u)\chi(\tau^j)\intd(\bar{u}uM\langle\tau\rangle)\\
&=\int_{B_{\omega_-}B_{\omega_+}\langle\tau\rangle/M\langle\tau\rangle}\int_M\sum_{j\in\mathbb{Z}}F(\bar{u}um\tau^j)\intd(\bar{u}uM\langle\tau\rangle)\\
&=\int_{G/M\langle\tau\rangle}\int_M\sum_{j\in\mathbb{Z}}F(gm\tau^j)\intd(gM\langle\tau\rangle)\\
&=\int_{G}F(g)\intd g\\
&=c\int_{B_{\omega_+}}\int_{B_{\omega_-}}\sum_{j\in\mathbb{Z}}F(u\bar{u}\tau^j)\intd\bar{u}\intd u\\
&=c\int_{B_{\omega_+}}\int_{B_{\omega_-}}f(\bar{u}uM\langle\tau\rangle),
\end{align*}
\endgroup
where $c$ denotes the constant from Proposition \ref{prop:Bruhat_integral}.
\end{proof}

We can also use the Bruhat decomposition to describe the measure on the boundary. For this we first claim that
\begin{gather*}
    B_{\omega_-}/M\to K/K\cap B_{\omega_+},\qquad \bar{u}M \mapsto k(\bar{u})(K\cap B_{\omega_+})
\end{gather*}
defines an isomorphism onto an open subset of $K/K\cap B_{\omega_+}$ whose complement has measure zero.
For the injectivity, let $\bar{u}_1=ku_1\tau^{j_1}$ and $\bar{u}_2=ku_2\tau^{j_2}$ be two elements of $B_{\omega_-}$ with $k(\bar{u}_1)=k(\bar{u}_2)$. Then
\begin{gather*}
    B_{\omega_-}\ni\bar{u}_1^{-1}\bar{u}_2=u_1^{-1}\tau^{j_2-j_1}u_2\in B_{\omega_+}\langle\tau\rangle
\end{gather*}
is contained in $M$. For the surjectivity note first that $K/K\cap B_{\omega_+}=K/K\cap \mathrm{Stab}_G(\omega_+)\cong\Omega$ by $k(K\cap B_{\omega_+}) \mapsto k\omega_+$. Thus, we have to show that
\begin{gather*}
    B_{\omega_-}/M\to\Omega,\qquad\bar{u}M \mapsto k(\bar{u})\omega_+
\end{gather*}
has open image of full measure. However, since $k(\bar{u})\omega_+=\bar{u}\omega_+$, this follows from the transitive action of $B_{\omega_-}$ on $\Omega\setminus\{\omega_-\}$. We now relate the corresponding measures.

\begin{proposition}\label{prop:int_KM_N}
    Normalizing the measures on $K/K\cap B_{\omega_+}$ and $B_{\omega_-}/M$ by
\begin{gather*}
    \int_{K/K\cap B_{\omega_+}}\intd k=\int_{B_{\omega_-}/M}q^{-H(\bar{u})}\intd\bar{u}=1
\end{gather*}
we have
\begin{gather*}
    \int_{K/K\cap B_{\omega_+}}F(k(K\cap B_{\omega_+}))\intd k=\int_{B_{\omega_-}/M}F(k(\bar{u}))q^{-H(\bar{u})}\intd\bar{u}
\end{gather*}
for each $F\in C(K/K\cap B_{\omega_+})$.
\end{proposition}

\begin{proof}
Let $f_1\in C(K/(K\cap B_{\omega_+})),\ f_2\in C_c((K\cap B_{\omega_+})\backslash B_{\omega_+})$ with $\int_{B_{\omega_+}}f_2((K\cap B_{\omega_+})u)\intd u=1$ and $\chi\in C_c(\langle\tau\rangle)$ with $\sum_{j\in\mathbb{Z}}\chi(\tau^j)=1$. Moreover, set $f(g)\coloneqq f_1(k(g))f_2(u_+(g))\chi(H(g))$ for each $g\in G$. Then, by Proposition \ref{prop:KNA_integral},
\begin{gather*}
    \int_Gf(g)\intd g=\int_Kf_1(k(K\cap B_{\omega_+}))\intd k.
\end{gather*}
On the other hand, Proposition \ref{prop:Bruhat_integral} implies (up to a constant)
\begin{gather*}
\int_Gf(g)\intd g=\int_{B_{\omega_-}}\int_{B_{\omega_+}}\sum_{j\in\mathbb{Z}}f(\bar{u}u\tau^j)\intd u\intd\bar{u}.
\end{gather*}
Now, if $\bar{u}=k_{\bar{u}}n_{\bar{u}}\tau^{H(\bar{u})}\in KB_{\omega_+}\langle\tau\rangle$, we have $\bar{u}u\tau^{j}=k_{\bar{u}}n_{\bar{u}}\tau^{H(\bar{u})}u\tau^{-H(\bar{u})}\tau^{j+H(\bar{u})}$ so that $k(\bar{u}u\tau^{j})=k(\bar{u}),\ u_+(\bar{u}u\tau^j)=u_+(\bar{u})\tau^{j_u}u\tau^{-j_u}$ and $H(\bar{u}u\tau^j)=j+H(\bar{u})$. Thus,
\begingroup
\allowdisplaybreaks
\begin{align*}
\int_Gf(g)\intd g&=\int_{B_{\omega_-}}\int_{B_{\omega_+}}\sum_{j\in\mathbb{Z}}f_1(k(\bar{u}))f_2(u_+(\bar{u})\tau^{H(\bar{u})}u\tau^{-H(\bar{u})})\chi(\tau^{j+H(\bar{u})})\intd u\intd\bar{u}\\
&=\int_{B_{\omega_-}}\int_{B_{\omega_+}}f_1(k(\bar{u}))f_2(u_+(\bar{u})\tau^{H(\bar{u})}u\tau^{-H(\bar{u})})\intd u\intd\bar{u}\\
&=\int_{B_{\omega_-}}\int_{B_{\omega_+}}f_1(k(\bar{u}))f_2(u_+(\bar{u})u)q^{-H(\bar{u})}\intd u\intd\bar{u}\\
&=\int_{B_{\omega_-}}\int_{B_{\omega_+}}f_1(k(\bar{u}))f_2((K\cap B_{\omega_+})u)q^{-H(\bar{u})}\intd u\intd\bar{u}\\
&=\int_{B_{\omega_-}}f_1(k(\bar{u}))q^{-H(\bar{u})}\intd\bar{u},
\end{align*}
where we used \cite[La.\@ 3.8]{Ve02} in the third step and the left-invariance of $\intd u$ in the fourth step. Finally we have $f_1(k(\bar{u}m))q^{-H(\bar{u}m)}=f_1(k(\bar{u}))q^{-H(\bar{u})}$ for each $m\in M\subseteq B_{\omega_+}$ and the normalization follows by evaluating both sides at $f_1\equiv 1$.
\endgroup
\end{proof}

\subsection{Computations of pushforward measures}\label{app:pushforwards}

In this section we compute the pushforwards of measures along the maps $\phi$ from Proposition~\ref{prop:G-coord1} and $\psi$ from Proposition \ref{prop:G-coord2}.

\begin{definition}\label{def:d}
   With the coordinates from Proposition \ref{prop:G-coord1} the function
\begin{gather*}
    d\colon (\Omega\times\Omega)\setminus \mathrm{diag}(\Omega)\to\R_{>0},\quad (g\omega_-,g\omega_+)\mapsto q^{H(g)+H(gs)}=q^{\langle go,k(g)\omega_+\rangle+\langle go,k(gs)\omega_+\rangle}
\end{gather*}
is well-defined.
\end{definition}
\begin{proof}
We first show that $d$ is independent of the representative in $G/M\langle\tau\rangle$. Let $g=ku\tau^j\in KB_{\omega_+}\langle\tau\rangle$, $m\in M$ and $\tau^i\in\langle\tau\rangle$. Then
\begin{gather*}
    H(gm\tau^i)=H(ku\tau^jm\tau^i)=H(ku\tau^jm\tau^{-j}\tau^{j+i})=j+i=H(g)+i
\end{gather*}
whereas (recall $\tau^is=s\tau^{-i}$)
\begin{align*}
    H(gm\tau^is)&=H(gs(s^{-1}ms)\tau^{-i})\\
&=H(k(gs)u_+(gs)\tau^{H(gs)}(s^{-1}ms)\tau^{-H(gs)}\tau^{H(gs)-i})=H(gs)-i,
\end{align*}
where we used the fact that $s^{-1}ms\in M\subseteq B_{\omega_+}$ since
\begin{gather*}
    (s^{-1}ms)(\tau^\ell o)=s^{-1}m\tau^{-\ell}o=s^{-1}\tau^{-\ell}o=\tau^{\ell}o
\end{gather*}
for each $\ell\in\mathbb{Z}$. Moreover,
\begin{gather*}
    \langle go,k(g)\omega_+\rangle=-H(g^{-1}k(g))=-H(\tau^{-H(g)}u_+(g)^{-1})=H(g)
\end{gather*}
and
\begin{align*}
    \langle go,k(gs)\omega_+\rangle&=-H(ss^{-1}g^{-1}k(gs))=-H(s\tau^{-H(gs)}u_+(gs)^{-1})\\
&=-H(s\tau^{-H(gs)}u_+(gs)^{-1}\tau^{H(gs)}\tau^{-H(gs)})=H(gs).\qedhere
\end{align*}
\end{proof}

\begin{proposition}\label{prop:pushforward_phi}
    Recall the map $\phi$ from Proposition~\ref{prop:G-coord1}. There exists a constant $c_{\phi}$ such that
\begin{gather*}
    \phi_*\intd(gM\langle\tau\rangle)=c_\phi d(k\omega_+,k'\omega_+)(\intd k(K\cap B_{\omega_+})\otimes \intd k'(K\cap B_{\omega_+}))|_{(\Omega\times\Omega)\setminus \mathrm{diag}(\Omega)}.
\end{gather*}
\end{proposition}

\begin{proof}
Let $f\in C_c((\Omega\times\Omega)\setminus \mathrm{diag}(\Omega))$. Then Lemma \ref{la:int_G_KN} implies
\begin{align*}
\int_{(\Omega\times\Omega)\setminus \mathrm{diag}(\Omega)}fd^{-1}\phi_*\intd(gM\langle\tau\rangle)&=\int_{G/M\langle\tau\rangle}f(g\omega_-,g\omega_+)d(g\omega_-,g\omega_+)^{-1}\intd(gM\langle\tau\rangle)\\
&=\int_K\int_{B_{\omega_+}/M}f(ku\omega_-,ku\omega_+)d(ku\omega_-,ku\omega_+)^{-1}\intd u\intd k\\
&=\int_K\int_{B_{\omega_+}/M}f(ku\omega_-,k\omega_+)q^{-H(kus)}\intd u\intd k
\end{align*}
since $H(ku)=0$. We now note that $\eta\colon G\to G,\ g \mapsto sgs^{-1}$ is an involutive automorphism of $G$ which maps $B_{\omega_-}$ onto $B_{\omega_+}$ and preserves $M$. Thus, $\eta_*(\intd \bar{u}M)$ defines a left-invariant Haar measure on $B_{\omega_+}/M$ so that there exists a constant $c$ with
\begin{align*}
\int_K\int_{B_{\omega_+}/M}f(ku\omega_-,k\omega_+)q^{-H(kus)}\intd u\intd k&=c\int_K\int_{B_{\omega_-}/M}f(ku\omega_-,k\omega_+)q^{-H(kus)}\eta_*(\intd \bar{u})\intd k\\
&=c\int_K\int_{B_{\omega_-}/M}f(ks\bar{u}s^{-1}\omega_-,k\omega_+)q^{-H(ks\bar{u})}\intd\bar{u}\intd k\\
&=c\int_K\int_{B_{\omega_-}/M}f(ksk(\bar{u})\omega_+,k\omega_+)q^{-H(ks\bar{u})}\intd\bar{u}\intd k\\
&=c\int_K\int_{B_{\omega_-}/M}f(ksk(\bar{u})\omega_+,k\omega_+)q^{-H(\bar{u})}\intd\bar{u}\intd k.
\end{align*}
Now Proposition \ref{prop:int_KM_N} implies that this equals (with a possibly different $c$)
\begin{gather*}
c\int_K\int_{K/K\cap B_{\omega_+}}f(ksk'\omega_+,k\omega_+)\intd k'\intd k=c\int_K\int_{K/K\cap B_{\omega_+}}f(k'\omega_+,k\omega_+)\intd k'\intd k.\qedhere
\end{gather*}
\end{proof}

\begin{lemma}\label{la:Jacobian_N}
Let
\begin{gather*}
    \nu\colon B_{\omega_-}/(B_{\omega_-}\cap K)\to B_{\omega_-}/(B_{\omega_-}\cap K),\qquad\bar{u}(B_{\omega_-}\cap K)\mapsto\bar{u}_{\langle\tau\rangle B_{\omega_-}K}(s\bar{u}s^{-1}),
\end{gather*}
where $\bar{u}_{\langle\tau\rangle B_{\omega_-}K}(g)=\bar{u}(B_{\omega_-}\cap K)$ if $g=\tau^j\bar{u}k\in\langle\tau\rangle B_{\omega_-}K$. Then
\begin{gather*}
    \int_{B_{\omega_-}/(B_{\omega_-}\cap K)}f(\bar{u})\intd \bar{u}=\int_{B_{\omega_-}/(B_{\omega_-}\cap K)}f(\nu(\bar{u}))\intd \bar{u}
\end{gather*}
for each integrable function $f$ on $B_{\omega_-}/(B_{\omega_-}\cap K)$.
\end{lemma}

\begin{proof}
Taking inverses in Corollary \ref{cor:KNA} we first obtain $G=\langle\tau\rangle B_{\omega_-}K$, where the $B_{\omega_-}$ (resp.\@ $K$) part is unique up to right (resp.\@ left) multiplication by $B_{\omega_-}\cap K$. In particular, $\bar{u}_{\langle\tau\rangle B_{\omega_-}K}(s\bar{u}s^{-1})=\bar{u}_{\langle\tau\rangle B_{\omega_-}K}(s\bar{u})$ is well-defined. Now let $f\in C_c(B_{\omega_-})$ be right $B_{\omega_-}\cap K$-invariant, $f_1\in C_c(K)$ be left $B_{\omega_-}\cap K$-invariant with integral one and $\chi\in C_c(\langle\tau\rangle)$ with sum one. By interchanging the role of $\omega_+$ and $\omega_-$, taking inverses and using the unimodularity of $K,B_{\omega_-}$ and $\mathbb{Z}$ Proposition \ref{prop:KNA_integral} implies for $F(\tau^j\bar{u}k)\coloneqq\chi(\tau^j)f(\bar{u})f_1(k)$
\begin{gather*}
    \int_GF(g)\intd g=\sum_{j\in\mathbb{Z}}\int_{B_{\omega_-}}\int_KF(\tau^j\bar{u}k)\intd k\intd\bar{u}=\int_{B_{\omega_-}}f(\bar{u})\intd\bar{u}.
\end{gather*}
Now we observe that for $g=\tau^j\bar{u}k$ and $s\bar{u}=\tau^{j_{s\bar{u}}}\bar{u}_{s\bar{u}}k_{s\bar{u}}$
\begin{gather*}
 sg=s\tau^j\bar{u}k=\tau^{-j}s\bar{u}k=\tau^{j_{s\bar{u}}-j}\bar{u}_{s\bar{u}}k_{s\bar{u}}k
\end{gather*}
gives rise to a decomposition of $sg$. Thus,
\begin{align*}
    \int_GF(g)\intd g&=\int_GF(sg)\intd g=\sum_{j\in\mathbb{Z}}\int_{B_{\omega_-}}\int_KF(s\tau^j\bar{u}k)\intd k\intd\bar{u}\\
&=\sum_{j\in\mathbb{Z}}\int_{B_{\omega_-}}\int_KF(\tau^{j_{s\bar{u}}-j}\bar{u}_{s\bar{u}}k_{s\bar{u}}k)\intd k\intd\bar{u}\\
&=\sum_{j\in\mathbb{Z}}\int_{B_{\omega_-}}\int_K\chi(\tau^{j})f(\bar{u}_{s\bar{u}})f_1(k)\intd k\intd\bar{u}\\
&=\int_{B_{\omega_-}}f(\bar{u}_{s\bar{u}})\intd\bar{u}.\qedhere
\end{align*}
\end{proof}

\begin{proposition}\label{prop:pushforward_psi}
    Recall the map $\psi$ from Proposition \ref{prop:G-coord2}. There exists a constant $c_{\psi}>0$ such that
\begin{gather*}
    (\psi^{-1})_*(\intd(gK)\otimes\intd(g'M\langle\tau\rangle))=c_{\psi}(\widetilde{\mathrm{pr}}_M)_*(\intd g\otimes\intd\bar{u}(B_{\omega_-}\cap K)),
\end{gather*}
where $\widetilde{\mathrm{pr}}_M\colon G\times B_{\omega_-}/(B_{\omega_-}\cap K)\to G\times_M  B_{\omega_-}/(B_{\omega_-}\cap K)$ denotes the canonical projection.
\end{proposition}

\begin{proof}
Note first that $\psi^{-1}(G/K\times B_{\omega_+}B_{\omega_-}\langle\tau\rangle/M\langle\tau\rangle)=B_{\omega_+}B_{\omega_-}\langle\tau\rangle\times_MB_{\omega_-}/(B_{\omega_-}\cap K)$ since
\begin{gather*}
    \psi(u\bar{u}\tau^j,\bar{u}'(B_{\omega_-}\cap K))=(u\bar{u}\tau^j\bar{u}'K,u\bar{u}M\langle\tau\rangle)
\end{gather*}
where $\tau^j\bar{u}'K$ runs through all points in $G/K\cong\mathfrak{X}=\bigcup_{j\in\mathbb{Z}}H_{\omega_-}(\tau^jo)$ (see Definition~\ref{def:horocycle}). For each function $h\in C_c(G/K\times B_{\omega_+}B_{\omega_-}\langle\tau\rangle/M\langle\tau\rangle)=C_c(G/K\times B_{\omega_+}B_{\omega_-}/M)$ Lemma~\ref{la:int_GMNN_-} and Proposition~\ref{prop:KNA_integral} (by taking inverses) imply
\begin{align}
\nonumber\int_{G/K\times B_{\omega_+}B_{\omega_-}\langle\tau\rangle/M\langle\tau\rangle}&h(gK,g'M\langle\tau\rangle)\intd gK\intd g'M\langle\tau\rangle\\
\nonumber&=c\int_{G/K}\int_{B_{\omega_+}}\int_{B_{\omega_-}}h(gK,u\bar{u}M\langle\tau\rangle)\intd\bar{u}\intd u\intd gK\\
&=c\sum_{j\in\mathbb{Z}}\int_{B_{\omega_+}}\int_{B_{\omega_+}}\int_{B_{\omega_-}}h(\tau^ju'K,u\bar{u}M\langle\tau\rangle)\intd\bar{u}\intd u\intd u'.\label{eq:pf_psi_const}
\end{align}
Now let $\tilde h\in C_c(B_{\omega_+}B_{\omega_-}\langle\tau\rangle\times B_{\omega_-}/(B_{\omega_-}\cap K))^M\cong C_c(B_{\omega_+}B_{\omega_-}\langle\tau\rangle\times_M B_{\omega_-}/(B_{\omega_-}\cap K))$. Using Equation \eqref{eq:pf_psi_const} we infer
\begingroup
\allowdisplaybreaks
\begin{align*}
\int_{B_{\omega_+}B_{\omega_-}\langle\tau\rangle\times_M B_{\omega_-}/(B_{\omega_-}\cap K)}&\tilde h(\psi^{-1})_*(\intd(gK)\otimes\intd(g'M\langle\tau\rangle))\\
&=\int_{G/K\times B_{\omega_+}B_{\omega_-}\langle\tau\rangle/M\langle\tau\rangle}\tilde h\circ\psi^{-1}\intd(gK)\intd(g'M\langle\tau\rangle)\\
&=c\sum_{j\in\mathbb{Z}}\int_{B_{\omega_+}}\int_{B_{\omega_+}}\int_{B_{\omega_-}}(\tilde{h}\circ\psi^{-1})(\tau^ju'K,u\bar{u}M\langle\tau\rangle)\intd\bar{u}\intd u\intd u'.
\end{align*}
We now use the explicit form of the inverse $\psi^{-1}$ from the proof of Proposition~\ref{prop:G-coord2} and rewrite it as
\begin{gather*}
\psi^{-1}\colon G/K \times  G/M\langle \tau\rangle  \to   G\times_M  B_{\omega_-}/(B_{\omega_-}\cap K) ,
\quad (hK,gM\langle \tau\rangle)  \mapsto [g\tau^j, \bar{u}(B_{\omega_-}\cap K)],
\end{gather*}
with $g^{-1}h=\tau^j\bar{u}k\in \langle\tau\rangle B_{\omega_-}K$ (in contrast to $(\tau^j\bar{u}\tau^{-j})\tau^jk\in B_{\omega_-}\langle\tau\rangle K$).
Then
\begin{align*}
\psi^{-1}(\tau^ju'K,u\bar{u}M\langle\tau\rangle)=(u\bar{u}\tau_{\langle\tau\rangle B_{\omega_-}K}((u\bar{u})^{-1}\tau^ju'),\bar{u}_{\langle\tau\rangle B_{\omega_-}K}((u\bar{u})^{-1}\tau^ju'))
\end{align*}
where $g=\tau_{\langle\tau\rangle B_{\omega_-}K}(g)\bar{u}_{\langle\tau\rangle B_{\omega_-}K}k\in\langle\tau\rangle B_{\omega_-}K$ denotes some decomposition of $g\in G$. Since $u^{-1}\tau^ju'=\tau^j\tau^{-j}u^{-1}\tau^ju'$ with $\tau^{-j}u^{-1}\tau^j\in B_{\omega_+}$ we can use the invariance of $\intd u'$ to restrict to
\begin{gather*}
    (u\bar{u}\tau_{\langle\tau\rangle B_{\omega_-}K}(\bar{u}^{-1}\tau^ju'),\bar{u}_{\langle\tau\rangle B_{\omega_-}K}(\bar{u}^{-1}\tau^ju')).
\end{gather*}
\endgroup
Now let $u'=\tau^{\tilde j}\tilde{\bar{u}}\bar{k}\in \langle\tau\rangle B_{\omega_-}K$ and observe that
\begin{gather*}
    \bar{u}^{-1}\tau^ju'=\tau^{j+\tilde{j}}(\tau^{-(j+\tilde{j})}\bar{u}^{-1}\tau^{j+\tilde{j}}\tilde{\bar{u}})\tilde{k}\in\langle\tau\rangle B_{\omega_-}K.
\end{gather*}
Altogether we obtain (where the constant may change from line to line)
\begin{align*}
&\hphantom{{}={}}\int_{B_{\omega_+}B_{\omega_-}\langle\tau\rangle\times_M B_{\omega_-}/(B_{\omega_-}\cap K)}\tilde h(\psi^{-1})_*(\intd(gK)\otimes\intd(g'M\langle\tau\rangle))\\
&=c\sum_{j\in\mathbb{Z}}\int_{B_{\omega_+}}\int_{B_{\omega_+}}\int_{B_{\omega_-}}\tilde{h}(u\bar{u}\tau^{j+\tilde{j}},\tau^{-(j+\tilde{j})}\bar{u}^{-1}\tau^{j+\tilde{j}}\tilde{\bar{u}})\intd\bar{u}\intd u\intd u'\\
&=c\sum_{j\in\mathbb{Z}}\int_{B_{\omega_+}}\int_{B_{\omega_+}}\int_{B_{\omega_-}}\tilde{h}(u\bar{u}\tau^{j},\tau^{-j}\bar{u}^{-1}\tau^{j}\tilde{\bar{u}})\intd\bar{u}\intd u\intd u'\\
&=c\sum_{j\in\mathbb{Z}}\int_{B_{\omega_-}}\int_{B_{\omega_+}}\int_{B_{\omega_-}}\tilde{h}(u\bar{u}\tau^{j},\tau^{-j}\bar{u}^{-1}\tau^{j}\bar{u}_{\langle\tau\rangle B_{\omega_-}K}(s\bar{u}'s^{-1}))\intd\bar{u}\intd u\intd \bar{u}'\\
&\stackrel{\mathclap{\ref{la:Jacobian_N}}}{=}c\sum_{j\in\mathbb{Z}}\int_{B_{\omega_-}}\int_{B_{\omega_+}}\int_{B_{\omega_-}}\tilde{h}(u\bar{u}\tau^{j},\tau^{-j}\bar{u}^{-1}\tau^{j}\bar{u}')\intd\bar{u}\intd u\intd \bar{u}'\\
&=c\sum_{j\in\mathbb{Z}}\int_{B_{\omega_-}}\int_{B_{\omega_+}}\int_{B_{\omega_-}}\tilde{h}(u\bar{u}\tau^{j},\bar{u}')\intd\bar{u}\intd u\intd \bar{u}'\\
&\stackrel{\mathclap{\ref{prop:Bruhat_integral}}}{=}c\int_G\int_{B_{\omega_-}/(B_{\omega_-}\cap K)}\tilde{h}(g,\bar{u}')\intd \bar{u}'\intd g,
\end{align*}
where we used that the homomorphism $B_{\omega_-}\to B_{\omega_+},\quad \bar{u} \mapsto s\bar{u}s^{-1}$ is measure-preserving. Thus, we end up with
\begin{align*}
    \int_{B_{\omega_+}B_{\omega_-}\langle\tau\rangle\times_M B_{\omega_-}/(B_{\omega_-}\cap K)}\tilde h(\psi^{-1})_*(\intd(gK)&\otimes\intd(g'M\langle\tau\rangle))\\
&=c\int_G\int_{B_{\omega_-}/(B_{\omega_-}\cap K)}\tilde{h}(g,\bar{u}')\intd \bar{u}'\intd g
\end{align*}
which finishes the proof since $G\setminus B_{\omega_+}B_{\omega_-}\langle\tau\rangle$ has measure zero.
\end{proof}

\bibliographystyle{amsalpha}
\bibliography{Literatur}

\end{document}